\documentclass[aop,preprint]{imsart}%
\usepackage{amssymb, mathrsfs}
\usepackage{bbm}
\usepackage{mathtools}
\usepackage{datetime}
\usepackage{dsfont}
\usepackage{verbatim}
\usepackage{enumerate}
\usepackage{MnSymbol}
\usepackage[titletoc, title]{appendix}
\usepackage{thmtools}
\usepackage{thm-restate}%
\usepackage{amsmath}%
\usepackage{amsfonts}%
\usepackage{amssymb}%
\usepackage{graphicx}

\providecommand{\U}[1]{\protect\rule{.1in}{.1in}}

\RequirePackage[OT1]{fontenc}
\RequirePackage{amsthm,amsmath}
\RequirePackage[numbers]{natbib}
\RequirePackage[colorlinks,citecolor=blue,urlcolor=blue]{hyperref}

\newcommand{\vertiii}[1]{{\left\vert\kern-0.25ex\left\vert\kern-0.25ex\left\vert #1
\right\vert\kern-0.25ex\right\vert\kern-0.25ex\right\vert}}
\newtheorem{theorem}{Theorem}[section]
\newtheorem*{theorem*}{Theorem}
\newtheorem{proposition}[theorem]{Proposition}
\newtheorem{lemma}[theorem]{Lemma}
\newtheorem{corollary}[theorem]{Corollary}

\newtheorem{definition}[theorem]{Definition}
\newtheorem{remark}[theorem]{Remark}
\newcommand{\wrt}[1]{\, \mathrm{d} #1}
\newcommand{\strato}[1]{\circ \mathrm{d} #1}

\newcommand{\exptn}[1]{\mathbb{E} \left[ #1\right]}
\newcommand{\fqquad}{\qquad \qquad \qquad \qquad}
\newcommand{\abs}[1]{\left\lvert#1\right\rvert}
\newcommand{\norm}[1]{\left\lVert#1\right\rVert}

\begin{document}
\begin{frontmatter}
\title{A Stratonovich-Skorohod integral formula for Gaussian rough paths}
\runtitle{Stratonovich-Skorohod formula for Gaussian RPs}
\begin{aug}
\author{\fnms{Thomas} \snm{Cass} \thanksref{t1} \ead[label=e1]{thomas.cass@imperial.ac.uk}}
\and
\author{\fnms{Nengli} \snm{Lim} \thanksref{t2} \ead[label=e2]{nengli\textunderscore lim@sutd.edu.sg}}
\thankstext{t1}{Supported by EPSRC grant EP/M00516X/1.}
\thankstext{t2}{Part of the work was carried out while the second author was a PhD candidate at ICL, supported by the Roth scholarship and the Imperial College-NUS joint PhD scholarship.}
\runauthor{T. Cass and N. Lim}
\affiliation{Imperial College London and Singapore University of Technology and Design}
\address{Department of Mathematics\\
Imperial College London\\
South Kensington Campus\\
180 Queen's Gate\\
London SW7 2AZ\\
United Kingdom\\
\printead{e1}}
\address{Engineering Systems \& Design \\
Singapore University of Technology \& Design \\
8 Somapah Road \\
Building 1, Room 1.502.02 \\
Singapore 487372 \\
\printead{e2}}
\end{aug}
\begin{abstract}
Given a Gaussian process $X$, its canonical geometric rough path lift $\mathbf{X}$, and a solution $Y$ to the rough differential equation (RDE) $\mathrm{d}Y_{t} = V\left (Y_{t}\right ) \circ \mathrm{d} \mathbf{X}_t$, we
present a closed-form correction formula for $\int Y \circ \mathrm{d} \mathbf{X} - \int Y \, \mathrm{d} X$, i.e. the difference between the rough and Skorohod integrals of $Y$ with respect to $X$. When $X$ is standard Brownian motion, we recover the classical Stratonovich-to-It{\^o} conversion formula, which we generalize to Gaussian rough paths with finite $p$-variation, $p < 3$, and satisfying an additional natural condition. This encompasses many familiar examples, including fractional Brownian motion with $H > \frac{1}{3}$. To prove the formula, we first show that the Riemann-sum approximants of the Skorohod integral converge in $L^2(\Omega)$ by using a novel characterization of the Cameron-Martin norm in terms of higher-dimensional Young-Stieltjes integrals. Next, we append the approximants of the Skorohod integral with a suitable compensation term without altering the limit, and the formula is finally obtained after a re-balancing of terms.
\end{abstract}
\end{frontmatter}

\section{Introduction}

Gaussian processes are used in modeling natural phenomena, from rough
stochastic volatility models in high-frequency trading \cite{bfg2015}, to
models of vortex filaments based on fractional Brownian motion \cite{nrt2003}.
To analyze stochastic processes with regularity lower than standard Brownian
motion, one can employ the theory of rough paths \cite{lcl2006}. In
particular, given a Gaussian process $X$, one can lift it canonically to a
geometric rough path $\mathbf{X}$ \cite{fv2010a}, and this allows one to study
the properties of rough differential equations (RDEs)
\begin{align}
\label{rdeMain}\mathrm{d}Y_{t} = V(Y_{t}) \circ\mathrm{d} \mathbf{X}_{t},\quad
Y_{0} = y_{0} \in\mathbb{R}^{e},
\end{align}
and of rough integrals of the form
\begin{align}
\label{int}\int_{0}^{T} Y_{t} \circ\mathrm{d} \mathbf{X}_{t}.
\end{align}
Furthermore, this geometric calculus generalizes Stratonovich's stochastic
calculus, and as such, it finds natural applications, e.g. in stochastic
geometry where the change-of-variable formula allows one to give an intrinsic
and coordinate-invariant definition of a rough path on a general smooth
manifold, cf. \cite{cll2012}, \cite{cdl2015}.

It\^{o} integrals, by contrast, preserve the local martingale property, which
is a useful feature when computing probabilistic quantities such as exit
distributions and conditional expectations. One can often gain insight into a
problem by transforming Stratonovich integrals to It\^{o} integrals and vice
versa, depending on the requirement at hand.

Now if $Y$ and $X$ are two continuous semi-martingales, both $\mathbb{R}^{d}%
$-valued, it is well-known that the difference between the two integrals is
given in terms of the quadratic covariation through the formula, cf.
\cite{ks98}, \cite{ry2005},
\begin{align*}
\int_{0}^{T} \left\langle Y_{t}, \circ\mathrm{d} X_{t} \right\rangle =
\int_{0}^{T} \left\langle Y_{t}, \mathrm{d} X_{t} \right\rangle + \frac{1}{2}
\left[  Y, X \right] _{T}.
\end{align*}
In the case where $Y_{t}$ solves RDE \eqref{rdeMain} and $X_{t}$ is taken to
be standard Brownian motion $B_{t}$, this becomes
\begin{align}
\label{usual formula}\int_{0}^{T} \left\langle Y_{t}, \circ\mathrm{d}B_{t}
\right\rangle = \int_{0}^{T} \left\langle Y_{t}, \mathrm{d}B_{t} \right\rangle
+\frac{1}{2} \int_{0}^{T} \mathrm{tr} \left[  V\left(  Y_{t} \right)  \right]
\, \mathrm{d} t,
\end{align}
where in the second term on the right side we have the usual trace of
$V(Y_{t}) \in\mathbb{R}^{d} \otimes\mathbb{R}^{d}$ considered as a $d$-by-$d$ matrix.

On the other hand, if $Y_{t}\equiv\nabla f(X_{t})$, where $f$ is sufficiently
smooth, then we get It\^{o}'s formula,
\begin{align}
\label{itoFormula}f(X_{T})-f(X_{0})  & =\int_{0}^{T} \left\langle \nabla
f(X_{t}), \circ\mathrm{d} \mathbf{X}_{t} \right\rangle \\
& = \int_{0}^{T} \left\langle \nabla f(X_{t}), \, \mathrm{d} X_{t}
\right\rangle + \frac{1}{2} \int_{0}^{T} \Delta f(X_{t}) \, \mathrm{d} R(t),
\end{align}
where the first term on the right side is the Skorohod integral of $\nabla
f(X)$ with respect to $X$, and $R(t)$ is the variance $\mathbb{E} \left[
\left( X_{t}^{(1)} \right) ^{2}\right] $. This has been well-studied for
general Gaussian processes, particularly fractional Brownian motion, over the
past two decades; see \cite{np1998}, \cite{amn2001}, \cite{ccm2003},
\cite{no2011}, and in particular \cite{hjt}, which uses rough path techniques
to prove the formula.

Our main result is the following theorem, where the driving signal
$\mathbf{X}$ is constructed from the limit of the piecewise-linear
approximations of a Gaussian process with i.i.d. components.

\begin{theorem*}
For $2 \leq p < 3$, let $Y \in\mathcal{C}^{p-var} \left(  [0, T];
\mathbb{R}^{d} \right) $ denote the path-level solution to
\begin{align*}
\mathrm{d} Y_{t} = V(Y_{t}) \circ\mathrm{d} \mathbf{X}_{t}, \quad Y_{0} =
y_{0},
\end{align*}
where $V \in\mathcal{C}^{6}_{b} \left(  \mathbb{R}^{d}; \mathbb{R}^{d}
\otimes\mathbb{R}^{d} \right) $ and $\mathbf{X} \in\mathcal{C}^{0, p-var}
\left(  [0, T]; G^{\lfloor p \rfloor} \left( \mathbb{R}^{d} \right) \right) $.
We assume the covariance function $R$ of $\mathbf{X}$ is of finite 2D $\rho
$-variation, $1 \leq\rho< \frac{3}{2}$, and satisfies
\begin{align}
\label{introCond}\left\|  R(t, \cdot) - R (s, \cdot) \right\| _{\rho-var; [0,
T]} \leq C \left|  t - s\right| ^{\frac{1}{\rho}},
\end{align}
for all $s, t \in[0, T]$. Then almost surely, we have
\begin{align}
\label{correctionFormula}%
\begin{split}
\int_{0}^{T} \left\langle Y_{t}, \circ\mathrm{d} \mathbf{X}_{t} \right\rangle
& = \int_{0}^{T} \left\langle Y_{t}, \mathrm{d} X_{t} \right\rangle + \frac
{1}{2} \int_{0}^{T} \mathrm{tr} \left[  V(Y_{t}) \right]  \, \mathrm{d} R(t)\\
& \qquad+ \int_{[0, T]^{2}} \mathds{1}_{[0, t)} (s) \mathrm{tr} \left[
J^{\mathbf{X}}_{t} \left( J^{\mathbf{X}}_{s}\right) ^{-1} V(Y_{s}) - V(Y_{t})
\right]  \, \mathrm{d} R(s, t).
\end{split}
\end{align}

\end{theorem*}

Here, $J^{\mathbf{X}}_{t}$ denotes the Jacobian of the flow map $y_{0}
\rightarrow Y_{t}$. The last term on the right side of
\eqref{correctionFormula} is a proper 2D Young-Stieltjes integral with respect
to the covariance function. When $X$ is standard Brownian motion, it vanishes
since the integrand is zero on the diagonal and $\mathrm{d} R(s,t) =
\delta_{\{s=t\}} \, \mathrm{d} s \, \mathrm{d} t$. This, together with the
fact that $R(t) = t$, allows us to recover the usual It\^{o}-Stratonovich
conversion formula \eqref{usual formula}.

Hence, an immediate contribution of the theorem is the generalization of
formula \eqref{usual formula} to the setting where the integrands are
solutions to Gaussian RDEs. Here, we are able to give a formula for
integrators other than standard Brownian motion without restriction on the
regularity of the integrand; compare this to \cite{d2000} e.g., where
essentially Young complementary regularity is required. Furthermore, the novel
2D Young-Stieltjes integral can be understood as measuring \textit{the failure
of the commutativity of $V$ with respect to the covariance of the Gaussian
process}. For studying the dynamics of Gaussian processes in cases where the
correction formula is indispensable, e.g. Gaussian processes evolving on
manifolds, this could lead to new insights.

We now provide the main idea behind the proof of the theorem. Denoting
$\mathbf{X} = (1, X, \mathbf{X}^{2})$, the solution $Y$ to RDE \eqref{rdeMain}
can be integrated against this rough path and
\begin{align}
\label{rpApprox}\int_{0}^{T} \left\langle Y_{t}, \circ\mathrm{d}\mathbf{X}_{t}
\right\rangle =\lim_{\left\Vert \pi\right\Vert \rightarrow0}\sum_{i}
\left\langle Y_{t_{i}}, X_{t_{i},t_{i+1}} \right\rangle + V(Y_{t_{i}})\left(
\mathbf{X}_{t_{i},t_{i+1}}^{2}\right)
\end{align}
almost surely. Continuing, we devote Section 4 to proving two claims. The
first is that $Y$ lies in the domain of the Skorohod integral operator w.r.t.
$X$, and the second is that, as a limit in $L^{2}(\Omega)$ we have
\begin{align}
\label{skoro}%
\begin{split}
& \int_{0}^{T} \left\langle Y_{t}, \, \mathrm{d} X_{t} \right\rangle \\
& \qquad= \lim_{\left\|  \pi\right\|  \rightarrow0} \sum_{i} \left[
\left\langle Y_{t_{i}}, X_{t_{i}, t_{i+1}} \right\rangle - \int_{0}^{t_{i}}
\mathrm{tr} \, \left[  J^{\mathbf{X}}_{t_{i}} \left(  J^{\mathbf{X}}_{s}
\right) ^{-1} V(Y_{s}) \right]  \, R(\Delta_{i}, \, \mathrm{d} s) \right] .
\end{split}
\end{align}
Proving these facts is less obvious than in the case where $Y_{t}=\nabla
f(X_{t})$, and we need to use the tail estimates of \cite{cll2013} in a
fundamental way. In Section 5, we use condition \eqref{introCond} to prove
that
\begin{align}
\label{vanishing terms}\sum_{i}V(Y_{t_{i}})\left(  \left(  \mathbf{X}%
_{t_{i},t_{i+1}}^{2}\right)  - \frac{1}{2}\mathbb{E}\left[  \left(
X_{t_{i},t_{i+1}}^{(1)}\right)  ^{2}\right]  \mathcal{I}_{d} \right)
\end{align}
has a vanishing $L^{2}(\Omega)$ limit as $\left\Vert \pi\right\Vert
\rightarrow0$. The proof of this relies on estimates coming from a delicate
interplay between the theory of Malliavin calculus and Gaussian rough paths;
see Proposition \ref{2ndlevel}. After augmenting \eqref{vanishing terms} to
\eqref{skoro} and extracting an almost sure subsequence, we can take the
difference between this subsequence and \eqref{rpApprox}. A careful
rearrangement of the terms in this difference will then yield the correction term.

We now outline the structure of the paper, as well as highlight other
contributions that are of independent interest. We begin in Section 2 with a
summary of Gaussian rough path concepts and a primer on the Malliavin calculus
as applied to RDE solutions. In Section \ref{HOMD}, we provide a general
closed-form expression and a novel bound for the higher-order Malliavin
derivatives of RDE solutions relative to the driving rough path (cf.
\cite{hp2013}, \cite{inahama2014}, \cite{chlt2015}). This will be then used in
Section 5 to show that \eqref{vanishing terms} has vanishing $L^{2}(\Omega)$ limit.

In Section 4, we give a new characterization of the Cameron-Martin norm in
terms of multi-dimensional Young-Stieltjes integrals. We show that one can
identify $\mathcal{C}^{p-var}([0,T])$ with a dense subspace of $\mathcal{H}%
_{1}$, the Hilbert space generated by the indicator functions which is
isomorphic to the Cameron-Martin space. In particular, for $f \in
\mathcal{C}^{p-var}([0,T])$, we have
\begin{align}
\label{norm}\left\Vert f\right\Vert _{\mathcal{H}_{1}}=\sqrt{\int_{[0,T]^{2}%
}f_{t}f_{s} \, \mathrm{d}R(s,t)}.
\end{align}
We also give a corresponding characterization with regards to the tensor
product of the Cameron-Martin space in Section \ref{tensorProd}, and revisit
the classical Ito-Skorohod isometry in Section \ref{itoSkorohodIsometry} by
giving it a new formulation in terms of multi-dimensional Young-Stieltjes
integrals. Finally, Section 4 is concluded with a proof of \eqref{skoro}.

The main theorem and its proof is given in Section 6.

\section{Preliminaries}

The following is a summary of basic notation that will be used throughout the paper.

We will use $\left\{  e_{j} \right\} $, $j = 1, \ldots, d$, to denote the
canonical basis for $\mathbb{R}^{d}$, and $\left|  \cdot\right| $ will denote
the standard Euclidean norm.

$\wedge$ and $\vee$ will be used to denote the min and max operators
respectively, and $C$, with or without subscript, will always denote a finite
constant which may vary from line to line.

The notation $\mathcal{C}^{n}_{b}$ will be used when denoting a class of
functions which have bounded derivatives up to $n^{th}$-order.

Given vector spaces $A$ and $B$, $\mathcal{L} (A; B)$ will denote the space of
linear maps from $A$ to $B$.

Frequently, we will canonically identify a tensors $\displaystyle \sum
_{i=1}^{k} \sum_{j=1}^{l} a_{ij} \, e_{i} \otimes e_{j} $ (or co-tensors
$\displaystyle \sum_{i=1}^{k} \sum_{j=1}^{l} a_{ij} \, \mathrm{d}e_{i}
\otimes\mathrm{d}e_{j} $) in $\mathbb{R}^{k} \otimes\mathbb{R}^{l}$ with a
$k$-by-$l$ matrix.

For simplification, we will denote both tensor spaces and co-tensor spaces
with $\mathbb{R}^{k} \otimes\mathbb{R}^{l}$, and if $A \in\mathbb{R}^{k}
\otimes\mathbb{R}^{k}$, $\displaystyle \mathrm{tr} \, A := \sum_{i=1}^{k}
a_{ii}$ will denote the usual trace operation.

$\displaystyle \mathcal{I}_{k} := \sum_{j=1}^{k} e_{j} \otimes e_{j}$ will be
used to denote the $k$-by-$k$ identity matrix.

\subsection{Rough paths, $p$-variation topology and controls}

We begin by reviewing the basic concepts and notation of rough paths theory.
The standard references in this area include \cite{lyons98}, \cite{lq2003},
\cite{fh2014} and \cite{fv2010b}.

Given $x \in\mathcal{C} \left( [0, T]; \mathbb{R}^{d}\right) $, i.e. a
continuous $\mathbb{R}^{d}$-valued path defined on the time interval $[0, T]$,
where $T$ is some arbitrary but fixed terminal time, a rough path $\mathbf{x}$
includes the higher-order iterated integrals of $x$ in addition to the
first-order increment $x_{s,t} := x_{t} - x_{s}$. To account for these
higher-order increments, the right space for $\mathbf{x}$ to take values in
turns out to be the step-$n$ nilpotent group, which we will define below.

First, let $T^{n} \left(  \mathbb{R}^{d} \right) $ denote the truncated tensor
algebra up to degree $n$:
\begin{align*}
T^{n} \left(  \mathbb{R}^{d} \right)  := \mathbb{R} \oplus\mathbb{R}^{d}
\oplus\cdots\oplus\left(  \mathbb{R}^{d} \right) ^{\otimes n}.
\end{align*}
Addition and scalar multiplication are defined in the usual fashion, and given
\newline$a = \left(  a^{0}, a^{1}, \ldots, a^{n} \right) , b = \left(  b^{0},
b^{1}, \ldots, b^{n} \right)  \in T^{n} \left(  \mathbb{R}^{d} \right) $,
multiplication is performed by
\begin{align*}
a \otimes b := \left(  c^{0}, c^{1}, \ldots, c^{n} \right) , \qquad c^{k} =
\sum_{i=0}^{k} a^{i} \otimes b^{k-i}, \quad\forall\, 0 \leq k \leq n,
\end{align*}
where here we abuse the notation by re-using the same symbol for the tensor
product in $\mathbb{R}^{d}$.

The tangent space of $T^{n}(\mathbb{R}^{d})$ at the unit element $e=(1, 0
\ldots, 0)$ is given by
\begin{align*}
A_{T}^{n} \left(  \mathbb{R}^{d} \right)  := 0 \oplus\mathbb{R}^{d}
\oplus\cdots\oplus\left(  \mathbb{R}^{d} \right) ^{\otimes n}.
\end{align*}
We will define the exponential map $\exp: A_{T}^{n} (\mathbb{R}^{d})
\rightarrow T^{n}(\mathbb{R}^{d}) $ by
\begin{align}
\label{expDefn}\exp(a) := \sum_{i=0}^{n} \frac{a^{\otimes i}}{i!},
\end{align}
(for $a \in\mathbb{R}^{d}$ we will occasionally abuse the notation by denoting
$\exp(a) := \exp((0, a, 0))$), and the logarithm map $\log: T^{n}%
(\mathbb{R}^{d}) \rightarrow A_{T}^{n}(\mathbb{R}^{d})$ by
\begin{align}
\label{logDefn}\log(a) = \sum_{i=1}^{n} (-1)^{i+1} \frac{(a - e)^{\otimes i}%
}{i}.
\end{align}

\begin{definition}
The step-$n$ nilpotent group (with $d$ generators), denoted by $G^{n} \left(
\mathbb{R}^{d} \right) $, is the subgroup of $T^{n} \left(  \mathbb{R}^{d}
\right) $ corresponding to the sub-Lie algebra of $A_{T}^{n} \left(
\mathbb{R}^{d} \right) $ generated by the Lie bracket $[a, b] = a \otimes b -
b \otimes a$.
\end{definition}

One can check that the inverse of any element $a = e + \tilde{a} \in
G^{n}\left( \mathbb{R}^{d}\right) $ is given by
\begin{align}
a^{-1} = \sum_{k=0}^{n} (-1)^{k} \tilde{a}^{\otimes k};
\end{align}
see Lemma 7.16 in \cite{fv2010b}.

$G^{n} \left(  \mathbb{R}^{d} \right) $ will be equipped with the topology
induced from the symmetric, sub-additive homogeneous norm
\begin{align}
\label{gNorm}\left\|  g\right\|  = \max_{i=1, \ldots, n} \left( i! \, \left|
g^{i}\right|  \right) ^{\frac{1}{i}}.
\end{align}

Consider now $\mathbf{x} \in\mathcal{C}\left(  [0,T];G^{n}\left(
\mathbb{R}^{d}\right)  \right)  $, a continuous $G^{n} \left(  \mathbb{R}^{d}
\right) $ valued path on $[0, T]$. We define the increment by setting
$\mathbf{x}_{s,t} := \mathbf{x}_{s}^{-1} \otimes\mathbf{x}_{t}$. Such a path
is called a multiplicative functional (cf. \cite{lcl2006}) as it satisfies
Chen's equality
\begin{align}
\label{chensEq}\mathbf{x}_{s,t} = \mathbf{x}_{s, u} \otimes\mathbf{x}_{u,t}
\quad\forall\; s, u, t \in[0, T], s \leq u \leq t.
\end{align}
We now define the $p$-variation distance as
\begin{align}
\label{pvarDist}d_{p-var; [0, T]} (\mathbf{x},\mathbf{y}):=\sup_{\pi}\left(
\sum_{i}d(\mathbf{x}_{t_{i},t_{i+1}},\mathbf{y}_{t_{i},t_{i+1}})^{p}\right)
^{\frac{1}{p}},
\end{align}
where the supremum runs over all partitions $\pi=\{t_{i}\}$ of $[0,T]$. We
also define
\begin{align*}
\left\Vert \mathbf{x}\right\Vert _{p-var; [0, T]}:=d_{p-var; [0, T]}
(\mathbf{x},0),
\end{align*}
where $0$ denotes the constant path $\mathbf{y}_{t}=e$ for all $t\in[0,T]$.

\begin{definition}
For $p \geq1$, the \textit{weakly geometric p-rough paths}, which we will
denote by $\mathcal{C}^{p-var}\left( [0,T];G^{\lfloor p\rfloor} \left(
\mathbb{R}^{d}\right)  \right)  $, is the set of continuous functions
$\mathbf{x}$ from $[0,T]$ onto $G^{\lfloor p\rfloor} \left(  \mathbb{R}^{d}
\right) $ such that $\left\|  \mathbf{x}\right\| _{p-var; [0, T]} < \infty$.
\end{definition}

The simplest example of a weakly geometric $p$-rough path is as follows, given
a bounded-variation path $x$ in $\mathbb{R}^{d}$, we can compute the signature
of $x$ in $G^{\lfloor p\rfloor} \left(  \mathbb{R}^{d} \right) $:
\begin{align*}
S_{\lfloor p\rfloor}(x)_{s,t}=\left(  1,\mathbf{x}_{s,t}^{1},\mathbf{x}%
_{s,t}^{2},\ldots,\mathbf{x}_{s,t}^{\lfloor p\rfloor}\right) ,
\end{align*}
where $\mathbf{x}_{s,t}^{k}$ is the conventional $k$-th iterated integral of
the path $x$ over the interval $[s,t]$:
\begin{align*}
\mathbf{x}_{s,t}^{k} = \sum_{j_{1}, \ldots, j_{k}=1}^{d} \left(  \int%
_{s<r_{1}<\cdots<r_{k}<t}\,\mathrm{d}x^{(j_{1})}_{r_{1}}\otimes\cdots
\otimes\,\mathrm{d}x^{(j_{k})}_{r_{k}} \right)  e_{j_{1}} \otimes\cdots\otimes
e_{j_{k}}.
\end{align*}
Let $\mathcal{C}^{\infty} \left(  [0, T]; G^{\lfloor p \rfloor} \left(
\mathbb{R}^{d}\right)  \right) $ denote the subset of weakly-geometric rough
paths which are also of bounded variation. Then the signature of $x$ is in
\newline$\mathcal{C}^{\infty} \left(  [0, T]; G^{\lfloor p \rfloor} \left(
\mathbb{R}^{d}\right)  \right) $, and we also have the following definition.

\begin{definition}
For $p \geq1$, the space of \textit{geometric p-rough paths}, which we will
denote by $\mathcal{C}^{0,p-var}\left(  [0,T]; G^{\lfloor p \rfloor}\left(
\mathbb{R}^{d}\right)  \right) $, is defined to be the closure of
$\mathcal{C}^{\infty} \left(  [0, T]; G^{\lfloor p \rfloor} \left(
\mathbb{R}^{d}\right)  \right) $ in $\mathcal{C}^{p-var} \left( [0, T];
G^{\lfloor p \rfloor} \left( \mathbb{R}^{d}\right) \right) $ with respect to
the topology given by the $p$-variation distance \eqref{pvarDist}.
\end{definition}

\begin{remark}
In finite dimensions, the difference between weakly-geometric rough paths and
geometric rough paths is fairly minor, as we have the following relation
\begin{align*}
\mathcal{C}^{0,p-var}\left(  [0,T]; G^{\lfloor p \rfloor}\left(
\mathbb{R}^{d}\right)  \right)   & \subset\mathcal{C}^{p-var}\left(  [0,T];
G^{\lfloor p \rfloor}\left(  \mathbb{R}^{d}\right)  \right) \\
& \subset\mathcal{C}^{0,p + \varepsilon-var}\left(  [0,T]; G^{\lfloor p
\rfloor}\left(  \mathbb{R}^{d}\right)  \right) ,
\end{align*}
where $\varepsilon> 0$ can be chosen arbitrarily small, cf. \cite{fv2006}.
\end{remark}

We will now extend the notion of finite $p$-variation to general metric
spaces. Given a metric space $(E, d)$, a path $f:[0,T] \rightarrow E$ is said
to have finite $p$-variation if
\begin{align}
\label{pvarBound}\left\|  f\right\| _{p-var;[s,t]}:=\sup_{\pi}\left(  \sum_{i}
d(f_{t_{i}}, f_{t_{i+1}})_{E}^{p} \right) ^{\frac{1}{p}} < \infty.
\end{align}
We will use $V^{p-var}\left( [0, T]; E \right) $ to denote the space of
functions which satisfy the bound above, and $\mathcal{C}^{p-var} \left(  [0,
T]; E \right) $ to denote the set of continuous functions which satisfy \eqref{pvarBound}.

We will also define the notation $\mathcal{C}^{p-var}_{pw} \left(  [0, T]; E
\right) $ for piecewise-continuous functions of bounded $p$-variation as
follows: $f$ is in $\mathcal{C}^{p-var}_{pw} \left(  [0, T]; E \right) $ if
there exists a partition $\{ t_{i} \}$ of $[0, T]$ such that $f$ is in
$\mathcal{C}^{p-var} \left(  (t_{i}, t_{i+1}); E \right) $ for all $i$.

We have the following simple proposition; cf. Proposition 5.3 in
\cite{fv2010b}.

\begin{proposition}
Let $f \in\mathcal{C}([0, T]; E)$. Then if $1 \leq p \leq p^{\prime}< \infty
$,
\begin{align*}
\left\|  f\right\| _{p^{\prime}-var; [0, T]} \leq\left\|  f\right\| _{p-var;
[0, T]}.
\end{align*}
In particular, $\mathcal{C}^{p-var} \left(  [0, T]; E \right)  \subset
\mathcal{C}^{p^{\prime}-var} \left(  [0, T]; E \right) $.
\end{proposition}

We will use the notation $\left\|  f\right\| _{p-var;[s,t]}$ when the supremum
is taken over partitions of $[s,t] \subset[0,T]$. Proposition 5.8 in
\cite{fv2010b} tells us that
\begin{align*}
\omega(s,t) := \left\|  f\right\| _{p-var; [s,t]}^{p}%
\end{align*}
defines a control, i.e. a continuous, non-negative, real-valued function that
is super-additive and vanishes on the diagonal, i.e. $\omega(t,t)=0$ for all
$t\in[0,T]$. We also note the following lemmas about controls.

\begin{lemma}
\label{controlLem1} Assume $\omega_{1}$ and $\omega_{2}$ are controls.
\vspace{-5pt}

\begin{enumerate}
[(i)]

\item If $\phi$ is a positive convex function, then $\phi(\omega_{1})$ is a control.

\item Given $\alpha, \beta> 0$ with $\alpha+ \beta\geq1$, $\omega_{1}^{\alpha}
\, \omega_{2}^{\beta}$ is also a control.
\end{enumerate}
\end{lemma}

\begin{proof}
Exercises 1.8 and 1.9 in \cite{fv2010b}.
\end{proof}

\begin{lemma}
[Proposition 5.10 in \cite{fv2010b}]\label{controlLem2} Let $\omega$ be a
control on $[0, T]$ and let $p \geq1$. Then the point-wise estimate
\begin{align*}
d(f_{s}, f_{t})^{p} \leq C \, \omega(s, t) \quad\forall\, s < t \in[0, T]
\end{align*}
implies the $p$-variation estimate
\begin{align*}
\left\|  f\right\| _{p-var; [s, t]} \leq C^{\frac{1}{p}} \, \omega
(s,t)^{\frac{1}{p}} \quad\forall\, s < t \in[0, T].
\end{align*}

\end{lemma}

If $E$ also has a norm $\left\|  \cdot\right\| _{E}$, we will also use the
notation $\left\|  f\right\| _{\mathcal{V}^{p}; [0, T]}$ to denote the norm
\begin{align*}
\left\|  f\right\| _{\mathcal{V}^{p}; [0, T]}  & := \left\|  f\right\|
_{p-var;[0,T]} + \sup_{t\in\lbrack0,T]}\left\Vert f_{t}\right\Vert _{E}\\
& \leq\left\|  f_{0}\right\| _{E} + 2 \left\|  f\right\| _{p-var; [0, T]}.
\end{align*}

For a function defined on $[0,T]^{2}$, $f:[0,T]^{2}\rightarrow E$ is said to
be of finite 2D $p$-variation if
\begin{align}
\label{2DpvarBound}\left\|  f\right\| _{p-var;[0,T]^{2}}:=\sup_{\pi}\left(
\sum_{i,j}\left\Vert f
\begin{pmatrix} u_{i},u_{i+1} \\ v_{j},v_{j+1} \end{pmatrix} \right\Vert
_{E}^{p}\right)  ^{\frac{1}{p}}<\infty,
\end{align}
where $\pi=\left\{  \left(  u_{i},v_{j}\right)  \right\}  $ is a partition of
$[0, T]^{2}$, and the rectangular increment is given by
\begin{align}
\label{rectInc}f
\begin{pmatrix} u_{i},u_{i+1} \\ v_{j},v_{j+1} \end{pmatrix} := f(u_{i}%
,v_{j})+f(u_{i+1},v_{j+1})-f(u_{i},v_{j+1})-f(u_{i+1},v_{j}).
\end{align}
Similar to the 1D case, we will use $V^{p-var} \left(  [0, T]^{2}; E\right) $
(resp. $\mathcal{C}^{p-var} \left( [0, T]^{2}; E\right) $) to denote the set
of functions (resp. continuous functions) which satisfy \eqref{2DpvarBound}.

On occasion, we will use the notation
\begin{align}
\label{2dInc}%
\begin{split}
f( \Delta_{i}, v) := f(u_{i+1}, v) - f(u_{i}, v),\\
f(u, \Delta_{j}) := f(u, v_{j+1}) - f(u, v_{j}).
\end{split}
\end{align}
Given a rectangle $R = [s, t] \times[u,v] \subset[0, T]^{2}$, we let $\Pi(R)$
denote the family of all partitions $\{[a_{i},b_{i}]\times[c_{i},d_{i}], \, i
= 1, \ldots, n\}$ of $R$, where in each partition, the sub-rectangles have
disjoint interiors and satisfy $\displaystyle \bigcup_{i=1}^{n} [a_{i}%
,b_{i}]\times[c_{i},d_{i}] = R$.

\begin{definition}
Let $\Delta_{T} := \{(s, t) \:\vert\: 0 \leq s \leq t \leq T\}$. A map
$\omega: \Delta_{T} \times\Delta_{T} \rightarrow[0, \infty)$ is called a 2D
control if it is continuous, zero on degenerate rectangles, and super-additive
in the sense that for all rectangles $R \subset[0, T]^{2}$,
\begin{align*}
\sum_{i=1}^{n} \omega(R_{i}) \leq\omega(R), \quad\mathrm{whenever} \, \{R_{i},
\, 1 \leq i \leq n \} \in\Pi(R).
\end{align*}

\end{definition}

Now let $f \in\mathcal{C}^{p-var} \left( [0, T]^{2}; E\right) $, $\varepsilon>
0$ and $R \subset[0, T]^{2}$. Then $\omega_{f,\rho+\varepsilon}^{\vphantom{0}}%
(R)$ defined as
\begin{align}
\label{controlledVar}\omega_{f,\rho+\varepsilon}^{\vphantom{0}}(R) :=
\sup_{\Pi(R) }\sum_{i} \left|  f
\begin{pmatrix} a_{i} & b_{i} \\ c_{i} & d_{i} \end{pmatrix} \right|
^{\rho+\varepsilon},
\end{align}
is a 2D control \cite{fv2011}.

\begin{remark}
Note that $\Pi(R)$ includes partitions which are not grid-like, in contrast to
\eqref{2DpvarBound}. Furthermore, we use $\omega_{f,\rho+\varepsilon} ([s, t]
\times[u, v])$ instead of what seems to be the more natural choice $\left\|
f\right\| _{\rho-var;[s,t]\times[u,v]}^{\rho}$ because the latter is actually
not super-additive, and is thus not a control; see \cite{fv2011}.

However, there exists a finite constant depending on $\varepsilon$ such that
\begin{align*}
\omega_{f,\rho+\varepsilon}^{\vphantom{0}}([s,t]\times[u,v]) \leq
C_{\varepsilon} \left\|  f\right\| _{\rho-var;[s,t]\times[u,v]}^{\rho
+\varepsilon} < \infty,
\end{align*}
for all $[s, t] \times[u, v] \subset[0, T]^{2}$. The reason to use
$\rho+\varepsilon$ regularity instead of $\rho$ is that otherwise
\eqref{controlledVar} might be infinite; cf. Example 1 in \cite{fv2011}.
\end{remark}

\begin{definition}
We say that the 2D Young-Stieltjes integral of $f$ with respect to $g$ exists
if there exists a scalar $I(f, g) \in\mathbb{R}$ such that
\begin{align}
\label{quant1}\lim_{\left\|  \pi\right\|  \rightarrow0} \left|  \sum_{i, j}
f\left(  u_{i}, v_{j} \right)  g
\begin{pmatrix} u_i & u_{i+1} \\ v_j & v_{j+1} \end{pmatrix} - I(f, g)
\right|  \rightarrow0,
\end{align}
i.e. for each $\varepsilon> 0$, there exists a $\delta> 0$ such that for all
partitions $\pi= \{ (u_{i}, v_{j}) \}$ of $[0, T]^{2}$ with $\left\|
\pi\right\|  < \delta$, the quantity on the left of \eqref{quant1} is less
than $\varepsilon$. In this case, we use $\int_{[0, T]^{2}} f \, \mathrm{d} g$
to denote $I(f, g)$, or $\int_{[s,t] \times[u, v]} f \, \mathrm{d} g$ whenever
we restrict ourselves to any particular subset $[s, t] \times[u,v]$ of $[0,
T]^{2}$.
\end{definition}

\begin{definition}
We say that $f \in V^{p-var}([s, t] \times[u,v])$ and \newline$g \in
V^{q-var}([s, t] \times[u,v])$ have complementary regularity if $\frac{1}{p} +
\frac{1}{q} > 1$.
\end{definition}

The significance of this definition lies in the following theorem, which gives
the existence of the Young-Stieltjes integral and Young's inequality in two
dimensions; see \cite{lcl2006}, \cite{fh2014}, \cite{fv2010b} for the
one-dimensional version.

\begin{theorem}
\label{2Dintegral} Let $f \in V^{p-var}([s, t] \times[u, v])$ and $g \in
V^{q-var}([s, t] \times[u, v])$ have complementary regularity. We also assume
that $f(s, \cdot)$ and $f(\cdot, u)$ have finite p-variation, and that $f$ and
$g$ have no common discontinuities. Then the 2D Young-Stieltjes integral
exists and the following Young's inequality holds;
\begin{align}
\label{2DYoungIneq}%
\begin{split}
\left|  \int_{[s,t] \times[u,v]} f \, \mathrm{d} g\right|  \quad\leq C_{p, q}
\, {\left\vert \kern-0.25ex\left\vert \kern-0.25ex\left\vert f \right\vert
\kern-0.25ex\right\vert \kern-0.25ex\right\vert } \left\|  g\right\| _{q-var,
[s,t] \times[u,v]},
\end{split}
\end{align}
where
\begin{align*}
{\left\vert \kern-0.25ex\left\vert \kern-0.25ex\left\vert f \right\vert
\kern-0.25ex\right\vert \kern-0.25ex\right\vert } = \left|  f(s, u)\right|  +
\left\|  f(s, \cdot)\right\| _{p-var; [u, v]} + \left\|  f(\cdot, u)\right\|
_{p-var; [s, t]} + \left\|  f\right\| _{p-var, [s,t] \times[u,v]}.
\end{align*}

\end{theorem}

\begin{proof}
See \cite{towghi2002}, \cite{fv2010a}.
\end{proof}

\subsection{Gaussian rough paths}

We will work with a stochastic process
\begin{align*}
X_{t} = \left( X_{t}^{(1)},\ldots, X_{t}^{(d)}\right) , \quad t \in[0, T],
\end{align*}
which denotes a centered (i.e. zero-mean), continuous Gaussian process in
$\mathbb{R}^{d}$ with i.i.d. components and starting at the origin.

This process is defined on the canonical probability space $\left(  \Omega,
\mathcal{F},\mathbb{P}\right)  $, where $\Omega= \mathcal{C} \left(
[0,T];\mathbb{R}^{d}\right)  $, the space of continuous $\mathbb{R}^{d}%
$-valued paths equipped with the supremum topology, $\mathcal{F}$ is the Borel
$\sigma$-algebra and $\mathbb{P}$ is the unique Borel measure under which
$X\left(  \omega\right)  =\left(  \omega_{t}\right)  _{t\in\left[  0,T\right]
}$ has the specified Gaussian distribution. We will use
\begin{align*}
R(s,t):=\mathbb{E}\left[  X_{s}^{(1)}X_{t}^{(1)}\right]
\end{align*}
to denote the covariance function common to the components. The variance
$R\left(  t,t\right)  $ will be denoted simply by $R(t)$, and we will also use
the notation
\begin{align}
\label{sigmaDefn}\sigma^{2}(s,t):=R
\begin{pmatrix} s & t \\ s & t \end{pmatrix} = \mathbb{E} \left[  \left(
X_{s,t}^{(1)}\right) ^{2}\right] ;
\end{align}
recall the definition of the rectangular increment in \eqref{rectInc}.

The triple $\left(  \Omega,\mathcal{H}^{d},\mathbb{P} \right)  $ denotes the
abstract Wiener space associated to $X$, where $\mathcal{H}^{d}=\bigoplus
_{i=1}^{d}\mathcal{H}$ is the Cameron-Martin space (or reproducing kernel
Hilbert space). The Cameron-Martin space, which is densely and continuously
embedded in $\Omega\text{,}$ is the completion of the linear span of the
functions
\begin{align*}
\left\{  R(t, \cdot)^{(u)} := R(t,\cdot) \, e_{u} \; \bigg\vert \; t\in[0,T],
\; u = 1, \ldots, d \right\}
\end{align*}
under the inner-product
\begin{align*}
\left\langle R(t, \cdot)^{(u)}, R(s,\cdot)^{(v)} \right\rangle _{\mathcal{H}%
^{d}} = \delta_{uv} \, R(t,s), \quad u, v = 1, \ldots, d.
\end{align*}
By definition, $\mathcal{H}^{d}$ satisfies the following reproducing property;
for any $f=\left(  f^{(1)},\ldots, f^{(d)}\right)  \in\mathcal{H}^{d}$,
\begin{align*}
\left\langle f_{\cdot}, R(t,\cdot)^{(u)} \right\rangle _{\mathcal{H}^{d}} =
f_{t}^{(u)}, \quad t \in[0,T].
\end{align*}

We assume that there exists $\rho< 2$ such that $R$ has finite 2D $\rho
$-variation. The following theorem in \cite{fv2010a} (see also \cite{cq2002}
in the case of fractional Brownian motion) then shows that one can canonically
lift $X$ via its piecewise linear approximants $X^{\pi}$ to a geometric
$p$-rough path for $p > 2 \rho$.

\begin{theorem}
\label{gaussianRP} Assume $X$ is a centered continuous $\mathbb{R}^{d}$-valued
Gaussian process with i.i.d. components. Let $\rho\in[1, 2)$ and assume that
the covariance function has finite 2D $\rho$-variation.

\begin{enumerate}
[(i)]

\item (Existence) There exists a random variable $\mathbf{X}=\left(  1,
\mathbf{X}^{1}, \mathbf{X}^{2}, \mathbf{X}^{3}\right)  $ on $\left(  \Omega,
\mathcal{F},\mathbb{P}\right)  $ which takes values almost surely in
$\mathcal{C}^{0, p-var} \left(  [0, T]; G^{3}(\mathbb{R}^{d}) \right) $ for $p
> 2\rho$, i.e. the set of geometric $p-$rough paths for $p \in(2\rho, 4)$, and
which lifts the Gaussian process $X$ in the sense that $\mathbf{X}_{s,t}%
^{1}=X_{t}-X_{s}$ almost surely for all $s,t\in[0, T]$.

\item (Uniqueness and consistency) The lift $\mathbf{X}$ is unique in the
sense that it is the $d_{p -v a r}$-limit in $L^{q} (\Omega)$, $q \in
[1,\infty)$, of any sequence $S_{^{\left\lfloor p\right\rfloor }} (X^{\pi}) $
with $\left\|  \pi\right\|  \rightarrow0$. Furthermore, if $X$ has a.s. sample
paths of finite $[1 ,2)$-variation, $\mathbf{X}$ coincides with the signature
of $X$.
\end{enumerate}
\end{theorem}

Moreover, Proposition 17 in \cite{fv2010a} shows that for all $h
\in\mathcal{H}^{d}$,
\begin{align}
\label{CMembedding}\left\|  h\right\| _{\rho-var; [0, T]} \leq\left\|
h\right\| _{\mathcal{H}^{d}} \sqrt{\left\|  R\right\| _{\rho-var; [0, T]^{2}}%
},
\end{align}
which implies that $\mathcal{H}^{d} \hookrightarrow\mathcal{C}^{\rho-var}([0
,T] ;\mathbb{R}^{d})$ whenever $R$ has finite 2D $\rho$-variation. Thus if
$\rho\in\left[ 1, \frac{3}{2}\right) $, corresponding to $2 \leq p < 3$, we
have complementary regularity between $X$ and any path in the Cameron-Martin
space, i.e. $\frac{1}{p} + \frac{1}{\rho} > 1$.

We will need to impose further conditions on the covariance function. For all
$s, t \in[0 ,T]$, we assume there exists $C < \infty$ such that
\begin{align}
\label{R1DBound}\left\|  R(t, \cdot) - R (s, \cdot) \right\| _{\rho-var; [0,
T]} \leq C \left|  t - s\right| ^{\frac{1}{\rho}}.
\end{align}
This bound will be later used to control the $L^{2} (\Omega)$ norm of the
iterated integrals. An immediate consequence of the bound is illustrated in
the following lemma.

\begin{lemma}
\label{RdotRhoVar} Let $X$ be a continuous, centered Gaussian process in
$\mathbb{R}$ and assume its covariance function satisfies
\begin{align*}
\left\|  R (t , \cdot) - R (s , \cdot)\right\| _{q-var; [0, T]} \leq C \left|
t -s\right| ^{\frac{1}{\rho}}, \quad\forall s < t \in[0, T],
\end{align*}
for some $q, \rho\geq1$. Then

\begin{enumerate}
[(i)]

\item $R(t) := R(t, t)$ is of bounded $\rho$-variation.

\item For $p > 2 \rho$, $X$ has a $\frac{1}{p}$-H\"{o}lder continuous modification.
\end{enumerate}
\end{lemma}

\begin{proof}
Let $f_{s,t} (\cdot)$ denote $R(t, \cdot) - R(s, \cdot)$. Then for any partition $\left\{ t_i \right\}$ of $[0, T]$, we have
\begin{align*}
\sum_i &\abs{R\left(t_{i+1}\right) - R\left(t_i\right)}^{\rho} \\
&\quad\qquad\leq \sum_i \abs{R\left(t_{i+1}, t_{i+1} \right) - R\left(t_i, t_{i+1}\right) + R\left(t_i, t_{i+1}\right) - R\left( t_i, t_i \right)}^{\rho} \\
&\quad\qquad\leq 2^{\rho - 1} \sum_i \abs{f_{t_i, t_{i+1}} \left( t_{i+1} \right) - f_{t_i, t_{i+1}} (0)}^{\rho} + \abs{f_{t_i, t_{i+1}} \left( t_i \right) - f_{t_i, t_{i+1}} (0)}^{\rho} \\
&\quad\qquad\leq 2^{\rho} \sum_i \norm{f_{t_i, t_{i+1}} (\cdot)}_{q-var; [0, T]}^{\rho} \\
&\quad\qquad\leq C \sum_i \abs{t_{i+1} - t_i} \leq C \, T.
\end{align*}
For the second part, for all $n \in \mathbb{N}$ we have
\begin{align*}
\exptn{X_{s,t}^{2n}} \leq C_n \exptn{X_{s,t}^2}^n
&\leq C_n \abs{f_{s,t}(t) - f_{s, t} (s)}^n \\
&\leq C_n \norm{R(t, \cdot) - R(s, \cdot)}_{q-var; [0, T]}^n \\
&\leq C_n \abs{t - s}^{\frac{n}{\rho}}, \quad s < t \in [0, T].
\end{align*}
By Kolmogorov's continuity theorem, there exists a $\gamma$-H\"{o}lder continuous modification of $X$ for all $\gamma < \frac{1}{2\rho}$. \par
\end{proof}

\subsection{Malliavin calculus}

\label{MDsection} We will primarily work with the following Hilbert space
which is isomorphic to $\mathcal{H}^{d}$.

\begin{definition}
Let $\mathcal{H}_{1}^{d}$ denote the completion of the linear span of
\begin{align*}
\left\{  \mathds{1}_{[0,t)}^{(u)} (\cdot) := \mathds{1}_{[0,t)} (\cdot) \,
e_{u} \; \bigg\vert \; t \in[0,T], \; u=1,\ldots,d \right\}
\end{align*}
(cf. \cite{amn2001}, \cite{nualart2006}) with respect to the inner-product
given by
\begin{align*}
\left\langle \mathds{1}_{[0,t)}^{(u)} (\cdot), \mathds{1}_{[0,s)}^{(v)}
(\cdot) \right\rangle _{\mathcal{H}_{1}^{d}} =\delta_{uv} \, R(t,s).
\end{align*}
Furthermore, let $\Phi: \mathcal{H}_{1}^{d} \rightarrow\mathcal{H}^{d}$ denote
the Hilbert space isomorphism obtained from extending the map
$\mathds{1}_{[0,t)}^{(u)}(\cdot) \mapsto R(t,\cdot)^{(u)}$, $t \in[0, T],\;
u=1, \ldots, d$.
\end{definition}

We record some basic properties about the Malliavin calculus. For simplicity,
we assume here that $d=1$. First we recall that the map $\mathds{1}_{\left[
0,t\right)  }\mapsto X_{t}$ extends to a unique linear isometry $I$ from
$\mathcal{H}_{1}$ to $L^{2}\left(  \Omega\right)  $. It follows that $I\left(
h\right)  $ is a mean-zero Gaussian random variable with variance $\left\Vert
h\right\Vert _{\mathcal{H}_{1}}^{2}$. The set $\mathcal{S}$ of smooth cylinder
functionals is the subset of random variables taking the form
\begin{align*}
F=f(I\left(  h_{1}\right)  ,\ldots,I\left(  h_{n}\right)  ),
\end{align*}
where $h_{1}, \ldots, h_{n}\in\mathcal{H}_{1}$ and $f:\mathbb{R}%
^{n}\rightarrow\mathbb{R}$ is infinitely differentiable with bounded
derivatives of all orders. The Malliavin derivative $\mathcal{D}F$ is the
$\mathcal{H}_{1}$-valued random variable which is defined for smooth cylinder
functionals as follows:
\begin{align*}
\mathcal{D}F:=\sum_{i=1}^{n}\frac{\partial f}{\partial x_{i}}\left(  I\left(
h_{1}\right)  ,\ldots,I\left(  h_{n}\right)  \right)  h_{i}.
\end{align*}
It can be shown that $\mathcal{D}$ is a closable operator, see e.g.
Proposition 1.2.1 in \cite{nualart2006}. For $p\geq1$ we let $\mathbb{D}%
^{1,p}$ denote the closure of $\mathcal{S}$ with respect to the norm
\begin{align*}
\left\Vert F\right\Vert _{1,p}^{p}:=\left\Vert F\right\Vert _{L^{p}\left(
\Omega\right)  }^{p}+\left\Vert \mathcal{D}F\right\Vert _{L^{p}\left(
\Omega;\mathcal{H}_{1}\right)  }^{p}.
\end{align*}
If $K$ is a separable Hilbert space, the higher-order derivatives
$\mathcal{D}^{n}$ and the corresponding Sobolev spaces $\mathbb{D}^{n, p} (K)$
can be defined iteratively.

Moving to the case $p=2$, for any $F$ in $\mathbb{D}^{1,2}$, we let
$\mathcal{D}_{h}F:=\left\langle \mathcal{D}F,h\right\rangle _{\mathcal{H}_{1}%
}$. The divergence operator $\delta^{X}$ is defined to be the adjoint of
$\mathcal{D}$. The domain of this operator consists of all $h\in L^{2}\left(
\Omega;\mathcal{H}_{1}\right)  $ such that
\begin{align*}
\left\vert \mathbb{E}\left[  \mathcal{D}_{h}F\right]  \right\vert \leq
C\left\Vert F\right\Vert _{L^{2}\left(  \Omega\right)  }%
\end{align*}
for all $F\in\mathcal{S},$ whereupon $\delta^{X}\left(  h\right)  $ is
characterized as the unique random variable in $L^{2}\left(  \Omega\right) $
for which
\begin{align*}
\mathbb{E}\left[  \left\langle \mathcal{D}F,h\right\rangle _{\mathcal{H}_{1}%
}\right]  =\mathbb{E}\left[  F\delta^{X}\left(  h\right)  \right] .
\end{align*}
We will use the notation $\delta^{X}(h)$ and $\int_{0}^{T} h_{s}
\,\mathrm{d}X_{s}$ interchangeably. It is well-known that the domain of
$\delta^{X}$ contains $\mathbb{D}^{1,2}\left(  \mathcal{H}_{1}\right)  $, see
e.g. Proposition 1.3.1 in \cite{nualart2006}.

Fixing a multi-index $a=(a_{1},\ldots,a_{M})$ where $\left|  a\right|  :=
\sum_{i=1}^{M}a_{i}=n$, we define $I_{n} : \mathcal{H}_{1}^{\otimes n}
\rightarrow\mathbb{R}$ as follows:
\begin{align*}
I_{n} \left(  h_{1}^{\otimes a_{1}}\otimes\cdots\otimes h_{M}^{\otimes a_{M}%
}\right)  = a! \prod_{i=1}^{M} H_{a_{i}}(\delta^{X} (h_{i})),
\end{align*}
where $a!:=\prod_{i=1}^{M}a_{i}!$ and $H_{m}(x)$ denotes the $m^{th}$ Hermite
polynomial. The following duality formula is then classical;
\begin{align}
\label{dualityFormula}\mathbb{E} \left[  F I_{n}(h) \right]  = \mathbb{E}
\left[  \left\langle D^{n}F, h \right\rangle _{\mathcal{H}_{1}^{\otimes n}}
\right] .
\end{align}
For $f \in\mathcal{H}_{1}^{\otimes n}$, $g \in\mathcal{H}_{1}^{\otimes m}$,
both $f$ and $g$ symmetric, we also have the following product formula (cf.
Proposition 1.1.3 in \cite{nualart2006})
\begin{align}
\label{productFormula}I_{n} (f) I_{m} (g) = \sum_{r=0}^{n \wedge m} r!
\begin{pmatrix} n \\ r \end{pmatrix} \begin{pmatrix} m \\ r \end{pmatrix} I_{n+m-2r}
\left(  f \tilde{\otimes}_{r} g \right) .
\end{align}
Here $f\tilde{\otimes}_{r} g$ denotes the symmetrization of the tensor $f
\otimes_{r} g$, which in turn denotes the contraction of $f$ and $g$ of order
$r$; i.e. given an orthonormal basis $\{h_{m}\}$ of $\mathcal{H}_{1}$,
\begin{align*}
f \otimes_{r} g := \sum_{k_{1}, \ldots, k_{r} = 1}^{\infty} \left\langle f,
h_{k_{1}} \otimes\cdots\otimes h_{k_{r}} \right\rangle _{\mathcal{H}%
_{1}^{\otimes r}} \otimes\left\langle g, h_{k_{1}} \otimes\cdots\otimes
h_{k_{r}} \right\rangle _{\mathcal{H}_{1}^{\otimes r}} \in\mathcal{H}%
_{1}^{\otimes(n + m - 2r)};
\end{align*}
cf. \cite{nnt2010}.

\begin{remark}
\label{mdRem} One can also define operators equivalent to $\mathcal{D}$ and
$\delta^{X}$ directly on the abstract Wiener space $\left(  \Omega
,\mathcal{H}, \mathbb{P} \right)  $. To make the presentation clear we
summarize the correspondence here. First, for every $l$ in the topological
dual $\Omega^{*} = \mathcal{C}\left(  \left[ 0,T\right]  ,\mathbb{R}\right)
^{*}$, there exists a unique $h_{l}$ in $\mathcal{H}$ such that $l\left(
h\right)  = \left\langle h_{l}, h \right\rangle $. Under this identification,
the random variable $\mathcal{I}\left(  h_{l}\right)  :\omega\mapsto l\left(
\omega\right)  $ is a centered normal random variable with variance $\left\|
h_{l}\right\| _{\mathcal{H}}^{2} $. Second, it can be shown that the set
$\left\{  h_{l}:l\in\Omega^{*} \right\}  $ is dense in $\mathcal{H}$,
whereupon $\mathcal{I}$ extends uniquely to an isometry between $\mathcal{H}$
and $L^{2}(\Omega)$, and is called the Paley-Wiener map. It is simple to see
that $I$ and $\mathcal{I}$ are related by $I\left(  h\right)  = \mathcal{I}%
\left(  \Phi(h) \right)  $ for all $h\in\mathcal{H}_{1}$, and therefore any
smooth cylinder functional $F$ can be represented as $F = f\left(
\mathcal{I}\left(  \Phi(h_{1}) \right) , \ldots, \mathcal{I}\left(  \Phi(h_{n}
)\right)  \right) $, and a derivative operator $\mathbf{D}$ can be defined
analogously to $\mathcal{D}$ by setting
\begin{align*}
\mathbf{D}F := \sum_{i=1}^{n}\frac{\partial f}{\partial x_{i}}\left(
\mathcal{I}\left(  \Phi(h_{1}) \right)  ,\ldots,\mathcal{I}\left(  \Phi(h_{n})
\right)  \right)  \Phi(h_{i}) = \Phi(\mathcal{D}F).
\end{align*}
This implies that
\begin{align*}
\mathbf{D}_{\Phi(h)} F = \left\langle \mathbf{D}F, \Phi(h) \right\rangle
_{\mathcal{H}} = \left\langle \mathcal{D}F, h \right\rangle _{\mathcal{H}_{1}}
= \mathcal{D}_{h} F, \quad\forall\, h\in\mathcal{H}_{1}.
\end{align*}

\end{remark}

The exposition above presents Shigekawa's definition of the Sobolev-type space
$\mathbb{D}^{n, p}\left(  K\right) $ for $K$-valued Wiener functionals, where
$K$ is a separable Hilbert space. Although this is the one most often used in
the literature, there are equivalent characterizations of these spaces. One of
these, which is attributed to Kusuoka and Stroock (cf. \cite{sugita1985}), is
especially convenient to study stochastic differential equations for which
bounds on the directional derivatives can be computed explicitly. The
definition relies on two properties. First, a measurable function
$F:\Omega\rightarrow K$ is called \textit{ray absolutely continuous} (RAC) if
for every $k\in\mathcal{H}$, there exists a measurable map $\tilde{F}_{k}:
\Omega\rightarrow K$ such that
\begin{align}
\label{version}F\left(  \cdot\right)  =\tilde{F}_{k}\left(  \cdot\right)
\text{, }\mathbb{P}-\text{a.s.},
\end{align}
and for any $\omega\in\Omega$ the function $s\mapsto\tilde{F}_{k} \left(
\omega+ sk \right)  $ is locally\footnote{Local absolute continuity is
important here and is a point often missed in the literature where RAC is
sometimes stated by demanding that $s\mapsto\tilde{F}_{k}\left(  \omega+
sk\right)  $ is absolutely continuous in $s\in\mathbb{R}$. See however
Definition 8.2.3 and Theorem 8.5.1 in \cite{bogachev2} for a proof that local
absolute continuity is enough.} absolutely continuous in $s\in\mathbb{R}$.
Second, $F$ has the property of being \textit{stochastically G\^{a}teaux
differentiable} (SGD) if there exists a measurable $G:\Omega
\mathcal{\rightarrow} \mathcal{L}_{\text{HS}}\left(  \mathcal{H}, K\right) $,
such that for any $k \in\mathcal{H}$
\begin{align*}
\frac{1}{\varepsilon}\left[  F\left(  \cdot+ \varepsilon k\right)  - F\left(
\cdot\right)  \right]  \overset{\mathbb{P}}{\rightarrow} G\left(
\omega\right)  \left(  k \right)  \quad\mathrm{as} \: t \rightarrow0,
\end{align*}
where $\mathcal{L}_{\text{HS}}\left(  \mathcal{H}, K\right) $ denotes the
space of linear Hilbert-Schmidt operators from $\mathcal{H}$ to $K$. In this
case, the derivative $G$ is unique $\mathbb{P}$-a.s. and we denote it by
$\mathcal{D}_{\text{KS}}F$. Higher order derivatives are defined inductively
in the obvious way: if $\mathcal{D}_{\text{KS}}^{n-1}F$ $\ $is SGD then
$\mathcal{D}_{\text{KS}}^{n}F := \mathcal{D}_{KS} \left(  \mathcal{D}%
_{\text{KS}}^{n-1}F \right) $.

Next, we define the spaces $\mathbb{D}_{\text{KS}}^{n,p}\left(  K\right)  $
for $1 < p < \infty$ inductively, first for $n = 1$ by setting
\begin{align*}
\mathbb{D}_{\text{KS}}^{1,p}\left(  K\right)  :=\left\{  F\in L^{p}\left(
K\right)  :F\text{ is RAC and SGD,}\,\mathcal{D}_{\text{KS}}F\in
L^{p}(\mathcal{L}_{\text{HS}}\left(  \mathcal{H},K\right)  )\right\}  ,
\end{align*}
and then analogously for $n = 2, 3, \ldots$ by
\begin{align*}
\mathbb{D}_{\text{KS}}^{n,p}\left(  K\right)  :=\left\{  F\in\mathbb{D}%
_{\text{KS}}^{n-1,p}\left(  K\right)  :\mathcal{D}_{\text{KS}}F \in
\mathbb{D}_{\text{KS}}^{n-1,p}(\mathcal{L}_{\text{HS}}\left(  \mathcal{H}%
,K\right)  )\right\}  .
\end{align*}
We have the following theorem.

\begin{theorem}
[Theorem 3.1 in \cite{sugita1985}]\label{sugitaThm} For $1 < p < \infty$ and
$n \in\mathbb{N}$ we have $\mathbb{D}_{\text{KS}}^{n,p}\left(  K\right)
=\mathbb{D}^{n,p}\left(  K\right) $, and for any element $F$ in this space,
$\mathcal{D}_{\text{KS}}F = \mathcal{D}F$ holds $\mathbb{P}$-a.s.
\end{theorem}

\begin{remark}
By applying the same result iteratively it follows that $\mathcal{D}%
_{\text{KS}}^{k}F = \mathcal{D}^{k}F$ holds $\mathbb{P}$-a.s for $k = 2,
\ldots, n$.
\end{remark}

\subsection{Rough integration and controlled rough paths}

\label{controlledRP} In this subsection, we will review rough integration via
the theory of controlled rough paths. We will develop the concepts in
$p$-variation topology rather than the usual H\"{o}lder topology (cf.
\cite{gub2004} and \cite{fh2014}), and henceforth, $\mathcal{U}, \mathcal{V}$
will denote finite-dimensional vector spaces.

We begin with the following definition.

\begin{definition}
Let $\mathbf{x} = \left(  1, x, \mathbf{x}^{2} \right)  \in\mathcal{C}^{p-var}
\left(  \left[  0,T\right] ; G^{2} \left(  \mathbb{R}^{d} \right)  \right) $.
A pair of paths $\left(  \phi,\phi^{\prime}\right) $, where $\phi
\in\mathcal{C}^{p-var} \left( [0, T]; \mathcal{U} \right) $ and $\phi^{\prime}
\in\mathcal{C}^{p-var} \left( [0, T]; \mathcal{L}(\mathbb{R}^{d}; \mathcal{U})
\right) $, is said to be controlled by $\mathbf{x}$ if for all $s,t \in[0,
T]$,
\begin{align}
\label{controlled}\phi_{s,t} = \phi^{\prime}_{s} x_{s,t} + R_{s,t}^{\phi},
\end{align}
where the remainder term satisfies
\begin{align*}
R^{\phi} \in\mathcal{C}^{\frac{p}{2}-var} \left(  \left[  0,T\right] ;
\mathcal{U} \right) .
\end{align*}

\end{definition}

If we define the controlled variation norm as
\begin{align*}
\left\|  \phi\right\| _{p-cvar} := \left\|  \phi\right\| _{\mathcal{V}^{p};
[0, T]} + \left\|  \phi^{\prime}\right\| _{\mathcal{V}^{p}; [0, T]} + \left\|
R^{\phi}\right\| _{\frac{p}{2}-var; \left[  0,T\right]  },
\end{align*}
then the preceding definition says that $(\phi, \phi^{\prime})$ is controlled
by $\mathbf{x}$ if \newline$\left\|  \phi\right\| _{p-cvar} < \infty$.

\begin{restatable} {theorem}{contThm} \label{controlledThm1}
Let $\mathbf{x}=\left( 1, x,\mathbf{x}^{2}\right) \in \mathcal{C}^{p-var}\left( [0, T]; G^2(\mathbb{R}^d)\right)$, where $2 \leq p <3$.  \par
Let $\phi \in \mathcal{C}^{p-var}\left([0,T]; \mathcal{L}(\mathbb{R}^{d};\mathbb{R}^{e})\right) $ and $\phi' \in \mathcal{C}^{p-var} \left( [0,T]; \mathcal{L}(\mathbb{R}^{d}; \mathcal{L}(\mathbb{R}^{d};\mathbb{R}^{e}))\right)$. If $\left( \phi ,\phi ^{\prime }\right) $ is controlled by $\mathbf{x}$, we can define the rough integral
\begin{align}  \label{controlledRPdefn1}
\int_0^t \phi _{r}\circ \mathrm{d}\mathbf{x}_{r}:=\underset{\left\Vert
\pi \right\Vert \rightarrow 0,\pi =\left\{ 0 = r_0 < \ldots < r_n = t \right\} }{\lim }\sum_{i=0}^{n-1}\left( \phi _{r_{i}}x_{r_{i},r_{i+1}}+\phi
_{r_{i}}^{\prime }\mathbf{x}_{r_{i},r_{i+1}}^{2}\right),
\end{align}
where we have made use of the canonical identification $\mathcal{L}(\mathbb{R}^{d}; \mathcal{L} (\mathbb{R}^{d};\mathbb{R}^{e}))\simeq \mathcal{L}(\mathbb{R}^{d}\otimes \mathbb{R}^{d}; \mathbb{R}^{e}).$ Furthermore, denoting
\begin{align*}
z_{t}:=\int_{0}^{t}\phi _{r}\circ \mathrm{d}\mathbf{x}_{r},\quad
z_{t}^{\prime }:=\phi _{t},
\end{align*}
$(z,z^{\prime })$ is again controlled by $\mathbf{x}$, and we have the bound
\begin{align}  \label{controlledBound1}
\norm{z}_{p-cvar} \leq C_{p} \norm{\phi}_{p-cvar} \left( 1 + \norm{x}_{p-var; [0, T]} + \norm{\mathbf{x}^2}_{\frac{p}{2}-var; [0, T]} \right).
\end{align}
\end{restatable}

The following propositions will provide us with various ways to construct
controlled rough paths from existing ones.
\begin{restatable} {proposition} {contSM} \label{controlledSmoothMap}
For $p \geq 2$, let
\begin{align*}
&y\in \mathcal{C}^{p-var}\left( [0,T]; \mathcal{U} \right), \\
&y' \in \mathcal{C}^{p-var} \left( [0,T]; \mathcal{L} \left( \mathbb{R}^{d}; \mathcal{U} \right) \right),
\end{align*}
and let $\phi $ be a $\mathcal{C}_{b}^{2}$ map from $\mathcal{U}$ to $\mathcal{V}$. \par
Then $\phi (y)\in \mathcal{C}^{p-var}\left( [0,T]; \mathcal{V} \right) $ and $\nabla \phi (y)y^{\prime }\in \mathcal{C}^{p-var}\left( [0,T]; \mathcal{L}(\mathbb{R}^{d}; \mathcal{V} )\right) $. Furthermore, if $(y, y^{\prime })$ is controlled by $\mathbf{x} \in \mathcal{C}^{p-var}\left( [0,T]; G^2\left( \mathbb{R}^{d} \right) \right) $, then $(\phi (y),\nabla \phi (y)y^{\prime })$ is also controlled by $\mathbf{x}$ and we have
\begin{align} \label{csmEstimateA}
\norm{\phi (y)}_{p-var;[0,T]}, \norm{\nabla \phi (y) \, y'}_{p-var;[0,T]}
\leq \norm{\phi}_{\mathcal{C}_b^2} \norm{y}_{\mathcal{V}^p; [0, T]} \left( 1 + \norm{y'}_{\mathcal{V}^p; [0, T]} \right),
\end{align}
and
\begin{align} \label{csmEstimateB}
\left\Vert R^{\phi (y)}\right\Vert _{\frac{p}{2}-var;[0,T]}
\leq \norm{\phi}_{\mathcal{C}_b^2} \left( \left\Vert y\right\Vert_{p-var;[0,T]}^{2}+\left\Vert R^{y}\right\Vert _{\frac{p}{2}
-var;[0,T]}\right).
\end{align}
\end{restatable}

\begin{restatable} {proposition} {contLeibniz} \label{controlledLeibniz1}
(Leibniz rule) For $p \geq 2$, let
\begin{align*}
&\phi \in \mathcal{C}^{p-var} \left([0, T]; \mathcal{L}(\mathcal{U}; \mathcal{V}) \right), \\
&\phi^{\prime} \in \mathcal{C}^{p-var} \left([0, T]; \mathcal{L}(\mathbb{R}^d; \mathcal{L}(\mathcal{U}; \mathcal{V})) \right),
\end{align*}
and we assume that $\left(\phi, \phi^{\prime }\right)$ is controlled by $\mathbf{x} \in \mathcal{C}^{p-var} \left([0, T]; G^2\left( \mathbb{R}^d \right) \right)$.
\begin{enumerate} [(i)]
\item Let $\psi \in \mathcal{C}^{p-var}\left( [0,T]; \mathcal{U} \right) $, $\psi' \in \mathcal{C}^{p-var}\left( [0,T]; \mathcal{L}(\mathbb{R}^{d}; \mathcal{U} )\right) $, and suppose that $\left( \psi ,\psi ^{\prime }\right) $ is controlled by $\mathbf{x}$. Then the path $\phi \psi \in \mathcal{C}^{p-var}([0,T]; \mathcal{V})$ given by the composition of $\phi$ and $\psi$ is also controlled by $\mathbf{x}$, with
derivative process $(\phi \psi)^{\prime } = \phi ^{\prime }\psi +\phi \psi
^{\prime }$. In addition, we have the bound
\begin{align}  \label{leibnizBound1a}
\norm{\phi\psi}_{p-cvar} \leq 2 \norm{\phi}_{p-cvar} \norm{\psi}_{p-cvar}
\end{align}
\item Suppose that $\psi \in \mathcal{C}^{\frac{p}{2}-var}([0,T]; \mathcal{U})$. Then $\phi \psi \in \mathcal{C}^{p-var}([0,T]; \mathcal{V})$ is also controlled by $\mathbf{x}$, with derivative process $(\phi \psi )^{\prime }=\phi ^{\prime }\psi $. Moreover, we have the bound
\begin{align}  \label{leibnizBound1b}
\norm{\phi \psi}_{p-cvar}
\leq \norm{\phi}_{p-cvar} \norm{\psi}_{\mathcal{V}^{\frac{p}{2}}; [0, T]}.
\end{align}
\end{enumerate}
\end{restatable}

\begin{remark}
The second part of the proposition clearly holds true if the roles of $\phi$
and $\psi$ are reversed. Furthermore, it asserts that one can trade increased
regularity in place of a controlled rough path structure in $\psi$ (or $\phi$)
for the composition to remain a controlled rough path.
\end{remark}

The proofs of the preceding theorem and propositions are routine and hence
deferred to the Appendix.

\subsection{Rough differential equations}

Now consider the following equation
\begin{align}
\label{smoothRDE}\mathrm{d} \mathbf{y}(t) = V(t, y(t)) \, \mathrm{d}
\mathbf{x}(t), \quad y(0) = y_{0},
\end{align}
where $V \in\mathcal{C}^{\lfloor p \rfloor}_{b} \left(  \mathbb{R}
\times\mathbb{R}^{e}; \mathcal{L} \left(  \mathbb{R}^{d}; \mathbb{R}^{e}
\right)  \right) $ is a differentiable function with bounded derivatives up to
degree $\lfloor p \rfloor$. Given $\mathbf{x} \in\mathcal{C}^{\infty} \left(
[0, T]; G^{\lfloor p \rfloor} \left( \mathbb{R}^{d}\right)  \right) $, the
unique solution $\mathbf{y} = S_{\lfloor p \rfloor} \left(  y \right) $ can be
obtained simply by solving \eqref{smoothRDE} as a regular ODE. Furthermore, we
have the following theorem (see \cite{lyons98}). 

\begin{theorem}
\label{univThm}  (Universal Limit Theorem)

The Ito map $\mathcal{I}: \mathbf{x} \mapsto\mathbf{y}$ is continuous from
$\mathcal{C}^{\infty} \left(  [0, T], G^{\lfloor p \rfloor} \left(
\mathbb{R}^{d} \right)  \right) $ to itself with respect to the $p$-variation
topology and thus admits a unique extension to the space of all $p$-geometric
rough paths $\mathcal{C}^{0, p-var} \left(  [0, T], G^{\lfloor p \rfloor}
\left(  \mathbb{R}^{d} \right)  \right) $. 
\end{theorem}

The Universal Limit Theorem allows one to transfer geometric results in the
smooth case to geometric rough paths, i.e. rough paths that satisfy the
change-of-variable rule. This effectively allows a generalization of the
Stratonovich integral to processes with higher $p$-variation.

We will mainly be considering RDEs with time-homogeneous vector fields driven
by Gaussian geometric rough paths. Furthermore, although the RDE
\begin{align}
\label{RDE}\mathrm{d} Y_{t} = V(Y_{t}) \circ\mathrm{d} \mathbf{X}_{t}, \quad
Y_{0} = y_{0},
\end{align}
outputs a full rough path $\mathbf{Y}_{t}$, we will be concerned only with the
first level/path-level solution, which satisfies
\begin{align*}
Y_{t} = y_{0} + \int_{0}^{t} V(Y_{s}) \circ\mathrm{d} \mathbf{X}_{s},
\end{align*}
and ignore the higher iterated integral terms. We will use the following
notation; writing $V(Y)$ as the co-tensor
\begin{align*}
\sum_{i=1}^{e} \sum_{j=1}^{d} V_{j}^{(i)}(Y) \, \mathrm{d}e_{i} \otimes
\mathrm{d}e_{j} \in\mathbb{R}^{e} \otimes\mathbb{R}^{d},
\end{align*}
we will denote
\begin{align*}
& V^{2}(Y) := \sum_{i,m =1}^{e} \sum_{j, k = 1}^{d} \frac{\partial V^{(i)}%
_{k}}{\partial e_{m}}(Y) V_{j}^{(m)}(Y) \, \mathrm{d}e_{i} \otimes
\mathrm{d}e_{j} \otimes\mathrm{d}e_{k} \in\mathbb{R}^{e} \otimes\mathbb{R}^{d}
\otimes\mathbb{R}^{d}.
\end{align*}

\begin{theorem}
\label{YBound} For all $s < t \in[0, T]$, $\left\|  Y\right\| _{p-var; [s,
t]}$ is in $L^{q}(\Omega)$ for all $q > 0$.
\end{theorem}

\begin{proof}
From equation 10.15 in \cite{fv2010b}, we have
\begin{align*}
\norm{Y}_{p-var; [s, t]}
&\leq C_p \left( \norm{V}_{C^{\lfloor p \rfloor}_b} \norm{\mathbf{X}}_{p-var; [s, t]} \vee \norm{V}^p_{C^{\lfloor p \rfloor}_b} \norm{\mathbf{X}}^p_{p-var; [s, t]} \right),
\end{align*}
and $\norm{\mathbf{X}}_{p-var; [s, t]}$ has moments of all orders; see Corollary 66 in \cite{fv2010a}.
\end{proof}

We will now show that
\begin{align*}
Y_{t} \in\mathbb{D}^{\infty}\left(  \mathbb{R}^{e}\right)  :=\bigcap_{p>1}
\bigcap_{k=1}^{\infty}\mathbb{D}^{k,p}\left(  \mathbb{R}^{e}\right)  ,\text{
for }t\geq0,
\end{align*}
where $Y$ solves RDE \eqref{rdeMain} with smooth vector fields. To do so, we
make use of the fact (cf. \cite{cfv2009}) that there exists a measurable
subset $\tilde{\Omega}\subset\Omega$ with $\mathbb{P(}\tilde{\Omega}) = 1$
such that for all $\omega\in\tilde{\Omega}$ we have the identity
\begin{align*}
\mathbf{X}\left(  \omega+ \Phi(h) \right)  =T_{\Phi(h)} \mathbf{X}\left(
\omega\right)  \quad\forall h \in\mathcal{H}_{1}^{d},
\end{align*}
where $T_{\Phi(h)} \mathbf{X}$ denotes the rough path translation of
$\mathbf{X}$ by $\Phi(h)$ (see \cite{cf2011}), which is well-defined via
Young-Stieltjes integration due to complementary regularity.

We then obtain
\begin{align*}
Y_{t}\left(  \omega+ s \, \Phi(h) \right)  :=Y_{t}^{\mathbf{X}\left(  \omega+
s \Phi(h) \right)  }=Y_{t}^{T_{s\Phi(h)}\mathbf{X}\left(  \omega\right)  },
\end{align*}
which is smooth in $s$ (see Theorem 11.6 in \cite{fv2010b}) and hence locally
absolutely continuous. It follows that $Y_{t}$ is RAC; indeed, in this case we
can even take the version $\tilde{F}_{\Phi(h)}$ in \eqref{version} to be
independent of $\Phi(h)$. Using Theorem \ref{sugitaThm}, it is immediate from
the definition of $\mathcal{D}_{\text{KS}}Y_{t}$ and the directional
derivatives that
\begin{align*}
\mathcal{D}_{\text{KS}}Y_{t}\left(  \omega\right)  \left(  \Phi(h) \right)  =
\mathcal{D}_{h} Y_{t}\left(  \omega\right)  ,\text{ }\mathbb{P}\text{-a.s.},
\end{align*}
for all $h\in$ $\mathcal{H}_{1}^{d}$, and henceforth we will use the latter
notation exclusively.

Moving on to the higher order derivatives, given $h_{1}, \ldots, h_{n}
\in\mathcal{H}_{1}^{d}$, we can take the directional derivatives of $Y_{t}$ in
the directions $\Phi(h_{1}), \ldots, \Phi(h_{n})$ in $\mathcal{H}^{d}$ by
setting
\begin{align}
\label{nderiv}\mathcal{D}_{h_{1}, \ldots, h_{n}}^{n} Y_{t} := \frac
{\partial^{n}}{\partial\varepsilon_{1} \ldots\partial\varepsilon_{n}}
Y_{t}^{\varepsilon_{1}, \ldots, \varepsilon_{n}} \bigg\vert_{\varepsilon_{1} =
\ldots= \varepsilon_{n}=0},
\end{align}
where $Y_{t}^{\varepsilon_{1}, \ldots, \varepsilon_{n}}$ solves
\begin{align*}
\mathrm{d} Y^{\varepsilon_{1}, \ldots, \varepsilon_{n}}_{t} = V(Y^{\varepsilon
_{1}, \ldots, \varepsilon_{n}}_{t}) \circ\mathrm{d} \left(  T_{\varepsilon_{1}
\Phi(h_{1}) + \cdots+ \varepsilon_{n} \Phi(h_{n})} \mathbf{X} \right) _{t},
\quad Y^{\varepsilon_{1}, \ldots, \varepsilon_{n}}_{0} = y_{0}.
\end{align*}
The path (\ref{nderiv}) again has finite $p$-variation and in Section
\ref{HOMD}, we will give it an explicit expression in terms of a sum of rough
integrals and/or Young-Stieltjes integrals when $n \geq2$. It only remains to
show that these derivatives are Hilbert-Schmidt operators with norms having
moments of all orders, and this has been proved in \cite{inahama2014}.

When $n=1$ the first-order derivative is given by (cf. \cite{fv2010b},
\cite{cf2011})
\begin{align}
\label{dd}\mathcal{D}_{h} Y_{t} = \int_{0}^{t} J_{t}^{\mathbf{X}} \left(
J_{s}^{\mathbf{X}}\right) ^{-1} V \left(  Y_{s} \right)  \, \mathrm{d}
\Phi(h)(s).
\end{align}
Here $J^{\mathbf{X}}_{t}$ denotes the Jacobian of the flow map $y_{0}
\rightarrow Y_{t}$ and satisfies
\begin{align}
\label{jacobianRDE}\mathrm{d} J^{\mathbf{X}}_{t} = \nabla V (Y_{t}) \left(
\circ\mathrm{d} \mathbf{X}_{t} \right)  J^{\mathbf{X}}_{t}, \quad
J^{\mathbf{X}}_{0} = \mathcal{I}_{e}.
\end{align}
On occasion, we will use the shorthand
\begin{align*}
J^{\mathbf{X}}_{t \leftarrow s} := J^{\mathbf{X}}_{t} \left( J^{\mathbf{X}%
}_{s}\right) ^{-1}, \quad0 \leq s < t \leq T,
\end{align*}
and for future reference, we also note that its inverse $\left(
J^{\mathbf{X}} \right) ^{-1}$ satisfies
\begin{align}
\label{jacobianInverseRDE}\mathrm{d} \left(  J^{\mathbf{X}}_{t} \right) ^{-1}
= -\left( J^{\mathbf{X}}_{t} \right) ^{-1} \nabla V (Y_{t}) \left(
\circ\mathrm{d} \mathbf{X}_{t} \right) , \quad\left( J^{\mathbf{X}}_{0}\right)
^{-1} = \mathcal{I}_{e}.
\end{align}

To bound the Jacobian, we will need the following definitions. Following
\cite{cll2013} we define, for a given interval $[s,t]\subset\lbrack0,T]$ and
$\beta>0$, the so-called greedy sequence $\{\tau_{i}(\beta)\}$, a finite
increasing sequence given by
\begin{align*}
&  \tau_{0}(\beta)=s,\\
&  \tau_{i+1}(\beta)=\inf\left\{  u\in(\tau_{i},t] \: \big\vert \: \left\Vert
\mathbf{X}\right\Vert _{p-var;[\tau_{i},u]}^{p}\geq\beta\right\}  \wedge t.
\end{align*}
We then denote
\begin{align}
\label{NDefn}N^{\mathbf{X}}_{\beta;[s,t]}:=\sup\left\{  n\in\mathbb{N}%
\cup\{0\} \: \big\vert \: \tau_{n}(\beta)<t\right\} ,
\end{align}
and note the following theorem.

\begin{theorem}
Let $X$ be an $\mathbb{R}^{d}$-valued centered Gaussian process with i.i.d.
components. For $1 \leq p < 4$, assume that $X$ has a natural lift to
$\mathbf{X} \in\mathcal{C}^{0, p-var} \left( [0, T]; G^{\lfloor p \rfloor}
\left( \mathbb{R}^{d}\right) \right) $, and that $\mathcal{H}^{d}
\hookrightarrow\mathcal{C}^{q-var} \left(  [0, T]; \mathbb{R}^{d} \right) $,
where $\frac{1}{p} + \frac{1}{q} > 1$. Then we have
\begin{align*}
\mathbb{P} \left[  N^{\mathbf{X}}_{\beta; [0, T]} > n \right]  \leq C_{1}
\exp\left(  - C_{2} \beta^{2} n^{\frac{2}{q}} \right) .
\end{align*}

\end{theorem}

\begin{proof}
See the proof of Theorem 6.3 in \cite{cll2013}.
\end{proof}

\begin{theorem}
\label{JBoundThm} For all $s < t \in[0, T]$, $\left\|  J^{\mathbf{X}}\right\|
_{p-var; [s, t]}$ is in $L^{q}(\Omega)$ for all $q > 0$.
\end{theorem}

\begin{proof}
Using the fact that $N_{1; [s, t]}^{\mathbf{X}}$ has Gaussian tails from the previous theorem, we see that $\exptn{\exp \left( C_2 q N_{1; [s, t]}^{\mathbf{X}} \right)} < \infty$ for all $q > 0$, $s < t \in [0, T]$. Now from equation (4.10) in \cite{cll2013}, we have the bound
\begin{align} \label{JBound}
\norm{J^{\mathbf{X}}}_{p-var; [s, t]} \leq C_1 \norm{\mathbf{X}}_{p-var; [s, t]} \exp \left( C_2 N^{\mathbf{X}}_{1; [s, t]} \right).
\end{align}
The statement of the theorem then follows immediately using Cauchy-Schwarz since $\norm{\mathbf{X}}_{p-var; [s, t]}$ also has moments of all orders.
\end{proof}

\section{High-order directional derivatives for solutions to RDEs}

We first begin with the following theorem.

\begin{theorem}
\label{controlledRDE} Consider the system of RDEs
\begin{align*}
& \mathrm{d} y_{t} = V(y_{t}) \circ\mathrm{d} \mathbf{x}_{t}, \quad y_{0} = a
\in\mathbb{R}^{e},\\
& \mathrm{d} J^{\mathbf{x}}_{t} = \nabla V (y_{t}) \left(  \circ\mathrm{d}
\mathbf{x}_{t} \right)  J^{\mathbf{x}}_{t}, \quad J^{\mathbf{x}}_{0} =
\mathcal{I}_{e},
\end{align*}
where $\mathbf{x} = \left(  1, x,\mathbf{x}^{2}\right)  \in\mathcal{C}^{p-var}
\left( [0, T]; G^{2} \left( \mathbb{R}^{d}\right) \right) $, $2 \leq p <3$,
and $V$ is in $\mathcal{C}^{3}_{b} (\mathbb{R}^{e}; \mathbb{R}^{e}
\otimes\mathbb{R}^{d})$.

In this case, both $(y,V(y))$ and $\left(  J^{\mathbf{x}}, \left(
J^{\mathbf{x}} \right) ^{\prime}\right) $ are controlled by $\mathbf{x}$.
Moreover, we have the bounds
\begin{align}
\label{boundedVF1}\left\|  y\right\| _{p-cvar} \leq C_{p} \left(  1 +
\left\Vert V\right\Vert _{\mathcal{C}_{b}^{2}} \right) ^{4} \left(  1 +
\left\Vert \mathbf{x}\right\Vert _{p-var;[0,T]} \right) ^{3},
\end{align}
and
\begin{align}
\label{linearVF1}\left\|  J^{\mathbf{x}}\right\| _{p-cvar} \leq C_{1} \left(
1 + \exp\left(  C_{2} N^{\mathbf{x}}_{1; [0, T]} \right)  \right) ^{4} \left(
1 + \left\Vert \mathbf{x}\right\Vert _{p-var;[0,T]} \right) ^{3},
\end{align}
where $C_{1}$, $C_{2}$ depend on $p$ and $\left\|  V\right\| _{\mathcal{C}%
^{3}_{b}}$.
\end{theorem}

\begin{proof}
(i) From Corollary 10.15 in \cite{fv2010b}, for $\gamma > p$ and $s,t \in [0, T]$, we have
\begin{align*}
\abs{y_{s,t} - V(y_s) x_{s,t} - V^2(y_s) \mathbf{x}^2_{s,t}}
\leq C_p \left( \norm{V}_{\mathcal{C}^2_b} \norm{\mathbf{x}}_{p-var;[s, t]} \right)^{\gamma}.
\end{align*}
This implies that
\begin{align} \label{REst1}
\begin{split}
\abs{R^y_{s,t}}^{\frac{p}{2}}
&\leq C_p \left( \abs{V^2(y_s) \mathbf{x}^2_{s,t}}^{\frac{p}{2}} + \left( \norm{V}_{\mathcal{C}^2_b} \norm{\mathbf{x}}_{p-var; [s,t]} \right)^{\frac{\gamma p}{2}} \right) \\
&\leq C_p \left( \norm{V}_{\mathcal{C}^2_b}^p \norm{\mathbf{x}}_{p-var;[s,t]}^p + \norm{V}_{\mathcal{C}^2_b}^{\frac{\gamma p}{2}} \norm{\mathbf{x}}_{p-var;[s,t]}^{\frac{\gamma p}{2}} \right),
\end{split}
\end{align}
and thus
\begin{align*}
\norm{R^y}_{\frac{p}{2}-var; [0, T]} \leq C_p \left( \norm{V}^2_{\mathcal{C}^2_b} \norm{\mathbf{x}}_{p-var; [0, T]}^2 \vee \norm{V}_{\mathcal{C}^2_b}^{\gamma} \norm{\mathbf{x}}_{p-var; [0, T]}^{\gamma} \right),
\end{align*}
from the super-additivity of the right side of \eqref{REst1}.
We will choose $\gamma$ to be in the interval $(p, 3)$, and since
\begin{align*}
\norm{V(y)}_{p-var;[ 0, T]} \leq \norm{V}_{\mathcal{C}^2_b} \norm{y}_{p-var;[0, T]}
\end{align*}
and
\begin{align*}
\norm{y}_{p-var;[0, T]}
\leq C_p \left( \norm{V}_{\mathcal{C}^2_b} \norm{\mathbf{x}}_{p-var; [0, T]} \vee \norm{V}_{\mathcal{C}^2_b}^p \norm{\mathbf{x}}^p_{p-var; [0, T]} \right),
\end{align*}
we obtain \eqref{boundedVF1}. \par
From Proposition 5 in \cite{fr2013}, we have
\begin{align*}
\norm{J^{\mathbf{x}}}_{p-var; [0, T]} \leq \exp \left( C_{p, \norm{V}_{\mathcal{C}^3_b}} \left( N^{\mathbf{x}}_{1; [0, T]} + 1 \right) \right),
\end{align*}
which gives us
\begin{align} \label{Jsup}
\norm{J^{\mathbf{x}}}_{\infty} \leq 1 + \exp \left( C_{p, \norm{V}_{\mathcal{C}^3_b}} \left( N^{\mathbf{x}}_{1; [0, T]} + 1 \right) \right) =: 1 + M.
\end{align}
For each $i = 1, \ldots, d$, we can construct $U_i \in \mathcal{C}^3_b(\mathbb{R}^e \times \mathbb{R}^{e^2}; \mathbb{R}^e \otimes \mathbb{R}^d)$ which is equal to the vector field $(y, z) \mapsto \nabla V_i (y) z$ on the set \\
$\mathcal{W}_1 = \left\{ z \in \mathbb{R}^{e^2} \; \big\vert \; \abs{z} \leq M + 1 \right\}$ and vanishes outside the set \\
$\mathcal{W}_2 = \left\{ z \in \mathbb{R}^{e^2} \; \big\vert \; \abs{z} < M + 2 \right\}$. Hence we have
\begin{align*}
\norm{U_i}_{\mathcal{C}^3_b}
&\leq \sup_{z \in \mathcal{W}_2} \norm{\nabla V_i (\cdot) z}_{\infty} + \norm{\nabla V_i (\cdot)}_{\infty} \\
&= \norm{V}_{\mathcal{C}^3_b} (M + 3), \quad i = 1, \ldots, d.
\end{align*}
Then the solution to
\begin{align*}
&\mathrm{d} y_t = V(y_t) \strato{\mathbf{x}_t}, \quad y_0 =  a \in \mathbb{R}^e, \\
&\mathrm{d} J^{\mathbf{x}}_t = U (y_t, J^{\mathbf{x}}_t) \strato{\mathbf{x}_t}, \quad J^{\mathbf{x}}_0 = \mathcal{I}_e,
\end{align*}
where $U = (U_1, \ldots, U_d)$, will be the same as the solution to the original system on $\mathbb{R}^e \times \mathcal{W}_1$, and it can be rewritten as
\begin{align*}
\mathrm{d} \tilde{y}_t = \tilde{V} (\tilde{y}_t) \strato{\mathbf{x}_t}, \quad \tilde{y}_0 = (a, \mathcal{I}_e),
\end{align*}
where $\tilde{y} = \left( y, J^{\mathbf{x}} \right) \in \mathbb{R}^e \times \mathbb{R}^{e^2}$ and $\norm{\tilde{V}}_{\mathcal{C}^3_b} \leq \norm{V}_{\mathcal{C}^3_b} (M + 3)$. \par
Hence, we can apply \eqref{boundedVF1} to obtain
\begin{align*}
\norm{\tilde{y}}_{p-cvar}
\leq C_p \left( 1 + \norm{V}_{\mathcal{C}^3_b} (M + 3) \right)^4 \left( 1 + \left\Vert \mathbf{x}\right\Vert
_{p-var;[0,T]} \right)^3,
\end{align*}
and since $J^{\mathbf{x}}$ is a component of $\tilde{y}$, we obtain \eqref{linearVF1}.
\end{proof}

\subsection{Upper bounds on the high-order directional derivatives} \label{HOMD} 
We now use the preceding theorem as well as results on controlled
rough paths from Section \ref{controlledRP} to obtain upper bounds on the
directional derivative
\begin{align} \label{nderiv2}
\mathcal{D}_{g_{1},\ldots,g_{n}}^{n}y_{t}:=\frac{\partial^{n}}{\partial
\varepsilon_{1}\ldots\partial\varepsilon_{n}}y_{t}^{\varepsilon_{1}%
,\ldots\varepsilon_{n}}\bigg\vert_{\varepsilon_{1}=\ldots=\varepsilon_{n}=0},
\end{align}
where the driving rough path $\mathbf{x}$ is perturbed in the directions
$g_{1}, \ldots, g_{n}$ which are now assumed to be paths having complementary
regularity with $\mathbf{x}$. 
To condense the notation we will write $\mathcal{D}_{A}^{\left\vert A\right\vert }y_{t}$, for any subset $A$ of
$\left\{  g_{1},\ldots,g_{n}\right\} $, noting that the symmetry of the
derivative ensures this is well-defined. For $i\in\left\{  1,...,n\right\}$
we then let $A_{i}^{n}\left(  \cdot\right)  :\left[  0,T\right]
\rightarrow\left(  \mathbb{R}^{e}\right)  ^{\otimes i}$ be defined by
\begin{align} \label{A}
A_{i}^{n}\left(  t\right)  :=\sum_{\pi=\left\{  \pi_{1},\ldots,\pi
_{i}\right\}  \in\mathcal{P}\left(  \left\{  g_{1},\ldots,g_{n}\right\}
\right)  }\mathcal{D}_{\pi_{1}}^{\left\vert \pi_{1}\right\vert }y_{t}%
\tilde{\otimes}\cdots\tilde{\otimes}\mathcal{D}_{\pi_{i}}^{\left\vert \pi
_{i}\right\vert }y_{t},\,t\in\left[  0,T\right].
\end{align}
Here $\tilde{\otimes}$ denotes the symmetric tensor product, and the summation
is over the set of all partitions of $\left\{  g_{1},\ldots,g_{n}\right\}  $
containing exactly $i$ elements. For all $i\in\left\{  1,\ldots,n-1\right\}  $
and $j\in\left\{  1,\ldots,n\right\}  $ we also let $B_{i,j}^{n}\left(
\cdot\right)  :\left[  0,T\right]  \rightarrow\left(  \mathbb{R}^{e}\right)
^{\otimes i}$ be defined by
\begin{align} \label{Bb}
B_{i,j}^{n}\left(  t\right)  :=\sum_{\pi=\left\{  \pi_{1},\ldots,\pi
_{i}\right\}  \in\mathcal{P}\left(  \left\{  g_{1},\ldots,g_{j-1}%
,g_{j+1},\ldots,g_{n}\right\}  \right)  }\mathcal{D}_{\pi_{1}}^{\left\vert
\pi_{1}\right\vert }y_{t}\tilde{\otimes}\cdots\tilde{\otimes}\mathcal{D}%
_{\pi_{i}}^{\left\vert \pi_{i}\right\vert }y_{t}.
\end{align}
The following result gives an integral equation for the formula for
$\mathcal{D}_{g_{1},\ldots,g_{n}}^{n}y_{t}$ in terms of these paths (cf.
\cite{hp2013}, \cite{inahama2014}, \cite{chlt2015}).

\begin{theorem}
\label{mdFormula} Let $p\geq1$ and $q\geq1$ be such that $1/p+1/q>1,$ and let
$n\in\mathbb{N}$ such that $n\geq2$ . Assume $\mathbf{x}\in\mathcal{C}%
^{0,p-var}\left(  [0,T];G^{\lfloor p\rfloor}\left(  \mathbb{R}^{d}\right)
\right) $ and suppose $y$ is the path-level solution to the RDE
\begin{align}
\mathrm{d}y_{t}=V\left(  y_{t}\right)  \circ\mathrm{d}\mathbf{x}_{t},\quad
y_{0}\in\mathbb{R}^{e}\text{ given},
\end{align}
where $V\in\mathcal{C}_{b}^{\lfloor p\rfloor+n}\left(  \mathbb{R}%
^{e};\mathbb{R}^{e}\otimes\mathbb{R}^{d}\right) $. Suppose that $g_{1}%
,\ldots,g_{n}\in C^{q-var}([0,T];\mathbb{R}^{d})$. Then the $n^{\text{th}}%
$-order directional derivative \eqref{nderiv2} satisfies the RDE
\begin{align}
\label{df}%
\begin{split}
&  \mathrm{d}\mathcal{D}_{g_{1},\ldots,g_{n}}^{n}y_{t}=\sum_{i=1}^{n}%
\nabla^{i}V\left(  y_{t}\right)  A_{i}^{n}\left(  t\right)  \circ
\mathrm{d}\mathbf{x}_{t}+\sum_{i=1}^{n-1}\sum_{j=1}^{n}\nabla^{i}V\left(
y_{t}\right)  B_{i,j}^{n}\left(  t\right)  \,\mathrm{d}g_{j}(t),\\
&  \mathcal{D}_{g_{1},\ldots,g_{n}}^{n}y_{0}=0,
\end{split}
\end{align}
where $A_{i}^{n}$ and $B_{i,j}^{n}$ are respectively defined by \eqref{A} and \eqref{Bb}.
\end{theorem}

\begin{remark}
The symmetry of the higher order derivatives of $V$ ensures that we may
simplify $\nabla^{i}V\left(  y_{t}\right)  A_{i}^{n}\left(  t\right)  $ in
\eqref{A} by replacing the symmetric tensor product with the usual tensor
product to give
\begin{align*}
\nabla^{i}V\left(  y_{t}\right)  A_{i}^{n}\left(  t\right)  =\sum
_{\pi=\left\{  \pi_{1},\ldots,\pi_{i}\right\}  \in\mathcal{P}\left(  \left\{
g_{1},\ldots,g_{n}\right\}  \right)  }\nabla^{i}V\left(  y_{t}\right)
\mathcal{D}_{\pi_{1}}^{\left\vert \pi_{1}\right\vert }y_{t}\otimes
\cdots\otimes\mathcal{D}_{\pi_{i}}^{\left\vert \pi_{i}\right\vert }y_{t}.
\end{align*}
The terms $\nabla^{i}V\left(  y_{t}\right)  B_{i,j}^{n}\left(  t\right)  $ may
also be simplified similarly. For this reason it is sufficient to prove
\eqref{df} for paths $A_{i}^{n}$ and $B_{i,j}^{n}$ whose symmetrizations
coincide with the right sides of \eqref{A} and \eqref{Bb} respectively.
\end{remark}

\begin{proof}
We begin with the case $n=2$. Taking the directional derivative of $\mathcal{D}_{g_1} y_t$ (see \eqref{dd}) in the direction of $g_2$, we see that $\mathcal{D}^2_{g_1, g_2} y_t$ solves the RDE
\begin{align}
\begin{split}
\mathrm{d} \mathcal{D}^2_{g_1, g_2} y_t &= \nabla V(y_t) \left( \mathcal{D}%
^2_{g_1, g_2} y_t \right) \circ \mathrm{d} \mathbf{x}_t + \nabla^2 V (y_t)
\left( \mathcal{D}_{g_1} y_t, \mathcal{D}_{g_2} y_t \right) \circ \mathrm{d}
\mathbf{x}_t \\
&\qquad + \nabla V (y_t) \left( \mathcal{D}_{g_2} y_t \right) \, \mathrm{d}
g_1(t) + \nabla V (y_t) \left( \mathcal{D}_{g_1} y_t \right) \,
\mathrm{d} g_2(t).
\end{split}
\end{align}
The proof is finished by induction. Assuming \eqref{df} is true for $n=2, \ldots, k$, one can take the directional derivative of $\mathcal{D}_{g_1, \ldots, g_k} y_t$ in the direction $g_{k+1}$ to obtain the identity
\begin{align*}
\mathcal{D}_{g_1, \ldots,g_{k+1}}^{k+1} y_t
&= \sum_{i=1}^k \mathcal{D}_{g_{k+1}} \int_0^t \nabla^i V \left( y_s \right) A_{i}^{k}\left( s\right) \strato{\mathbf{x}_s} \\
&\fqquad+ \sum_{i=1}^{k-1} \sum_{j=1}^{k} \mathcal{D}_{g_{k+1}} \int_{0}^{t} \nabla^{i} V\left( y_{s} \right) B_{i,j}^{k}\left( s\right) \wrt{g_j(s)} \\
&= \sum_{i=1}^{k+1} \int_{0}^{t} \nabla^{i} V\left( y_{s} \right) \tilde{A}_{i}^{k+1}\left( s\right) \strato{\mathbf{x}_s} \\
&\fqquad+ \sum_{i=1}^{k}\sum_{j=1}^{k+1} \int_{0}^{t} \nabla^{i} V\left( y_{s} \right) \tilde{B}_{i,j}^{k+1}\left( s \right) \wrt{g_j (s)},
\end{align*}
where the coefficients $\tilde{A}_{i}^{k+1}$and $\tilde{B}_{i}^{k+1}$ are the $\left( \mathbb{R}^{e}\right) ^{\otimes i}$-valued paths defined for $t\in \left[ 0,T\right]$ by
\begin{align} \label{r1}
\tilde{A}_{i}^{k+1}\left( t\right) =
\begin{cases}
\; \mathcal{D}_{g_{k+1}} A_{1}^{k} \left( t\right), & i=1, \\
\; \mathcal{D}_{g_{k+1}} A_{i}^{k} \left( t\right) + A_{i-1}^{k}\left( t\right) \otimes \mathcal{D}_{g_{k+1}} y_{t}, & i=2, \ldots, k, \\
\; A_{k}^{k}\left( t\right) \otimes \mathcal{D}_{g_{k+1}} y_{t}, & i=k+1,
\end{cases}
\end{align}
and
\begin{align} \label{r2}
\tilde{B}_{i,j}^{k+1}\left( t\right) =
\begin{cases}
\; \mathcal{D}_{g_{k+1}} B_{1,j}^{k}\left( t\right), & i=1, \, j=1, \ldots, k, \\
\; \mathcal{D}_{g_{k+1}}B_{i,j}^{k}\left( t\right) + B_{i-1,j}^{k}\left( t \right) \otimes \mathcal{D}_{g_{k+1}} y_{t}, & i=2, \ldots, k-1, \, j = 1, \ldots, k, \\
\; B_{i-1,j}^{k}\left( t \right) \otimes \mathcal{D}_{g_{k+1}} y_{t}, & i=k, \, j = 1, \ldots, k, \\
\; A_{i}^{k} \left( t \right), & i=1, \ldots, k, \, j = k+1.
\end{cases}
\end{align}
To finish the inductive step we first show that for every $t\in \left[ 0,T \right]$,
\begin{align} \label{a1}
\tilde{A}_{i}^{k+1}\left( t\right) \tilde{=} A_{i}^{k+1}\left( t\right) \, \forall i = 1, \ldots, k+1,
\end{align}
where $a\tilde{=}b$ means that the symmetrizations of the tensors $a$ and $b$ are equal. From this it immediately follows that $\nabla^{i} V\left( y_{t} \right) \tilde{A}_{i}^{k+1}\left( t\right) = \nabla^{i} V\left( y_{t} \right) A_{i}^{k+1}\left( t\right) $ for all $i=1,...,k+1.$ We check \eqref{a1} for the boundary cases first. For $i=1$ the induction hypothesis gives at once that
\begin{align*}
\tilde{A}_{1}^{k+1}\left( t\right) = \mathcal{D}_{g_1, \ldots, g_{k+1}}^{k+1} y_{t},
\end{align*}
whereas the case $i=k+1$ follows from
\begin{align*}
\tilde{A}_{k+1}^{k+1}\left( t\right)
&= A_{k}^{k}\left( t\right) \otimes \mathcal{D}_{g_{k+1}} y_{t} \\
&\; \tilde{=} \, A_{k}^{k}\left( t\right) \tilde{\otimes}\mathcal{D}_{g_{k+1}} y_{t} \\
&= \mathcal{D}_{g_1} y_{t} \tilde{\otimes} \cdots \tilde{\otimes}\mathcal{D}_{g_{k+1}} y_{t}
= A_{k+1}^{k+1}\left( t\right) .
\end{align*}
For the remaining cases $i=2, \ldots, k$ we note that any partition of $\left\{ g_1, \ldots, g_{k+1} \right\} $ of size $i$ can be formed from a partition $\pi$ of $\left\{ g_1, \ldots, g_k \right\}$ in one of two ways. The first way is that $\pi = \left\{ \pi_1, \ldots, \pi_i \right\}$ itself has size $i$ and $g_k$ is assigned to one of the subsets $\pi_1, \ldots, \pi_i$. The second way is that $\pi = \left\{ \pi_1, \ldots, \pi_{i-1} \right\}$ has size $i-1$ and $\left\{g_{k+1} \right\}$ is adjoined as a singleton to give $\left\{ \pi_1, \ldots, \pi_{i-1}, \left\{ g_{k+1}\right\} \right\}$. The two terms in \eqref{r1} obtained by differentiation and the tensor product respectively correspond to these operations. By the induction hypothesis, $A_{i}^{k}$ (resp $A_{i-1}^{k})$ includes a summation over all partitions of $\left\{ g_1, \ldots, g_n \right\} $ of size $i$ (resp. $i-1$), thus every partition of $\left\{ g_1, \ldots, g_{k+1} \right\}$ of size $i$ is accounted for in \eqref{r1}. It follows immediately that
\begin{align*}
\tilde{A}_{i}^{k+1}\left( t\right) \tilde{=} A_{i}^{k+1}\left( t\right) .
\end{align*}
Finally we show that for every $t\in \left[ 0,T\right] $,
\begin{align*}
\tilde{B}_{i,j}^{k+1}\left( t\right) \tilde{=} B_{i,j}^{k+1}\left( t\right) \, \forall i = 1, \ldots, k, \, j = 1, \ldots, k+1.
\end{align*}
Again we treat the boundary cases separately. For $j=k+1$, from the definition of $\tilde{B}$ and $A$, we have
\begin{align*}
\tilde{B}_{i,k+1}^{k+1}\left( t\right) =A_{i}^{k}\left( t\right)
=B_{i,k+1}^{k+1}\left( t\right), \, \forall i = 1, \ldots, k.
\end{align*}
For $i=1$ we have $\forall j = 1, \ldots, k$
\begin{align*}
\tilde{B}_{1,j}^{k+1}\left( t\right)
&= \mathcal{D}_{g_{k+1}} B_{1,j}^{k} \left( t\right) \\
&= \mathcal{D}_{g_{k+1}} \mathcal{D}^{k-1}_{g_1, \ldots, g_{j-1}, g_{j+1}, \ldots, g_{k}} y_{t} \\
&= \mathcal{D}^k_{g_1, \ldots, g_{j-1}, g_{j+1}, \ldots, g_{k+1}} y_{t} = B_{1,j}^{k+1} \left( t\right).
\end{align*}
The remaining terms are dealt with by exactly the same argument used for the non-boundary $\tilde{A}$ terms, and the induction is thereby complete.
\end{proof}The following corollary makes explicit the expression given in
Proposition 11.5 of \cite{fv2010b}.

\begin{corollary}
\label{cf} Under the conditions of the preceding theorem,
\begin{align}
\label{exp2}%
\begin{split}
\mathcal{D}_{g_{1},\ldots,g_{n}}^{n}y_{t}  &  =\sum_{i=2}^{n}\int_{0}^{t}
J_{t}^{\mathbf{x}}\left(  J_{s}^{\mathbf{x}}\right)  ^{-1}\nabla^{i}V\left(
y_{s}\right)  A_{i}^{n}\left(  s\right)  \circ\mathrm{d}\mathbf{x}_{s}\\
&  \quad\qquad+\sum_{i=1}^{n-1}\sum_{j=1}^{n}\int_{0}^{t}J_{t}^{\mathbf{x}%
}\left(  J_{s}^{\mathbf{x}}\right)  ^{-1}\nabla^{i}V\left(  y_{s}\right)
B_{i,j}^{n}\left(  s\right)  \,\mathrm{d}g_{j}(s)
\end{split}
\end{align}
for all $n\geq2$.
\end{corollary}

\begin{proof}
From formula \eqref{df} and the fact that
\begin{align*}
A_{1}^{n} \left( s \right) = \mathcal{D}_{g_1, \ldots, g_n}^n y_{s},
\end{align*}
\eqref{exp2} can be recovered using Duhamel's principle.
\end{proof}

We now arrive at the main result of this section.

\begin{proposition}
\label{dir der est} Let $p\in\lbrack2,3)$, $q:=\frac{p}{2}$ and $n\in
\mathbb{N}$. Let $y$ be the solution to the RDE
\begin{align*}
\mathrm{d}y_{t}=V\left(  y_{t}\right)  \circ\mathrm{d}\mathbf{x}_{t},\quad
y_{0}\in\mathbb{R}^{e}\text{ given},
\end{align*}
where $\mathbf{x}\in\mathcal{C}^{0,p-var}\left(  [0,T];G^{2}\left(
\mathbb{R}^{d}\right)  \right)  $ and $V\in\mathcal{C}_{b}^{2+n}
(\mathbb{R}^{e};\mathbb{R}^{e}\otimes\mathbb{R}^{d})$. Then there exists a
polynomial $P_{d(n)}:\mathbb{R}_{+}\times\mathbb{R}_{+} \rightarrow
\mathbb{R}_{+}$ of finite degree $d(n)$ for which
\begin{align}
\label{ubound}%
\begin{split}
&  \left\Vert \mathcal{D}_{g_{1},\ldots,g_{n}}^{n}y_{\cdot}\right\Vert
_{\mathcal{V}^{p};\left[  0,T\right]  }\\
&  \qquad\leq P_{d(n)}\left(  \left\Vert \mathbf{x}\right\Vert _{p-var;[0,T]}%
,\exp\left(  C\,N_{1;[0,T]}^{\mathbf{x}}\right)  \right)  \prod\limits_{i=1}%
^{n}\left\Vert g_{i}\right\Vert _{q-var;[0,T]},
\end{split}
\end{align}
for any $g_{1},\ldots,g_{n}\in C^{q-var}([0,T];\mathbb{R}^{d})$. Here
$N_{1}^{\mathbf{x}}$ is defined as in \eqref{NDefn}, and the constant $C$ as
well as the coefficients of $P_{d(n)}$ depend only on $\left\Vert V\right\Vert
_{\mathcal{C}_{b}^{2+n}}$ and $p$.
\end{proposition}

\begin{proof}
The proof will proceed by induction. First, we denote
\begin{align*}
F^i_t := \left( J^{\mathbf{x}}_t \right)^{-1} \nabla^i V (y_t), \quad i = 0, 1, \ldots.
\end{align*}
Applying Theorem \ref{controlledRDE} together with Proposition \ref{controlledSmoothMap} to $\nabla^i V(y)$, we see that there exists an integer $k_1$ such that
\begin{align} \label{nablaBound}
\norm{\nabla^i V(y)}_{p-cvar}
\leq C_1 \left( 1 + \norm{\mathbf{x}}_{p-var; [0, T]} \right)^{k_1},
\end{align}
(note from \eqref{gNorm} that $\norm{\mathbf{x}^k}_{\frac{p}{k}-var; [0, T]} \leq C \norm{\mathbf{x}}^k_{p-var; [0, T]}$ for all $k$) and again from Theorem \ref{controlledRDE}, we know that there exist $k_2$ and $l$ such that
\begin{align} \label{jacobianBound}
\norm{J^{\mathbf{x}}}_{p-cvar}, \, \norm{\left(J^{\mathbf{x}}\right)^{-1}}_{p-cvar}
\leq C_2 \left( 1 + \exp \left(C_3 N^{\mathbf{x}}_{1; [0, T]} \right) \right)^{k_2} \left( 1 + \norm{\mathbf{x}}_{p-var;[0,T]} \right)^l.
\end{align}
Now applying Leibniz rule, i.e. Proposition \ref{controlledLeibniz1}, we have, for some integer $k$,
\begin{align} \label{FBound}
\norm{F^i}_{p-cvar}
\leq C_1 \left( 1 + \exp \left( C_2 \, N^{\mathbf{x}}_{1; [0, T]} \right) \right)^k \left( 1 + \norm{\mathbf{x}}_{p-var;[0,T]} \right)^l,
\end{align}
where $C_1$ and $C_2$ depend only on $p$ and $\norm{V}_{\mathcal{C}^{2 + i}_b}$. \par
We now begin with the base case $n=1$. Let $\phi_t$ denote $J^{\mathbf{x}}_t$ and $\psi_t$ denote $\int_0^t \left( J_s^{\mathbf{x}} \right)^{-1} V(y_s) \wrt{g(s)}$. Then $\mathcal{D}_g y_t = \phi_t \psi_t$ and applying Young's inequality to $\psi_t$, we obtain
\begin{align}
\begin{split}
&\norm{\int_0^t \left( J_s^{\mathbf{x}} \right)^{-1} V(y_s) \wrt{g(s)}}_{q-var;[0,T]} \\
&\fqquad\leq C_p \norm{\left(J^{\mathbf{x}} \right)^{-1} V (y) }_{\mathcal{V}^p; [0, T]} \norm{g}_{q-var; [0, T]}.
\end{split}
\end{align}
Continuing, the second part of Proposition \ref{controlledLeibniz1} tells us that $\mathcal{D}_g y$ is controlled by $\mathbf{x}$, and from the bounds \eqref{FBound} and \eqref{jacobianBound}, we have
\begin{align} \label{DhyEst}
\norm{\mathcal{D}_g y_{\cdot}}_{p-cvar}
\leq C_3 \left( 1 + \exp \left( C \, N^{\mathbf{x}}_{1; [0, T]} \right) \right)^k \left( 1 + \norm{\mathbf{x}}_{p-var;[0,T]} \right)^l \norm{g}_{q-var; \left[ 0,T\right] }.
\end{align}
The cases $n \geq 2$ are proved in the same manner. Let $z^n_t$ denote
\begin{align*}
z^n_t := \mathcal{D}^n_{g_1, \ldots, g_n} y_t,
\end{align*}
where $\{g_i\}_{i=1}^n$ is an arbitrary subset of $C^{q-var}([0,T];\mathbb{R}^{d}) $, and the induction hypothesis is as follows: \par
Assume that for all $n = 1, \ldots, k$, $z^n$ is controlled by $\mathbf{x}$, and that
\begin{align*}
\norm{z^n}_{p-cvar}
\leq P_{d(n)} \left( \norm{\mathbf{x}}_{p-var; [0, T]}, \exp \left( C \, N^{\mathbf{x}}_{1; [0, T]} \right) \right) \prod\limits_{i=1}^{n} \norm{g_{i}}_{q-var; \left[ 0,T\right] }.
\end{align*}
To show the result for $n = k+1$, first recall from Theorem \ref{mdFormula} that $z^{k+1}_t = \mathcal{D}_{g_1, \ldots, g_{k+1}}^{k+1} y_{t}$ equals
\begin{align*}
J^{\mathbf{x}}_t \left( H^{k+1}_t + G_t^{k+1} \right),
\end{align*}
where
\begin{align*}
H^{k+1}_t
:= \int_0^t \sum_{i=2}^{k+1} F^i_s A^{k+1}_i (s) \strato{\mathbf{x}_s}
\end{align*}
and
\begin{align*}
G^{k+1}_t
:= \int_{0}^{t} \sum_{i=1}^{k} \sum_{j=1}^{k+1} F^i_s B_{i,j}^{k+1}\left( s\right) \wrt{g_j (s)}.
\end{align*}
From the induction hypothesis and Leibniz rule, for a partition \\
$\pi =\left\{ \pi_1, \ldots, \pi_i \right\}$ in $\mathcal{P}\left( \left\{ g_1, \ldots, g_{k+1} \right\} \right)$, we have the bound
\begin{align} \label{Abound}
\begin{split}
&\norm{ \mathcal{D}_{\pi_1}^{\abs{\pi_1}} y_{\cdot }\tilde{\otimes} \cdots \tilde{\otimes}\mathcal{D}_{\pi_{i}}^{\abs{\pi_i}} y_{\cdot } }_{p-cvar} \\
&\qquad\leq \prod\limits_{l=1}^{i} P_{d(\abs{ \pi_l}) } \left( \norm{\mathbf{x}}_{p-var; [0, T]}, \exp \left( C \, N^{\mathbf{x}}_{1; [0, T]} \right) \right) \prod\limits_{m=1}^{k+1} \norm{ g_{m} }_{q-var; \left[ 0,T\right] }.
\end{split}
\end{align}
Similarly, for a partition $\pi =\left\{ \pi_1, \ldots, \pi_i \right\} \in \mathcal{P}\left( \left\{ g_1, \ldots, g_{j-1}, g_{j+1}, \ldots, g_{k+1} \right\} \right)$ we have the bound
\begin{align} \label{Bbound}
\begin{split}
&\norm{ \mathcal{D}_{\pi_1}^{\abs{\pi_1}} y_{\cdot }\tilde{\otimes} \cdots \tilde{\otimes}\mathcal{D}_{\pi
_{i}}^{\abs{\pi_i}} y_{\cdot } }_{p-cvar} \\
&\qquad\leq \prod\limits_{l=1}^{i} P_{d(\abs{ \pi_l}) } \left( \norm{\mathbf{x}}_{p-var; [0, T]}, \exp \left( C \, N^{\mathbf{x}}_{1; [0, T]} \right) \right) \prod\limits_{m=1, m \neq j}^{k+1} \norm{ g_{m} }_{q-var; \left[ 0,T\right]}.
\end{split}
\end{align}
Recalling the definition of $A^{k+1}_i (s)$ in \eqref{A}, we use \eqref{Abound} together with bound \eqref{FBound} and apply Leibniz rule. After summing over $i$ and invoking Theorem \ref{controlledThm1}, we see that $H^{k+1}$ is controlled by $\mathbf{x}$, and there exists a polynomial $\tilde{P}_1$ such that
\begin{align*}
\norm{H^{k+1}}_{p-cvar} \leq \tilde{P}_1 \left( \norm{\mathbf{x}}_{p-var; [0, T]}, \exp \left( C \, N^{\mathbf{x}}_{1; [0, T]} \right) \right) \prod \limits_{i=1}^{k+1} \norm{g_{i}}_{q-var; \left[ 0,T\right] }.
\end{align*}
For $G^{k+1}$, we will show that its $q$-variation is bounded similarly, and then add it to the remainder term of $H^{k+1}$ to make $H^{k+1} + G^{k+1}$ a controlled rough path (with the appropriate bounds). Fixing $i$ and $j$, from Young's inequality we have
\begin{align*}
&\norm{ \int_{0}^{\cdot } F^i_s B_{i,j}^{k+1}\left( s\right) \wrt{g_j (s)} } _{q-var; [0, T]} \\
&\fqquad\leq C_p \norm{ F^i}_{\mathcal{V}^p; [0, T]} \norm{ B_{i,j}^{k+1} }_{\mathcal{V}^p; [0, T]} \norm{g_{j}}_{q-var}.
\end{align*}
Now if we recall the definition of $B^{k+1}_{i, j}(s)$ in \eqref{Bb} and use \eqref{Bbound} in the above expression, after summing over all $i$ and $j$, we obtain some polynomial $\tilde{P}_2$ such that
\begin{align} \label{niceTermBound}
\norm{G^{k+1}} _{q-var; [0, T]}
\leq \tilde{P}_2 \left( \norm{\mathbf{x}}_{p-var; [0, T]}, \exp \left( C \, N^{\mathbf{x}}_{1; [0, T]} \right) \right) \prod \limits_{i=1}^{k+1} \norm{g_{i}}_{q-var; \left[ 0,T\right] }.
\end{align}
Finally, applying Leibniz rule to $z^{k+1}_t = J^{\mathbf{x}}_t \left( H^{k+1}_t + G_t^{k+1} \right)$ shows us that $z^{k+1}$ is controlled by $\mathbf{x}$ with the bound
\begin{align*}
\norm{z^{k+1}}_{p-cvar} \leq P_{d(k+1)}\left( \norm{\mathbf{x}}_{p-var; [0, T]}, \exp \left( C \, N^{\mathbf{x}}_{1; [0, T]} \right) \right) \prod\limits_{i=1}^{k+1} \norm{g_{i}}_{q-var; \left[ 0,T\right] }.
\end{align*}
\end{proof}

\begin{remark}
Our main use of this result will be in the stochastic setting where
$\mathbf{X}$ is a Gaussian rough path zero mean continuous Gaussian process
with i.i.d. components, and a covariance function $R$ of finite 2D $\rho
$-variation for some $\rho\in\left[  1 ,\frac{3}{2} \right) $. Under these
assumptions, Proposition \ref{dir der est} can be applied with $g_{i}
:=\Phi\left(  h_{i} \right) $ for any collection $\left\{  h_{i} \right\}
\subset\mathcal{H}_{1}^{d}$. Note also in this section we abuse the notation
by using $\mathcal{D}_{g_{i}} F$ rather than $\mathbf{D}_{g_{i}} F$ (see
Remark \ref{mdRem}); in later sections, the subscript will be elements in
$\mathcal{H}_{1}^{d}$ rather than $\mathcal{H}^{d}$.
\end{remark}

\section{An isomorphism and dense subspace of the Cameron-Martin space}

In this section, we will identify a dense subspace of the Cameron-Martin space
which will be of importance later. The motivation is as follows: let $Y$ be a
solution to RDE \eqref{RDE}. We would like to show that $Y \in\mathbb{D}^{1,2}
(\mathcal{H}_{1}^{d})$, which in turn implies that $Y$ is Skorohod integrable.
To do so, consider a partition $\pi= \{ r_{i} \}$ of $[0, T]$, and observe
that
\begin{align*}
Y^{\pi} (t) := \sum_{i} Y_{r_{i}} \mathds{1}_{\left[  r_{i}, r_{i+1} \right) }
(t)
\end{align*}
is almost surely an element of $\mathcal{H}_{1}^{d}$. Using Ito-Skorohod
isometry, we have
\begin{align*}
\mathbb{E} \left[  \delta^{X}\left( Y^{\pi} - Y \right) ^{2}\right]
\leq\mathbb{E} \left[  \left\|  Y^{\pi} - Y \right\| ^{2}_{\mathcal{H}_{1}%
^{d}}\right]  + \mathbb{E} \left[  \left\|  \mathcal{D} Y^{\pi} - \mathcal{D}
Y \right\| ^{2}_{\mathcal{H}_{1}^{d} \otimes\mathcal{H}_{1}^{d}}\right] .
\end{align*}
Thus if we can show that almost surely, $\left\|  Y^{\pi} - Y\right\|
_{\mathcal{H}_{1}^{d}} $ and $\left\|  \mathcal{D}Y^{\pi} - \mathcal{D}%
Y\right\| _{\mathcal{H}_{1}^{d} \otimes\mathcal{H}_{1}^{d}}$ both vanish as
$\left\|  \pi\right\|  \rightarrow0$, then with further integrability
assumptions one can apply dominated convergence to show that $\delta^{X}
(Y^{\pi})$ converges to $\delta^{X} (Y)$ in $L^{2}(\Omega)$.

We will investigate the (almost sure) regularity required of $Y$ to identify
it as an element of $\mathcal{H}_{1}^{d}$ and to have $\left\|  Y^{\pi} -
Y\right\| _{\mathcal{H}_{1}^{d}} \rightarrow0$. We first note the following
lemma, which is a direct consequence of Theorem \ref{2Dintegral}.

\begin{lemma}
\label{keyProp1} Let $f \in\mathcal{C}_{pw}^{p-var} \left(  [0, T];\mathbb{R}%
^{e} \right) $ and $R$ be continuous and of finite $\rho$-variation where we
assume that $\frac{1}{p} + \frac{1}{\rho} > 1$. For any partition $\pi= \{
r_{i} \}$ of $[0, T]$, let $f^{\pi}$ denote
\begin{align*}
f^{\pi} (t) := \sum_{i} f (r_{i}) \mathds{1}_{[r_{i}, r_{i+1})} (t).
\end{align*}
Then
\begin{align}
\label{newlim}\lim_{\left\|  \pi\right\|  \rightarrow0} \left|  \int%
_{[0,T]^{2}} \left\langle f^{\pi}_{s} - f_{s}, \, f^{\pi}_{t} - f_{t}
\right\rangle _{\mathbb{R}^{e}} \, \mathrm{d} R(s,t) \right|  = 0.
\end{align}

\end{lemma}

\begin{proof}
Since $f \in \mathcal{C}_{pw}^{p-var} \left( [0, T];\mathbb{R}^e \right)$, we can partition $[0, T]^2$ into finite sub-regions $[x_i, x_{i+1}] \times [x_j, x_{j+1}]$, where $\left\langle f_s, f_t \right\rangle_{\mathbb{R}^e}$ is continuous in the interior of each such sub-region. We will show \eqref{newlim} in each sub-region, and the case for $[0, T]^2$ will follow from summing the Young integral over all the sub-regions. \par
By assumption  $f$ has finite $p$-variation over $\left[ 0, T \right]$ and it is therefore a \textit{regulated} function, i.e. it has finite right and left limits everywhere in $\left(  0,T\right)$ and a finite right (resp. left) limit at the left (resp. right) end point. Hence for all $i$, there exists a unique continuous function $F_i: \left[ x_i, x_{i+1} \right] \rightarrow \mathbb{R}^{e}$ which agrees with $f$ on $\left( x_i, x_{i+1} \right)$. \par
To avoid the discontinuities on the boundary, we will bound the integral over the subset $A^{\pi}_{i, j} := \left[ r_{\overline{k}(i)}, r_{\underline{k}(i+1)} \right] \times \left[ r_{\overline{k}(j)}, r_{\underline{k}(j+1)} \right]$ of $A_{i,j} := [x_i, x_{i+1}] \times [x_j, x_{j+1}]$, where
\begin{align*}
r_{\overline{k}(l)} := \inf \left\{ r_k \:\vert \: r_k > x_{l} \right\}, \quad
r_{\underline{k}(l)} := \sup \left\{ r_k \:\vert \: r_k < x_{l} \right\}.
\end{align*}
From Young's inequality \eqref{2DYoungIneq} and the fact that $\omega_{R, \rho + \varepsilon}$ is a control, the integral over the region $A_{i, j} - A^{\pi}_{i, j}$ will be arbitrarily small as $\norm{\pi}$ tends to zero.
Let $\omega^{\vphantom{0}}_{i, j}$ denote the control
\begin{align*}
\omega^{\vphantom{0}}_{i, j} ([s, t] \times [u,v])
:= \norm{F_i}_{p-var; [s,t]}^p \norm{F_j}_{p-var; [u, v]}^p, \quad (s, t) \times (u, v) \in A_{i, j}.
\end{align*}
Continuing,
\begin{align} \label{omegaDefn}
\omega([s,t] \times [u,v]) := \omega^{\frac{1}{\theta p}}_{i,j} ([s,t] \times [u,v]) \, \omega^{\frac{1}{\theta (\rho + \varepsilon)}}_{R, \rho + \varepsilon} ([s,t] \times [u,v]),
\end{align}
is again a control, where $\varepsilon$ is chosen such that $\theta := \frac{1}{p} + \frac{1}{\rho + \varepsilon} > 1$. \par
Now we have
\begin{align*}
&\abs{ \int_{A^{\pi}_{i,j}} \left\langle f^{\pi}_s - f_s, \, f^{\pi}_t - f_t \right\rangle_{\mathbb{R}^e} \wrt{R(s,t)} } \\
&\fqquad\leq \sum_{k, l} \abs{ \int_{r_k}^{r_{k+1}} \int_{r_l}^{r_{l+1}} \left\langle f_s - f_{r_k}, \, f_t - f_{r_l} \right\rangle_{\mathbb{R}^e} \wrt{R(s, t)} },
\end{align*}
where the sum is taken over the partition points in $A^{\pi}_{i, j}$. By Young's inequality \eqref{2DYoungIneq} and the fact that $\left\langle f_s - f_{r_k}, \, f_t - f_{r_l} \right\rangle_{\mathbb{R}^e}$ vanishes at all points $(r_k, r_l)$ in $A^{\pi}_{i, j}$, the expression above is bounded by
\begin{align*}
C_{p, \rho} \sum_{k, l} \norm{F_i}_{p-var; [r_k, r_{k+1}]} &\norm{F_j}_{p-var; [r_l, r_{l+1}]} \norm{R}_{\rho + \varepsilon-var; [r_k, r_{k+1}] \times [r_l, r_{l+1}]} \\
&\leq C_{p, \rho} \sum_{k, l} \omega^{\theta} ([r_k, r_{k+1}] \times [r_l, r_{l+1}]) \\
&\leq C_{p, \rho} \max_{k, l} \omega^{\theta - 1} ([r_k, r_{k+1}] \times [r_l, r_{l+1}]) \, \omega(A_{i, j}),
\end{align*}
which tends to zero as the mesh of the partition goes to zero.
\end{proof}

\subsection{A dense subspace of $\mathcal{H}_{1}^{d}$}

We now give a novel characterization of a subspace of $\mathcal{H}_{1}^{d}$
using Young-Stieltjes integrals. Let $R$ be of finite 2D $\rho$-variation,
where $\rho\in[1, 2)$. We define
\begin{align*}
\mathcal{W}_{\rho}^{d} := \bigcup_{p < \frac{\rho}{\rho- 1}} \mathcal{C}%
_{pw}^{p-var} \left(  [0, T]; \mathbb{R}^{d} \right)
\end{align*}
and equip it with the inner product
\begin{align}
\label{W1innerproduct}\left\langle f, g \right\rangle _{\mathcal{W}_{\rho}%
^{d}} := \int_{[0, T]^{2}} \left\langle f_{s}, \, g_{t} \right\rangle
_{\mathbb{R}^{d}} \, \mathrm{d} R(s,t).
\end{align}
One can check that $\langle\cdot, \cdot\rangle_{\mathcal{W}_{\rho}^{d}}$
defines a semi-inner product; it is bilinear due to the linearity of the
Young-Stieltjes integral, and positive semi-definite as well as symmetric
because the covariance function $R$ is positive semi-definite and symmetric.
We will identify $f$ and $g$ to be in the same equivalence class if
$\left\langle f - g, f - g \right\rangle _{\mathcal{W}_{\rho}^{d}} = 0$, and
quotient $\mathcal{W}_{\rho}^{d}$ with respect to these classes. This then
makes $\left\langle \cdot, \cdot\right\rangle _{\mathcal{W}_{\rho}^{d}}$ a
proper inner product.

\begin{proposition}
\label{W1eqH1} $\mathcal{W}_{\rho}^{d}$ is a dense subspace of $\mathcal{H}%
_{1}^{d}$, and the inclusion map $i: \left(  \mathcal{W}_{\rho}^{d}, \,
\left\langle \cdot, \cdot\right\rangle _{\mathcal{W}^{d}_{\rho}} \right)
\rightarrow\left( \mathcal{H}_{1}^{d}, \, \left\langle \cdot, \cdot
\right\rangle _{\mathcal{H}_{1}^{d}} \right) $ is an isometry.
\end{proposition}

\begin{proof}
Let $f \in \mathcal{W}^d_{\rho}$ and let $\pi(n) = \left\{ r^{(n)}_i \right\}$ be a sequence of partitions whose mesh vanishes as $n \rightarrow \infty$. As usual, we denote
\begin{align*}
f^{\pi(n)} := \sum_i f \left( r^{(n)}_i \right) \mathds{1}_{\left[ r^{(n)}_i, r^{(n)}_{i+1} \right)} (t).
\end{align*}
Now the key point to note is that for each $n$, $f^{\pi(n)}$ is in $\mathcal{W}_{\rho}^d \cap \mathcal{H}_1^d$; moreover
\begin{align}  \label{iso1}
\begin{split}
\norm{f^{\pi(n)}}_{\mathcal{H}_1^d}^2
&= \sum_{i, j} \left\langle
f_{r^{(n)}_i}, f_{r^{(n)}_j} \right\rangle_{\mathbb{R}^d} \left\langle \mathds{1}_{\left[ r^{(n)}_i, r^{(n)}_{i+1} \right)}, \mathds{1}_{\left[ r^{(n)}_j, r^{(n)}_{j+1} \right)} \right\rangle_{\mathcal{H}_1} \\
&= \sum_{i, j} \left\langle f_{r^{(n)}_i}, \, f_{r^{(n)}_j} \right\rangle_{\mathbb{R}^d} R
\begin{pmatrix}
r^{(n)}_i & r^{(n)}_{i+1} \\
r^{(n)}_j & r^{(n)}_{j+1}
\end{pmatrix} \\
&= \norm{f^{\pi(n)}}_{\mathcal{W}_{\rho}^d}^2.
\end{split}
\end{align}
From Lemma \ref{keyProp1}, $\norm{f^{\pi(n)} - f}_{\mathcal{W}^d_{\rho}} \rightarrow 0$, which means that $f^{\pi(n)}$ is Cauchy and from \eqref{iso1} and the completeness of $\mathcal{H}_1^d$, $\lim_{n \rightarrow \infty} f^{\pi(n)}$ exists in $\mathcal{H}_1^d$. We identify $f$ with this limit and under this identification we have
\begin{align}  \label{WH1norm}
\norm{f}_{\mathcal{H}^d_1}^2 = \int_{[0, T]^2} \left\langle f_s, \, f_t \right\rangle_{\mathbb{R}^d} \, \mathrm{d} R(s,t).
\end{align}
Since $\mathcal{W}^d_{\rho}$ contains all the generating functions $\left\{ \mathds{1}^{(u)}_{[0, t)} (\cdot) \right\}$ of $\mathcal{H}^d_1$, its completion, and hence closure, is all of $\mathcal{H}^d_1$.
\end{proof}

\begin{remark}
We recall the following non-degeneracy condition on Gaussian processes which
is featured in \cite{cfv2009}. We say that $R$ (or equivalently, $X$) is
non-degenerate on $\left[ 0, T\right]  $ if the following implication holds:
\begin{align}
\label{ndc}\int_{\left[ 0, T\right] ^{2}} \left\langle f_{s}, f_{t}
\right\rangle _{\mathbb{R}^{d}} \, \mathrm{d} R(s, t) = 0 \quad
\mathcal{\Rightarrow} \quad f = 0 \:\: \text{a.e.}.
\end{align}
Under this condition, each equivalence class of $\mathcal{W}_{\rho}^{d}$ would
then consist of functions which agree almost everywhere.
\end{remark}

\subsection{The Malliavin derivative and convergence in the tensor norm}

\label{tensorProd} We will now extend the results of the last section to the
tensor space $\mathcal{H}_{1}^{d} \otimes\mathcal{H}_{1}^{d}$. Let $\mathbf{X}
\in\mathcal{C}^{0, p-var} \left(  [0, T]; G^{\lfloor p \rfloor} \left(
\mathbb{R}^{d}\right)  \right) $ and assume that for all $h \in\mathcal{H}%
_{1}^{d}$, $\Phi(h)$ can be embedded in $\mathcal{C}^{q-var}\left(  \left[
0,T \right] ; \mathbb{R}^{d}\right) $ where $\frac{1}{p} + \frac{1}{q} > 1$.
Then the Malliavin derivative of $Y$ satisfying
\begin{align*}
\mathrm{d} Y_{t} = V(Y_{t}) \circ\mathrm{d} \mathbf{X}_{t}, \quad Y_{0} =
y_{0},
\end{align*}
is given by
\begin{align*}
\mathcal{D}_{h} Y_{t}  & = \int_{0}^{t} J_{t}^{\mathbf{X}} \left(
J_{s}^{\mathbf{X}}\right) ^{-1} V\left(  Y_{s} \right)  \, \mathrm{d}
\Phi(h)(s)\\
& = \int_{0}^{T} \mathds{1}_{[0,t)} \left(  s\right)  J_{t}^{\mathbf{X}}
\left(  J_{s}^{\mathbf{X}}\right) ^{-1} V\left(  Y_{s} \right)  \, \mathrm{d}
\Phi(h)(s).
\end{align*}
Using $g(s, t)$ to denote the Malliavin derivative,
\begin{align*}
g(s,t) := \mathcal{D}_{s} Y_{t} = \mathds{1}_{[0, t)} (s) J_{t}^{\mathbf{X}}
\left( J_{s}^{\mathbf{X}} \right) ^{-1} V\left(  Y_{s} \right) ,
\end{align*}
then with respect to any partition $\pi= \{ r_{i} \}$ of $[0, T]$, we also
have
\begin{align*}
\mathcal{D}_{s} Y^{\pi}_{t}  & = \sum_{i} \mathcal{D}_{s} Y_{r_{i}}
\mathds{1}_{\left[  r_{i}, r_{i+1} \right) } (t)\\
& = \sum_{i} g \left(  s, r_{i} \right)  \mathds{1}_{\left[  r_{i}, r_{i+1}
\right) } (t).
\end{align*}
We will proceed to show that $\mathcal{D} Y^{\pi}$ lies in $\mathcal{H}%
_{1}^{d} \otimes\mathcal{H}_{1}^{d}$ almost surely, and under suitable
regularity assumptions on $\mathcal{D} Y$, we have $\left\|  \mathcal{D}
Y^{\pi} - \mathcal{D}Y\right\| _{\mathcal{H}_{1}^{d} \otimes\mathcal{H}%
_{1}^{d}} \rightarrow0$ as $\left\|  \pi\right\|  \rightarrow0$. Coupled with
the results in the previous section, this will mean that $Y^{\pi}$ converges
to $Y$ in $\mathbb{D}^{1,2} \left( \mathcal{H}_{1}^{d}\right) $, and
$\delta^{X}(Y)$ is then the $L^{2}(\Omega)$ limit of $\delta^{X} \left(
Y^{\pi} \right) $.

\begin{proposition}
\label{keyProp2} Let $g: [0, T]^{2} \rightarrow\mathbb{R}^{e} \otimes
\mathbb{R}^{d}$ be of the form $g(s, t) = \mathds{1}_{[0, t)}(s) \tilde{g}_{1}
(t) \tilde{g}_{2} (s)$, where $\tilde{g}_{1} \in\mathcal{C}_{pw}^{p-var}
\left(  [0, T];\mathbb{R}^{e} \otimes\mathbb{R}^{e} \right) $ and
\newline$\tilde{g}_{2} \in\mathcal{C}_{pw}^{p-var} \left(  [0, T];\mathbb{R}%
^{e} \otimes\mathbb{R}^{d} \right) $. Let $R$ be continuous and of finite 2D
$\rho$-variation, $\rho\in\left[  1, \frac{3}{2} \right) $, and we assume that
$\frac{1}{p} + \frac{1}{\rho} > 1$. For any partition $\pi= \{ r_{i} \}$ of
$[0, T]$, let $g^{\pi}: [0, T]^{2} \rightarrow\mathbb{R}^{e} \otimes
\mathbb{R}^{d}$ denote
\begin{align}
\label{gApprox}g^{\pi} (s, t) := \sum_{i} g \left(  s, r_{i} \right)
\mathds{1}_{\left[  r_{i}, r_{i+1} \right) } (t).
\end{align}
Then
\begin{align}
\label{HSNorm}\int_{[0, T]^{2}} \left(  \int_{[0, T]^{2}} \left\langle \left(
g^{\pi} - g \right)  (u, s), \, \left(  g^{\pi} - g \right)  (v, t)
\right\rangle _{\mathbb{R}^{e} \otimes\mathbb{R}^{d}} \, \mathrm{d} R(u, v)
\right)  \, \mathrm{d} R(s, t) \rightarrow0.
\end{align}

\end{proposition}

\begin{remark}
Here and henceforth, we canonically identify 2-tensors with matrices, and
$\tilde{g}_{1} (t) \tilde{g}_{2} (s)$ denotes matrix multiplication of
$\tilde{g}_{1}(t)$ with $\tilde{g}_{2} (s)$.
\end{remark}

\begin{proof}
Similar to the proof of Lemma \ref{keyProp1}, we will first partition $[0, T]^2 \times [0, T]^2$ into sub-regions $A \times B = ([a_1, a_2] \times [a_3, a_4]) \times ([b_1, b_2] \times [b_3, b_4])$ on which the integrand is continuous (shrinking each region if necessary to deal with discontinuities at the boundaries), then sum the results to obtain the full proof. \par
Under the conditions imposed on $g$, we will show that the 4D-integral in \eqref{HSNorm} can be written as an iterated 2D-integral and
\begin{align*}
\int_{A} \left( \int_{B} \left\langle (g - g^{\pi}) (u, s), \, (g - g^{\pi}) (v, t) \right\rangle_{\mathbb{R}^e \otimes \mathbb{R}^d} \wrt{R(u, v)} \right) \wrt{R(s, t)} \\
= \sum_{i, j} \int_{[r_i, r_{i+1}] \times [r_j, r_{j+1}]} I^{i,j} (s,t) \wrt{R(s, t)} \xrightarrow{\norm{\pi} \rightarrow 0} 0,
\end{align*}
where
\begin{align*}
I^{i,j} (s,t) := \int_{B} \left\langle g (u, s) - g(u, r_i), \, g (v, t) - g(v, r_j) \right\rangle_{\mathbb{R}^e \otimes \mathbb{R}^d} \wrt{R(u, v)}.
\end{align*}
First observe that for any $r\leq s$,
\begin{align*}
g(u, s) - g(u, r)
&= \mathds{1}_{[0, s)} (u) \tilde{g}_1(s) \tilde{g}_2(u) - \mathds{1}_{[0, r)} (u) \tilde{g}_1(r) \tilde{g}_2(u) \\
&= \mathds{1}_{[0, r)} (u) \big( \tilde{g}_1(s) - \tilde{g}_1(r) \big) \tilde{g}_2(u) + \mathds{1}_{[r, s)} (u) \tilde{g}_1(s) \tilde{g}_2(u).
\end{align*}
Thus for $(s, t) \in [r_i, r_{i+1}] \times [r_j, r_{j+1}] \subset A$,
\begin{align*}
I^{i,j}(s,t) = I_1^{i, j} (s,t) + I_2^{i, j} (s,t) + I_3^{i, j} (s,t) + I_4^{i, j} (s,t),
\end{align*}
where
\begin{align*}
I_{1}^{i,j} (s, t)
&:= \sum_{l, m, n, k} \big( \tilde{g}_1(s) - \tilde{g}_1(r_i) \big)_{lm}  \big( \tilde{g}_1 (t) - \tilde{g}_1(r_j) \big)_{ln} \\
&\fqquad\times \int_{[b_1, r_i] \times [b_3, r_j]} \big( \tilde{g}_2 (u) \big)_{mk} \big( \tilde{g}_2 (v) \big)_{nk} \wrt{R\left( u, v \right)}, \\
I_{2}^{i,j} (s, t)
&:= \sum_{l, m, n, k} \big( \tilde{g}_1(s) \big)_{lm} \big( \tilde{g}_1 (t) - \tilde{g}_1(r_j) \big)_{ln} \\
&\fqquad\times \int_{[r_i, s] \times [b_3, r_j]} \big( \tilde{g}_2 (u) \big)_{mk} \big( \tilde{g}_2 (v) \big)_{nk} \wrt{R\left( u, v \right)}, \\
I_{3}^{i,j} (s, t)
&:= \sum_{l, m, n, k} \big( \tilde{g}_1(s) - \tilde{g}_1(r_i) \big)_{lm} \big( \tilde{g}_1(t) \big)_{ln} \\
&\fqquad\times \int_{[b_1, r_i] \times[r_j, t]} \big( \tilde{g}_2 (u) \big)_{mk} \big( \tilde{g}_2 (v) \big)_{nk} \wrt{R\left( u, v \right)}, \\
I_{4}^{i,j} (s, t)
&:= \sum_{l, m, n, k} \big( \tilde{g}_1(s) \big)_{lm} \big( \tilde{g}_1(t) \big)_{ln} \\
&\fqquad\times \int_{[r_i, s] \times [r_j, t]} \big( \tilde{g}_2 (u) \big)_{mk} \big( \tilde{g}_2 (v) \big)_{nk} \wrt{R\left( u, v \right)},
\end{align*}
and $(g)_{lm}$ denotes the $(l, m)^{th}$ entry of the matrix $g$. Note also that it is possible that $r_i \leq b_1$ or $r_j \leq b_3$, in which case we define the integral to be zero. \par
For $k = 1, 2, 3$, since the summands in $I^{i, j}_k (s,t)$ are products of 1D functions, we have the following bounds on the 2D $p^{th}$-variation of $I_k^{i,j} (s,t)$ in $[r_i, r_{i+1}] \times [r_j, r_{j+1}]$:
\begin{align} \label{I123Est}
\begin{split}
&\norm{I_1^{i,j}}_{p-var; [r_i, r_{i+1}] \times [r_j, r_{j+1}]} \\
&\quad\quad\leq C_{p, \rho, d} \norm{\tilde{g}_2}_{\mathcal{V}^p;[0,T]}^2 \norm{R}_{\rho-var; [0, T]^2} \norm{\tilde{g}_1}_{p-var; [r_i, r_{i+1}]} \norm{\tilde{g}_1}_{p-var; [r_j, r_{j+1}]}, \\
&\norm{I_2^{i,j}}_{p-var; [r_i, r_{i+1}] \times [r_j, r_{j+1}]} \\
&\quad\quad\leq C_d \norm{\tilde{g}_1}_{p-var; [r_j, r_{j+1}]} \left( \norm{\tilde{g}_1}_{p-var; [r_i, r_{i+1}]} \norm{h_1}_{\infty} + \norm{\tilde{g}_1}_{\infty} \norm{h_1}_{p-var; [r_i, r_{i+1}]} \right), \\
&\norm{I_3^{i,j}}_{p-var; [r_i, r_{i+1}] \times [r_j, r_{j+1}]} \\
&\quad\quad\leq C_d \norm{\tilde{g}_1}_{p-var; [r_i, r_{i+1}]} \left( \norm{\tilde{g}_1}_{p-var; [r_j, r_{j+1}]} \norm{h_2}_{\infty} + \norm{\tilde{g}_1}_{\infty} \norm{h_2}_{p-var; [r_j, r_{j+1}]} \right).
\end{split}
\end{align}
Here, $h_1$ and $h_2$ denote the functions (suppressing the dependence on $m, n$ and $k$ in the notation since the bounds are independent of them)
\begin{align*}
&h_1(s) := \int_{[r_i, s] \times [b_3, r_j]} \big( \tilde{g}_2 (u) \big)_{mk} \big( \tilde{g}_2 (v) \big)_{nk} \wrt{R\left( u, v \right)}, \\
&h_2(t) := \int_{[b_1, r_i] \times[r_j, t]} \big( \tilde{g}_2 (u) \big)_{mk} \big( \tilde{g}_2 (v) \big)_{nk} \wrt{R\left( u, v \right)}.
\end{align*}
Choosing $\varepsilon$ sufficiently small such that $\rho + \varepsilon < p$ and $\frac{1}{p} + \frac{1}{\rho + \varepsilon} > 1$, we have
\begin{align} \label{hEst}
\begin{split}
\norm{h_1}_{p-var; [r_i, r_{i+1}]}
&\leq \norm{h_1}_{\rho + \varepsilon-var; [r_i, r_{i+1}]} \\
&\leq C_{p,\rho} \norm{\tilde{g}_2}_{\mathcal{V}^p; [0, T]}^2 \omega^{\vphantom{0}}_{R, \rho + \varepsilon} ([r_i, r_{i+1}] \times [0, T])^{\frac{1}{\rho + \varepsilon}} \\
\norm{h_2}_{p-var; [r_j, r_{j+1}]}
&\leq \norm{h_2}_{\rho + \varepsilon-var; [r_j, r_{j+1}]} \\
&\leq C_{p,\rho} \norm{\tilde{g}_2}_{\mathcal{V}^p; [0, T]}^2 \omega^{\vphantom{0}}_{R, \rho + \varepsilon} ( [0, T] \times [r_j, r_{j+1}] )^{\frac{1}{\rho + \varepsilon}}, \\
\norm{h_1}_{\infty}, \norm{h_2}_{\infty}
&\leq C_{p,\rho} \norm{\tilde{g}_2}_{\mathcal{V}^p; [0, T]}^2 \norm{R}_{\rho-var; [0, T]^2}.
\end{split}
\end{align}
From \eqref{I123Est} and \eqref{hEst}, we see that the 2D $p^{th}$-variations of $I_1^{i, j}$, $I_2^{i, j}$ and $I_3^{i, j}$ over $[s, t] \times [u, v] \subset [a_1, a_2] \times [a_3, a_4]$ are controlled respectively by
\begin{align*}
&\omega^{\vphantom{0}}_{I_1} \left([s, t] \times[u, v] \right)
:= C_1 \norm{\tilde{g}_2}_{\mathcal{V}^p; [0, T]}^{2p} \norm{\tilde{g}_1}_{p-var; [s, t]}^p \norm{\tilde{g}_1}_{p-var; [u, v]}^p \norm{R}_{\rho-var; [0, T]^2}^p, \\
&\omega^{\vphantom{0}}_{I_2} \left([s, t] \times[u, v]\right)
:= C_2 \norm{\tilde{g}_2}_{\mathcal{V}^p; [0, T]}^{2p} \norm{\tilde{g}_1}^p_{p-var; [u, v]} \\
&\quad\qquad\qquad\times \left( \norm{\tilde{g}_1}^p_{p-var; [s, t]} \norm{R}^p_{\rho-var; [0, T]^2} + \norm{\tilde{g}_1}_{\infty}^p \omega^{\vphantom{0}}_{R, \rho + \varepsilon} ([s, t] \times [0, T])^{\frac{p}{\rho + \varepsilon}} \right), \\
&\omega^{\vphantom{0}}_{I_3} \left([s, t] \times[u, v]\right)
:= C_3 \norm{\tilde{g}_2}_{\mathcal{V}^p; [0, T]}^{2p} \norm{\tilde{g}_1}^p_{p-var; [s, t]} \\
&\quad\qquad\qquad\times \left( \norm{\tilde{g}_1}^p_{p-var; [u, v]} \norm{R}^p_{\rho-var; [0, T]^2} + \norm{\tilde{g}_1}_{\infty}^p \omega^{\vphantom{0}}_{R, \rho + \varepsilon} ([0, T] \times [u, v])^{\frac{p}{\rho + \varepsilon}} \right).
\end{align*}
For $I_4^{i,j} (s, t)$, if we let
\begin{align*}
h_3(s,t) := \int_{[r_i, s] \times [r_j, t]} \big( \tilde{g}_2 (u) \big)_{mk} \big( \tilde{g}_2 (v) \big)_{nk} \wrt{R\left( u, v \right)},
\end{align*}
we have
\begin{align*}
\norm{h_3}_{p-var; [r_i, r_{i+1}] \times [r_j, r_{j+1}]}
&\leq \norm{h_3}_{\rho + \varepsilon-var; [r_i, r_{i+1}] \times [r_j, r_{j+1}]} \\
&\leq C_{p, \rho} \norm{\tilde{g}_2}_{\mathcal{V}^p; [0, T]}^2 \omega^{\vphantom{0}}_{R, \rho + \varepsilon} \left( [r_i, r_{i+1}] \times [r_j, r_{j+1}] \right)^{\frac{1}{\rho + \varepsilon}}.
\end{align*}
From Lemma \ref{productLemma}, we then conclude that the 2D $p^{th}$-variation of $I^{i,j}_4$ over $[s, t] \times [u, v]$ is controlled by
\begin{align*}
\omega^{\vphantom{0}}_{I_4} \left([s, t] \times[u, v]\right)
&:= C \norm{\tilde{g}_2}_{\mathcal{V}^p; [0, T]}^{2p} \omega^{\vphantom{0}}_{R, \rho + \varepsilon} \left( [s, t] \times [u, v] \right)^{\frac{p}{\rho + \varepsilon}} \\
&\qquad\times \left( \norm{\tilde{g}_1}^p_{\infty}  + \norm{\tilde{g}_1}^p_{p-var; [s, t]} \right) \left( \norm{\tilde{g}_1}^p_{\infty}  + \norm{\tilde{g}_1}^p_{p-var; [u, v]} \right).
\end{align*}
Now define $\omega$ as
\begin{align*}
\omega([s,t] \times [u,v]) = \omega^{\frac{1}{\theta p}}_I ([s,t] \times [u,v]) \, \omega^{\frac{1}{\theta(\rho + \varepsilon)}}_{R, \rho + \varepsilon} ([s,t] \times [u,v]),
\end{align*}
where $\omega^{\vphantom{0}}_I$ denotes the control $\omega^{\vphantom{0}}_{I_1} + \omega^{\vphantom{0}}_{I_2} + \omega^{\vphantom{0}}_{I_3} + \omega^{\vphantom{0}}_{I_4}$ and $\theta = \frac{1}{p} + \frac{1}{\rho + \varepsilon}$.
Then observing that $I^{i,j} (r_i, \cdot) = I^{i, j} (\cdot, r_j) = 0$ for all $i, j$, we use Young's inequality \eqref{2DYoungIneq} to obtain
\begin{align} \label{4Dvanish}
\sum_{i, j} \int_{[r_i, r_{i+1}] \times [r_j, r_{j+1}]} I^{i,j}(s,t) \wrt{R(s, t)}
&\leq \sum_{i, j} \omega^{\theta} ([r_i, r_{i+1}] \times [r_j, r_{j+1}]) \rightarrow 0.
\end{align}
\end{proof}
The following lemma was used in Proposition \ref{keyProp2}.

\begin{lemma}
\label{productLemma} Let $g_{1} \in\mathcal{C}^{p-var} \left(  [s_{1}, s_{2}];
\mathbb{R} \right) $ and $g_{2} \in\mathcal{C}^{p-var} \left(  [t_{1}, t_{2}];
\mathbb{R} \right) $. Given a 2D control $\omega$, let $f \in\mathcal{C}%
^{p-var}\left(  [s_{1}, s_{2}] \times[t_{1}, t_{2}]; \mathbb{R} \right) $ have
finite 2D $p$-variation controlled by $\omega$. In addition, assume that
$f(s_{1}, t) = f(s, t_{1}) = 0$ for all $s, t$ in $[s_{1}, s_{2}] \times
[t_{1}, t_{2}]$. Then the 2D $p^{th}$-variation of $f(u, v) g_{1}(u) g_{2}(v)$
over $[s_{1}, s_{2}] \times[t_{1}, t_{2}]$ is controlled by
\begin{align*}
4^{p-1} \omega\left( [s_{1}, s_{2}] \times[t_{1}, t_{2}] \right)  \left(
\left\|  g_{1}\right\| ^{p}_{\infty} + \left\|  g_{1}\right\| ^{p}_{p-var;
[s_{1}, s_{2}]} \right)  \left(  \left\|  g_{2}\right\| ^{p}_{\infty} +
\left\|  g_{2}\right\| ^{p}_{p-var; [t_{1}, t_{2}]} \right) .
\end{align*}

\end{lemma}

\begin{proof}
Let $\{ (u_i, v_j) \}$ be any partition of $[s_1, s_2] \times[t_1, t_2]$. We have
\begin{align*}
&\sum_{i,j} \abs{ fg_1 \begin{pmatrix}
u_i & u_{i+1} \\
v_j & v_{j+1}
\end{pmatrix}}^p \\
&= \sum_{i, j} \abs{ \big( f(u_i, v_j) - f(u_i, v_{j+1}) \big) g_1(u_i) + \big( f(u_{i+1}, v_{j+1}) - f(u_{i+1}, v_j) \big) g_1(u_{i+1})}^p \\
&\leq 2^{p-1} \left[ \sum_{i, j} \abs {f \begin{pmatrix}
u_i & u_{i+1} \\
v_j & v_{j+1}
\end{pmatrix} g_1(u_i)}^p
+ \sum_i \abs{g_1(u_{i+1}) - g_1(u_i)}^p \sum_j \abs{f \begin{pmatrix}
s_1 & u_{i+1} \\
v_j & v_{j+1}
\end{pmatrix}}^p \right] \\
&\leq 2^{p-1} \left[ \norm{g_1}^p_{\infty} \norm{f}^p_{p-var} + \norm{f}^p_{p-var} \sum_i \abs{g_1(u_{i+1}) - g_1(u_i)}^p \right],
\end{align*}
which tells us that the 2D $p^{th}$-variation of $fg_1$ is controlled by
\begin{align} \label{newCont}
2^{p-1} \omega \left( [s_1, s_2] \times [t_1, t_2] \right) \left( \norm{g_1}^p_{\infty} + \norm{g_1}^p_{p-var; [s_1, s_2]} \right).
\end{align}
Repeating the same procedure with $fg_1$ (controlled by \eqref{newCont}) in place of $f$ and $g_2$ in place of $g_1$ completes the proof.
\end{proof}

\begin{proposition}
\label{H1tensorH1equiv1} Let $g(s, t) = \mathds{1}_{[0, t)}(s) \tilde{g}_{1}
(t) \tilde{g}_{2} (s)$, where \newline$\tilde{g}_{1}, \tilde{g}_{2}
\in\mathcal{C}_{pw}^{p-var} \left(  [0, T];\mathbb{R}^{d} \otimes
\mathbb{R}^{d} \right) $, and let $g^{\pi}$ be defined as in \eqref{gApprox}.
Let $R$ be of finite 2D $\rho$-variation, $\rho\in\left[  1, \frac{3}{2}
\right) $, and we assume that $\frac{1}{p} + \frac{1}{\rho} > 1$. Then $g
\in\mathcal{H}_{1}^{d} \otimes\mathcal{H}_{1}^{d}$, with norm given by
\begin{align}
\label{tensorNorm}\left\|  g\right\| _{\mathcal{H}_{1}^{d} \otimes
\mathcal{H}_{1}^{d}} = \sqrt{\int_{[0, T]^{2}} \left(  \int_{[0, T]^{2}}
\left\langle g(u, s), \, g(v, t) \right\rangle _{\mathbb{R}^{d} \otimes
\mathbb{R}^{d}} \, \mathrm{d} R(u, v) \right)  \, \mathrm{d} R(s, t)},
\end{align}
and
\begin{align}
\label{tensorVanishing}\left\|  g^{\pi} - g\right\| _{\mathcal{H}_{1}^{d}
\otimes\mathcal{H}_{1}^{d}} \rightarrow0
\end{align}
as $\left\|  \pi\right\|  \rightarrow0$.
\end{proposition}

\begin{proof}
Given a $d$-by-$d$ matrix function $A(s)$, let $a^{(i)}_j (s)$ denote the $(i, j)^{th}$ entry of $A(s)$. Using the canonical identification
\begin{align} \label{canonID}
A(s) \mathds{1}_{[a, b)} (t) \simeq \sum_{j=1}^d \sum_{k=1}^d a^{(k)}_j (s) e_k \otimes \mathds{1}^{(j)}_{[a, b)} (t), \quad a, b \in [0, T],
\end{align}
we see that $g^{\pi}$ is a member of $\mathcal{W}^d_{\rho} \otimes \mathcal{H}_1^d$, and thus lies in $\mathcal{H}_1^d \otimes \mathcal{H}_1^d$ by Proposition \ref{W1eqH1}. Furthermore, we can compute the square of its norm
\begin{align*}
\norm{g^{\pi}}^2_{\mathcal{H}_1^d \otimes \mathcal{H}_1^d}
&= \sum_{k, l} \int_{[0, T]^2} \sum_{j=1}^d \left\langle g_j (u, r_k), g_j (v, r_l) \right\rangle_{\mathbb{R}^d} \wrt{R(u, v)} R \begin{pmatrix}
r_k & r_{k+1} \\
r_l & r_{l+1}
\end{pmatrix} \\
&= \int_{[0, T]^2} \left( \int_{[0, T]^2} \left\langle g^{\pi} (u, s), g^{\pi} (v, t) \right\rangle_{\mathbb{R}^d \otimes \mathbb{R}^d} \wrt{R(u, v)} \right) \wrt{R(s, t)}.
\end{align*}
Taking any sequence of partitions $\pi(n)$ with vanishing mesh, we know that $g^{\pi(n)}$ is Cauchy as $n \rightarrow \infty$ by Proposition \ref{keyProp2}, and we identify $g$ with its limit in $\mathcal{H}_1^d \otimes \mathcal{H}_1^d$. Invoking Proposition \ref{keyProp2} again gives us \eqref{tensorNorm} and \eqref{tensorVanishing}.
\end{proof}

\subsection{The It\^{o}-Skorohod isometry revisited}

\label{itoSkorohodIsometry}

\begin{theorem}
\label{isometry1} Let $X$ be a continuous, centered Gaussian process in
$\mathbb{R}^{d}$ with i.i.d. components, and assume that its continuous
covariance function satisfies $\left\|  R\right\| _{\rho-var; [0, T]^{2}} <
\infty$ for some $\rho\in\left[ 1, \frac{3}{2} \right) $. Given $p$ satisfying
$\frac{1}{p} + \frac{1}{\rho} > 1$, let $Y$ be a random variable which
satisfies, almost surely,

\begin{enumerate}
[(i)]

\item $Y \in\mathcal{C}^{p-var}_{pw} \left( [0, T]; \mathbb{R}^{d} \right) $,

\item $\mathcal{D}Y: [0, T]^{2} \rightarrow\mathbb{R}^{d} \otimes
\mathbb{R}^{d}$ is of the form $\mathds{1}_{[0, t)}(s) \tilde{g}_{1} (t)
\tilde{g}_{2} (s)$, where $\tilde{g}_{1}, \tilde{g}_{2}$ are both in
$\mathcal{C}_{pw}^{p-var} \left(  [0, T];\mathbb{R}^{d} \otimes\mathbb{R}^{d}
\right) $.
\end{enumerate}

Then $\lim_{\left\|  \pi\right\|  \rightarrow0} Y^{\pi} = Y$ in $\mathbb{D}%
^{1, 2} (\mathcal{H}_{1}^{d})$ if and only if
\begin{align*}
\mathbb{E} \left[  \int_{[0, T]^{2}} \left\langle Y^{\pi}_{s} - Y_{s}, \,
Y^{\pi}_{t} - Y_{t} \right\rangle _{\mathbb{R}^{d}} \, \mathrm{d}
R(s,t)\right]  \rightarrow0
\end{align*}
and
\begin{align*}
\mathbb{E} \left[  \int_{[0, T]^{2}} \left(  \int_{[0, T]^{2}} \left\langle
\mathcal{D}_{r} \left(  Y^{\pi}_{t} - Y_{t} \right) , \, \mathcal{D}_{q}
\left(  Y^{\pi}_{s} - Y_{s} \right)  \right\rangle _{\mathbb{R}^{d}
\otimes\mathbb{R}^{d}} \, \mathrm{d} R(r, q) \right)  \, \mathrm{d} R(s,
t)\right]  \rightarrow0
\end{align*}
as $\left\|  \pi\right\|  \rightarrow0$, in which case $\lim_{\left\|
\pi\right\|  \rightarrow0} \mathbb{E} \left[ \delta^{X}\left( Y^{\pi} - Y
\right) ^{2}\right]  = 0$ and $\mathbb{E} \left[  \delta^{X}\left(  Y \right)
^{2}\right] $ is equal to
\begin{align*}
\mathbb{E} \left[ \int_{[0, T]^{2}} \left\langle Y_{s}, \, Y_{t} \right\rangle
_{\mathbb{R}^{d}} \, \mathrm{d} R(s,t) \right]  + \mathbb{E} \left[  \int_{[0,
T]^{4}} \mathrm{tr} \left(  \mathcal{D}_{r} Y_{t} \, \mathcal{D}_{q} Y_{s}
\right)  \, \mathrm{d} R(s, r) \, \mathrm{d} R(t, q) \right] .
\end{align*}

\end{theorem}

\begin{proof}
For a Malliavin-smooth real-valued random variable $F$, we will write
\begin{align*}
\mathcal{D} F = \left( \mathcal{D}^{(1)} F, \ldots, \mathcal{D}^{(d)} F \right) \in \mathcal{H}_1^d,
\end{align*}
which means that $\mathcal{D}^{(j)}_s Y^{(i)}_t$ will denote the $(i, j)^{th}$-entry of the matrix $\mathcal{D}_s Y_t$. \par
From Propositions \ref{W1eqH1} and \ref{H1tensorH1equiv1},
\begin{align*}
\exptn{\left\| Y^{\pi} - Y \right\|^2_{\mathcal{H}_1^d}}
&= \exptn{\int_{[0, T]^2} \left\langle Y^{\pi}_s - Y_s, \, Y^{\pi}_t - Y_t \right\rangle_{\mathbb{R}^d} \wrt{R(s, t)} },
\end{align*}
and $\exptn{\norm{\mathcal{D} Y^{\pi} - \mathcal{D} Y}^2_{\mathcal{H}_1^d \otimes \mathcal{H}_1^d}} $ is equal to
\begin{align*}
\exptn{\int_{[0, T]^2} \left( \int_{[0, T]^2} \left\langle \mathcal{D}_r \left( Y^{\pi}_t - Y_t \right), \, \mathcal{D}_q \left( Y^{\pi}_s - Y_s \right) \right\rangle_{\mathbb{R}^d \otimes \mathbb{R}^d} \wrt{R(r, q)} \right) \wrt{R(s, t)}}.
\end{align*}
Furthermore, It\^{o}-Skorohod isometry (see \cite{nualart2006}) gives us
\begin{align*}
\exptn{\delta^X\left( Y \right)^2}
&= \exptn{\norm{Y}^2_{\mathcal{H}_1^d}} + \exptn{\mathrm{trace} \left( \mathcal{D} Y \circ \mathcal{D} Y \right)} \\
&= \lim_{\norm{\pi} \rightarrow 0} \exptn{\norm{Y^{\pi}}^2_{\mathcal{H}_1^d}} + \lim_{\norm{\pi} \rightarrow 0} \exptn{\mathrm{trace} \left( \mathcal{D} Y^{\pi} \circ \mathcal{D} Y^{\pi} \right)},
\end{align*}
since
\begin{align*}
\exptn{\delta^X\left(Y^{\pi} - Y \right)^2}
\leq \exptn{\norm{ Y^{\pi} - Y }^2_{\mathcal{H}_1^d}} +
\exptn{\norm{ \mathcal{D} Y^{\pi} - \mathcal{D} Y }^2_{\mathcal{H}_1^d \otimes \mathcal{H}_1^d}}.
\end{align*}
Recall that the trace term is given by
\begin{align*}
\mathrm{trace} \left( \mathcal{D} Y \circ \mathcal{D} Y \right)
= \sum_{m=1}^{\infty} \left\langle \mathcal{D} Y (h_m), \, \left( \mathcal{D} Y \right)^* (h_m) \right\rangle_{\mathcal{H}^d_1},
\end{align*}
where $\{h_m\}$ denotes any orthonormal basis for $\mathcal{H}_1^d$ and
\begin{align*}
&\mathcal{D} Y (h_m) (r)
= \sum_{k=1}^d \left[ \left\langle \mathcal{D}_{\cdot} Y^{(k)}_r, h_m (\cdot) \right\rangle_{\mathcal{H}_1^d} \right] e_k, \\
&\left( \mathcal{D} Y \right)^* (h_m) (r)
= \sum_{k=1}^d \left[ \left\langle \mathcal{D}^{(k)}_r Y_{\cdot}, h_m (\cdot) \right\rangle_{\mathcal{H}_1^d} \right] e_k, \quad r \in [0, T], \, m = 1, \ldots.
\end{align*}
For the first term, we have
\begin{align*}
\lim_{\norm{\pi} \rightarrow 0} \exptn{\norm{Y^{\pi}}^2_{\mathcal{H}_1^d}}
&= \lim_{\norm{\pi} \rightarrow 0} \exptn{\sum_{j, k} \left\langle Y_{t_j}, Y_{t_k} \right\rangle_{\mathbb{R}^d} R \begin{pmatrix}
t_j & t_{j+1} \\
t_k & t_{k+1}
\end{pmatrix} } \\
&= \exptn{\int_{[0, T]^2} \left\langle Y_s, Y_t \right\rangle_{\mathbb{R}^d} \wrt{R(s,t)}},
\end{align*}
and for the second term, we need to compute
\begin{align*}
\exptn{\mathrm{trace} \left( \mathcal{D} Y^{\pi} \circ \mathcal{D} Y^{\pi} \right)}
= \exptn{\sum_{m=1}^{\infty} \left\langle \mathcal{D} Y^{\pi} (h_m), \, \left( \mathcal{D} Y^{\pi} \right)^* (h_m) \right\rangle_{\mathcal{H}^d_1}}.
\end{align*}
We have
\begin{align*}
&\mathcal{D} Y^{\pi} (h_m) (r)
= \sum_{k=1}^d \left[ \sum_i \left\langle \mathcal{D}_{\cdot} Y^{(k)}_{t_i}, h_m (\cdot) \right\rangle_{\mathcal{H}_1^d} \mathds{1}_{\Delta_i} (r) \right] e_k, \quad \mathrm{and} \\
&\left( \mathcal{D} Y^{\pi} \right)^* (h_m) (r)
= \sum_{k=1}^d \left[ \sum_j \left\langle \mathcal{D}^{(k)}_r Y_{t_j} \mathds{1}_{\Delta_j} (\cdot), h_m (\cdot) \right\rangle_{\mathcal{H}_1^d} \right] e_k, \\
&\qquad r \in [0, T], \, m = 1, \ldots,
\end{align*}
which yields
\begin{align*}
&\exptn{\sum_{m=1}^{\infty} \left\langle \mathcal{D} Y^{\pi} h_m, \, \left( \mathcal{D} Y^{\pi} \right)^* h_m \right\rangle_{\mathcal{H}^d_1}} \\
&\quad = \exptn{\sum_{i,j} \sum_{k=1}^d \sum_{m=1}^{\infty} \left\langle \left\langle \mathcal{D}^{(k)}_r Y_{t_j} \mathds{1}_{\Delta_j} (\cdot), h_m(\cdot) \right\rangle_{\mathcal{H}_1^d}, \left\langle \mathcal{D}_{\cdot} Y^{(k)}_{t_i}, h_m(\cdot) \right\rangle_{\mathcal{H}_1^d} \mathds{1}_{\Delta_i} (r) \right\rangle_{\mathcal{H}_1}} \\
&\quad = \exptn{\sum_{i,j} \sum_{k=1}^d \left\langle \sum_{m=1}^{\infty} \left\langle \mathcal{D}^{(k)}_r Y_{t_j} \mathds{1}_{\Delta_j} (\cdot), h_m(\cdot) \right\rangle_{\mathcal{H}_1^d} \left\langle \mathcal{D}_{\cdot} Y^{(k)}_{t_i}, h_m(\cdot) \right\rangle_{\mathcal{H}_1^d}, \mathds{1}_{\Delta_i} (r) \right\rangle_{\mathcal{H}_1}} \\
&\quad = \exptn{\sum_{i,j} \sum_{k=1}^d \left\langle \left\langle \mathcal{D}_r^{(k)} Y_{t_j} \mathds{1}_{\Delta_j} (\cdot), \mathcal{D}_{\cdot} Y^{(k)}_{t_i} \right\rangle_{\mathcal{H}_1^d}, \mathds{1}_{\Delta_i} (r) \right\rangle_{\mathcal{H}_1}},
\end{align*}
where $r$ is the variable for the outer $\mathcal{H}_1$-inner product. \par
Since
\begin{align*}
\left\langle \mathcal{D}_r^{(k)} Y_{t_j} \mathds{1}_{\Delta_j} (\cdot), \mathcal{D}_{\cdot} Y^{(k)}_{t_i} \right\rangle_{\mathcal{H}_1^d}
&= \sum_{l=1}^d \left\langle \mathcal{D}_r^{(k)} Y^{(l)}_{t_j} \mathds{1}_{\Delta_j} (\cdot), \mathcal{D}^{(l)}_{\cdot} Y^{(k)}_{t_i} \right\rangle_{\mathcal{H}_1} \\
&= \sum_{l=1}^d \mathcal{D}_r^{(k)} Y^{(l)}_{t_j} \left\langle \mathds{1}_{\Delta_j} (\cdot), \mathcal{D}^{(l)}_{\cdot} Y^{(k)}_{t_i} \right\rangle_{\mathcal{H}_1},
\end{align*}
with $R(\Delta_i, \mathrm{d}r)$ denoting $R(t_{i+1}, \mathrm{d}r) - R(t_i, \mathrm{d}r)$, cf. \eqref{2dInc}, we obtain
\begin{align} \label{traceComp}
\begin{split}
&\exptn{\mathrm{trace} \left( \mathcal{D} Y^{\pi} \circ \mathcal{D} Y^{\pi} \right)} \\
&\qquad= \exptn{\sum_{i, j} \sum_{k,l=1}^d \left\langle \mathcal{D}_{\cdot}^{(k)} Y^{(l)}_{t_j}, \mathds{1}_{\Delta_i} (\cdot) \right\rangle_{\mathcal{H}_1} \left\langle \mathcal{D}^{(l)}_{\cdot} Y^{(k)}_{t_i}, \mathds{1}_{\Delta_j} (\cdot) \right\rangle_{\mathcal{H}_1} } \\
&\qquad= \exptn{\sum_{i, j} \sum_{k,l=1}^d \int_0^T \mathcal{D}^{(k)}_r Y^{(l)}_{t_j} R(\Delta_i, \mathrm{d} r) \int_0^T \mathcal{D}_q^{(l)} Y^{(k)}_{t_i} R(\Delta_j, \mathrm{d}q) } \\
&\qquad\rightarrow \exptn{\int_{[0, T]^4} \mathrm{tr} \left( \mathcal{D}_r Y_t \, \mathcal{D}_q Y_s \right) \wrt{R(s, r)} \wrt{R(t, q)} } \quad \mathrm{as} \; \norm{\pi} \rightarrow 0.
\end{split}
\end{align}
\end{proof}

\begin{remark}
In the case of standard Brownian motion, if we use the fact that $\,
\mathrm{d} R(s, t) = \delta_{\{s=t\}} \, \mathrm{d} s \, \mathrm{d} t$ in
Theorem \ref{isometry1}, we recover the usual It\^{o}-Skorohod isometry
\begin{align*}
\mathbb{E} \left[  \delta^{X} (Y)^{2}\right]   & = \mathbb{E} \left[  \int%
_{0}^{T} \left|  Y_{t}\right| ^{2} \, \mathrm{d} t\right]  + \mathbb{E}
\left[  \int_{[0, T]^{2}} \mathrm{tr} \left(  \mathcal{D}_{t} Y_{s} \,
\mathcal{D}_{s} Y_{t} \right)  \, \mathrm{d} s \, \mathrm{d} t\right] .
\end{align*}

\end{remark}

\subsection{Riemann sum approximation to the Skorohod integral}

\begin{proposition}
\label{skorohodLimit2ndLevel} Let $X$ be a continuous, centered Gaussian
process in $\mathbb{R}^{d}$ with i.i.d. components, and assume that its
continuous covariance function satisfies $\left\|  R\right\| _{\rho-var; [0,
T]^{2}} < \infty$ for some $\rho\in\left[ 1, \frac{3}{2} \right) $. For $p
\in[1, 3)$, let $\mathbf{X} \in\mathcal{C}^{0, p-var} \left( [0, T];
G^{\lfloor p \rfloor} \left( \mathbb{R}^{d}\right) \right) $ denote the
geometric rough path constructed from the limit of the piecewise-linear
approximations of $X$.

Furthermore, let $Y \in\mathcal{C}^{p-var} \left(  [0, T]; \mathbb{R}^{d}
\right) $ denote the path-level solution to
\begin{align*}
\mathrm{d} Y_{t} = V(Y_{t}) \circ\mathrm{d} \mathbf{X}_{t}, \quad Y_{0} =
y_{0},
\end{align*}
where $V \in\mathcal{C}^{\lfloor p \rfloor+ 1}_{b} \left(  \mathbb{R}^{d};
\mathbb{R}^{d} \otimes\mathbb{R}^{d} \right) $. Then $Y \in\mathbb{D}^{1,2}
(\mathcal{H}_{1}^{d})$ and
\begin{align}
\label{skorohodLimit1}\int_{0}^{T} Y_{r} \, \mathrm{d} X_{r}  & =
\lim_{\left\|  \pi= \{ r_{i} \} \right\|  \rightarrow0} \sum_{i} \left[
Y_{r_{i}} \left(  X_{r_{i}, r_{i+1}} \right)  - \int_{0}^{r_{i}} \mathrm{tr}
\, \left[  J^{\mathbf{X}}_{r_{i} \leftarrow s} V(Y_{s}) \right]  \,
R(\Delta_{i}, \, \mathrm{d} s) \right] ,
\end{align}
where the limit is taken in $L^{2}\left( \Omega\right) $.
\end{proposition}

\begin{proof}
Using integration-by-parts, we have
\begin{align*}
\delta^X(Y^{\pi})
&= \sum_i \left[ \left\langle Y_{r_i}, \, X_{r_i, r_{i+1}} \right\rangle_{\mathbb{R}^d} - \sum_{k=1}^d \left\langle \mathcal{D}_{\cdot} Y_{r_i}^{(k)}, \mathds{1}^{(k)}_{[r_i, r_{i+1})} (\cdot) \right\rangle_{\mathcal{H}_1^d} \right] \\
&= \sum_i \left[ \left\langle Y_{r_i}, \, X_{r_i, r_{i+1}} \right\rangle_{\mathbb{R}^d} - \int_0^{r_i} \mathrm{tr} \, \left[ J^{\mathbf{X}}_{r_i \leftarrow s} V(Y_s) \right] \, R(\Delta_i, \wrt{s}) \right].
\end{align*}
From Propositions \ref{W1eqH1} and \ref{H1tensorH1equiv1}, we have $Y \in \mathcal{H}^d_1$ and \\
$\mathcal{D} Y = \mathds{1}_{[0, t)} (s) J^{\mathbf{X}}_{t \leftarrow s} V(Y_s) \in \mathcal{H}_1^d \otimes \mathcal{H}_1^d$ almost surely. So in light of Theorem \ref{isometry1}, we need to show that
\begin{align} \label{estimate1}
\exptn{\int_{[0, T]^2} \left\langle Y^{\pi}_s - Y_s, \, Y^{\pi}_t - Y_t \right\rangle_{\mathbb{R}^{d}} \wrt{R(s,t)}} \rightarrow 0,
\end{align}
and
\begin{align} \label{estimate2}
\exptn{\int_{[0, T]^2} \left( \int_{[0, T]^2} \left\langle \mathcal{D}_u \left( Y^{\pi}_s - Y_s \right), \, \mathcal{D}_v \left( Y^{\pi}_t - Y_t \right) \right\rangle_{\mathbb{R}^{d} \otimes \mathbb{R}^d} \wrt{R(u, v)} \right) \wrt{R(s, t)}} \rightarrow 0.
\end{align}
as $\norm{\pi} \rightarrow 0$. \par
For \eqref{estimate1}, we have
\begin{align} \label{1stTerm}
\begin{split}
&\exptn{\int_{[0, T]^2} \left\langle Y^{\pi}_s - Y_s, \, Y^{\pi}_t - Y_t \right\rangle_{\mathbb{R}^{d}} \wrt{R(s,t)}} \\
&\fqquad= \exptn{ \sum_{i, j} \int_{r_i}^{r_{i+1}} \int_{r_j}^{r_{j+1}} \left\langle Y_s - Y_{r_i}, \, Y_t - Y_{r_j} \right\rangle_{\mathbb{R}^{d}} \wrt{R(s, t)}},
\end{split}
\end{align}
and from Lemma \ref{keyProp1} with $\omega$ and $\theta$ defined as in \eqref{omegaDefn},
\begin{align*}
&\sum_{i, j} \left| \int_{r_i}^{r_{i+1}} \int_{r_j}^{r_{j+1}} \left\langle Y_s - Y_{r_i}, \, Y_t - Y_{r_j} \right\rangle_{\mathbb{R}^{d}} \wrt{R(s, t)} \right| \\
&\fqquad\leq C_{p, \rho} \sum_{i, j} \omega^{\theta} ([r_i, r_{i+1}] \times [r_j, r_{j+1}]),
\end{align*}
which tends to zero almost surely as the mesh of the partition goes to zero and is also bounded above uniformly for all partitions by the random variable (up to multiplication by a non-random constant)
\begin{align*}
\norm{Y}_{p-var; [0, T]}^2 \norm{R}_{\rho-var; [0, T]^2}.
\end{align*}
This is in $L^1(\Omega)$ by Theorem \ref{YBound}, and thus the limit of \eqref{1stTerm} vanishes by dominated convergence theorem. \par
We will use Proposition \ref{keyProp2} to show \eqref{estimate2}. We have
\begin{align}
\label{estimate2a}
\begin{split}
&\int_{[0, T]^2} \left( \int_{[0, T]^2} \left\langle \mathcal{D}_u \left( Y^{\pi}_s - Y_s \right), \, \mathcal{D}_v \left( Y^{\pi}_t - Y_t \right) \right\rangle_{\mathbb{R}^{d} \otimes \mathbb{R}^d} \wrt{R(u, v)} \right) \wrt{R(s, t)} \\
&= \sum_{i,j} \int_{U_{i,j}} \int_{[0, T]^2} \left\langle \mathcal{D}_u Y_s - \mathcal{D}_u Y_{r_i}, \mathcal{D}_v Y_t - \mathcal{D}_v Y_{r_j} \right\rangle_{\mathbb{R}^{d} \otimes \mathbb{R}^d} \wrt{R\left( u, v \right)}  \wrt{R\left( s, t\right)}.
\end{split}
\end{align}
where $U_{i,j} := [r_i, r_{i+1}] \times [r_j, r_{j + 1}]$.
With $g(s, t) := \mathcal{D}_s Y_t = \mathds{1}_{[0, t)} (s) \tilde{g}_1(t) \tilde{g}_2(s)$, where
\begin{align*}
\tilde{g}_1(t) := J_t^{\mathbf{X}}, \quad \tilde{g}_2 (s) := \left( J_s^{\mathbf{X}} \right)^{-1} V\left( Y_s \right),
\end{align*}
we see that $g$ satisfies the conditions of Proposition \ref{keyProp2} almost surely.
Hence, from \eqref{4Dvanish}, the expression in \eqref{estimate2a} vanishes almost surely as the mesh of the partition goes to zero. \par
Furthermore, it is bounded above uniformly for all partitions by the random variable (up to multiplication by a non-random constant)
\begin{align} \label{integrableRV}
\norm{\tilde{g}_1}_{\mathcal{V}^p; [0, T]}^2 \norm{\tilde{g}_2}_{\mathcal{V}^p; [0, T]}^2 \norm{R}^2_{\rho-var; [0, T]^2}.
\end{align}
As with the case of $\norm{Y}_{p-var}$, $\norm{J^{\mathbf{X}}}_{p-var}$ and $\norm{\left( J^{\mathbf{X}} \right)^{-1}}_{p-var}$ have finite moments of all orders by Theorem \ref{JBoundThm}. This in conjunction with the fact that $V(Y)$ is bounded almost surely ensures that \eqref{integrableRV} is integrable, and thus we can apply dominated convergence theorem again to complete the proof.
\end{proof}

\section{Appending the Riemann sum approximation to the Skorohod Integral}

The main purpose of this section is to show that the usual Riemann-sum
approximation to the Skorohod integral can be augmented with suitably
corrected second-level rough path terms which vanish in $L^{2}(\Omega)$ as the
mesh of the partition goes to zero.

We will use $\pi(n) := \left\{  t^{n}_{i} \right\} $ to denote the $n^{th}$
dyadic partition of $[0, T]$, i.e. $t^{n}_{i} = \frac{iT}{2^{n}}$ for $i = 0,
\ldots, 2^{n}$, and $\Delta^{n}_{i}$ to denote the interval $\left[ t^{n}_{i},
t^{n}_{i+1}\right] $.

In addition, $\rho^{\prime}$ will denote the H\"{o}lder conjugate of $\rho$,
i.e. $\frac{1}{\rho} + \frac{1}{\rho^{\prime}} = 1$.

\begin{proposition}
\label{2ndlevel} Let $X$ be a continuous, centered Gaussian process in
$\mathbb{R}^{d}$ with i.i.d. components, and for $p \in[2, 4)$, let
$\mathbf{X} \in\mathcal{C}^{0, p-var} \left( [0, T]; G^{\lfloor p \rfloor}
\left( \mathbb{R}^{d}\right) \right) $ denote the geometric rough path
constructed from the limit of the piecewise-linear approximations of $X$.

Let $\rho$ and $q$ be such that $\rho\in\left[ 1, 2 \right) $ and $\frac{1}{p}
+ \frac{1}{q} > 1$. We assume that the covariance function of $X$ satisfies

\begin{enumerate}
[(a)]

\item $\left\|  R\right\| _{\rho-var; [0, T]^{2}} < \infty$,

\item $\left\|  R(t, \cdot) - R (s, \cdot)\right\| _{q-var; [0, T]} \leq C
\left|  t - s\right| ^{\frac{1}{\rho}}$, for all $s, t \in[0, T]$.
\end{enumerate}

Now let $\psi: \Omega\times[0, T] \rightarrow\mathbb{R}^{d} \otimes
\mathbb{R}^{d}$ be a stochastic process satisfying $\displaystyle \psi_{t} =
\sum_{a, b = 1}^{d} \psi_{t}^{(a, b)} \mathrm{d} e_{a} \otimes\mathrm{d} e_{b}
\in\mathbb{D}^{4,2} (\mathbb{R}^{d} \otimes\mathbb{R}^{d})$ for all $t \in[0,
T]$. Furthermore, assume there exists $C < \infty$ such that for all $s, t
\in[0, T]$ and $a, b = 1, \ldots, d$, we have
\begin{align}
\left|  \mathbb{E} \left[  \psi^{(a, b)}_{s} \psi^{(a, b)}_{t}\right] \right|
\leq C,
\end{align}
and for $k = 2, 4$, we have
\begin{align}
\label{propCond}\left|  \mathbb{E} \left[  \mathcal{D}^{k}_{h_{1}, \ldots,
h_{k}} \left(  \psi^{(a, b)}_{s} \psi^{(a, b)}_{t} \right)  \right] \right|
\leq C \prod_{i=1}^{k} \left\|  \Phi(h_{i})\right\| _{q-var; [0, T]},
\end{align}
for all $h_{1}, \ldots, h_{k} \in\mathcal{H}_{1}^{d}$.

Then
\begin{align}
\lim_{n \rightarrow\infty} \left\|  \sum_{i=0}^{2^{n} - 1} \psi_{t^{n}_{i}}
\left(  \mathbf{X}^{2}_{t^{n}_{i}, t^{n}_{i+1}} - \frac{1}{2} \sigma^{2}
\left(  t^{n}_{i}, t^{n}_{i+1} \right)  \mathcal{I}_{d} \right)  \right\|
_{L^{2}(\Omega)} = 0.
\end{align}

\end{proposition}

\begin{proof}
First note that
\begin{align} \label{mainQuant1}
\begin{split}
&\left\| \sum_{i=0}^{2^n - 1} \psi_{t^n_i} \left( \mathbf{X}^2_{t^n_i, t^n_{i+1}} -  \frac{1}{2} \sigma^2 \left( t^n_i, t^n_{i+1} \right) \mathcal{I}_d \right) \right\|_{L^2(\Omega)} \\
&\leq \left\| \sum_{i=0}^{2^n - 1} \psi_{t^n_i} \left( \left(\mathbf{X}^2_{t^n_i, t^n_{i+1}} \right)^S -  \frac{1}{2} \sigma^2 \left( t^n_i, t^n_{i+1} \right) \mathcal{I}_d \right) \right\|_{L^2(\Omega)}
+ \left\| \sum_{i=0}^{2^n - 1} \psi_{t^n_i} \left( \left( \mathbf{X}^2_{t^n_i, t^n_{i+1}} \right)^A \right) \right\|_{L^2(\Omega)},
\end{split}
\end{align}
where $\left( \mathbf{X}^2 \right)^S$ denotes the symmetric part of $\mathbf{X}^2$ and $\left( \mathbf{X}^2 \right)^A$ denotes the anti-symmetric part. The two parts will be tackled separately, and since
\begin{align*}
\norm{\sum_{i=0}^{2^n - 1} \psi_{t^n_i} \left( \left( \mathbf{X}^2_{t^n_i, t^n_{i+1}} \right)^A \right)}_{L^2(\Omega)}
\leq \sum_{a, b = 1}^d \norm{ \sum_{i=0}^{2^n - 1} \psi_{t^n_i}^{(a, b)} \left( \left( \mathbf{X}^2_{t^n_i, t^n_{i+1}} \right)^A \right)^{(a, b)} }_{L^2(\Omega)},
\end{align*}
and similarly for the symmetric part, we can study the convergence of each fixed $(a, b)^{th}$ tensor component individually. For simplicity, we will henceforth suppress the notation for the component in the superscript of $\psi$. \par
Let $h_1, h_2, g_1, g_2 \in \mathcal{H}_1^d$ be such that $\left\langle h_i, g_j \right\rangle_{\mathcal{H}_1^d} = 0$ for all $i, j = 1, 2$. Then from the product formula \eqref{productFormula}, we have the following identities
\begin{subequations} \label{hermiteIdentities}
\begin{align}
\label{identity1}
&I_1 (h_1) I_1 (h_2) = I_2 (h_1 \tilde{\otimes} h_2) + \left\langle h_1, h_2 \right\rangle_{\mathcal{H}_1^d}, \\
\label{identity2}
&I_2(h_1 \tilde{\otimes} h_2) I_2 (g_1 \tilde{\otimes} g_2)
= I_4 (h_1 \tilde{\otimes} h_2 \tilde{\otimes} g_1 \tilde{\otimes} g_2), \\
\mathrm{and} \nonumber \\
\label{identity3}
&I_2(h_1 \otimes h_1) I_2 (h_2 \otimes h_2) \\
&\qquad= I_4 (h_1 \otimes h_1 \tilde{\otimes} h_2 \otimes h_2) + 4 I_2(h_1 \tilde{\otimes} h_2) \left\langle h_1, h_2 \right\rangle_{\mathcal{H}_1^d} + 2 \left\langle h_1, h_2 \right\rangle_{\mathcal{H}_1^d}^2.
\end{align}
\end{subequations}
Following \cite{nnt2010}, the idea of the proof is to rewrite \eqref{mainQuant1} in such a way that the summands take the form
\begin{align*}
\exptn{\psi_{t^n_i} \psi_{t^n_j} X^{(a)}_{u_1, u_2} X^{(b)}_{u_3, u_4} X^{(a)}_{v_1, v_2} X^{(b)}_{v_3, v_4}},
\end{align*}
where $[u_1, u_2], [u_3, u_4] \subset \left[ t^n_i, t^n_{i+1} \right]$ and $[v_1, v_2], [v_3, v_4] \subset \left[t^n_j, t^n_{j+1}\right]$, or
\begin{align*}
\exptn{\psi_{t^n_i} \psi_{t^n_j} \left( \left( X^{(a)}_{t^n_i, t^n_{i+1}} \right)^2 - \sigma^2 \left( t^n_i, t^n_{i+1} \right) \right) \left( \left( X^{(a)}_{t^n_j, t^n_{j+1}} \right)^2 - \sigma^2 \left( t^n_j, t^n_{j+1} \right) \right) }
\end{align*}
as appearing in the symmetric part.
After applying the identities in \eqref{hermiteIdentities} and using the duality formula \eqref{dualityFormula}, \eqref{propCond} will be used to bound the summands. \par
(a) For the symmetric part of the second level rough path, we have
\begin{align*}
\left( \left( \mathbf{X}^2_{t^n_i, t^n_{i+1}} \right)^S - \frac{1}{2} \sigma^2 \left( t^n_i, t^n_{i+1} \right) \mathcal{I}_d \right)^{(a, b)} = \frac{1}{2} \left( X^{(a)}_{t^n_i, t^n_{i+1}} X^{(b)}_{t^n_i, t^n_{i+1}} - \delta_{ab} \, \sigma^2 \left( t^n_i, t^n_{i+1} \right) \right).
\end{align*}
In the case where $a = b$, we need to estimate
\begin{align} \label{threeTermEstimate}
\begin{split}
&\exptn{\left( \sum_{i=0}^{2^n-1} \psi_{t^n_i} \left( \left( X^{(a)}_{t_i, t_{i+1}} \right)^2 - \sigma^2 \left( t^n_i, t^n_{i+1} \right) \right) \right)^2} \\
&\qquad = \sum_{i, j = 0}^{2^n-1} \exptn{ \psi_{t^n_i} \psi_{t^n_j} \left( \left( X^{(a)}_{t^n_i, t^n_{i+1}} \right)^2 - \sigma^2 \left( t^n_i, t^n_{i+1} \right) \right) \left( \left( X^{(a)}_{t^n_j, t^n_{j+1}} \right)^2 - \sigma^2 \left( t^n_j, t^n_{j+1} \right) \right) } \\
&\qquad = \sum_{i, j = 0}^{2^n-1} \exptn{\psi_{t_i^n} \psi_{t_j^n} I_2 \left( \mathds{1}^{(a)}_{\Delta^n_i} \otimes \mathds{1}^{(a)}_{\Delta^n_i} \right) I_2 \left( \mathds{1}^{(a)}_{\Delta^n_j} \otimes \mathds{1}^{(a)}_{\Delta^n_j} \right) },
\end{split}
\end{align}
where the last line follows from \eqref{identity1}.
Using \eqref{identity3} with $h_1 = \mathds{1}^{(a)}_{\Delta^n_i}(\cdot)$ and $h_2 = \mathds{1}^{(a)}_{\Delta^n_j}(\cdot)$ and applying the duality formula \eqref{dualityFormula}, the expression above is equal to
\begin{align*}
\sum_{i, j = 0}^{2^n-1} \exptn{\mathcal{D}^4_{ h_1, h_1, h_2, h_2} \psi_{t^n_i} \psi_{t^n_j}}
+ 4 \exptn{\mathcal{D}^2_{h_1, h_2} \psi_{t^n_i} \psi_{t^n_j} } &R\begin{pmatrix}
t^n_i & t^n_{i+1} \\
t^n_j & t^n_{j+1}
\end{pmatrix} \\
&+ 2\exptn{\psi_{t^n_i} \psi_{t^n_j}} R \begin{pmatrix}
t^n_i & t^n_{i+1} \\
t^n_j & t^n_{j+1}
\end{pmatrix}^2.
\end{align*}
For the first term, we have
\begin{align} \label{sym1}
\begin{split}
\sum_{i, j = 0}^{2^n-1} \exptn{\mathcal{D}^4_{h_1, h_1, h_2, h_2} \psi_{t^n_i} \psi_{t^n_j} }
&\leq C \sum_{i, j = 0}^{2^n-1} \norm{R \left( \Delta^n_i, \cdot \right)}_{q-var; [0, T]}^2 \norm{R \left( \Delta^n_j, \cdot \right) }_{q-var; [0, T]}^2 \\
&\leq \frac{C}{2^{2n \left( \frac{2}{\rho} - 1 \right)}} \rightarrow 0
\end{split}
\end{align}
since $\rho < 2$. \par
For the second term, we have
\begin{align} \label{sym2}
\begin{split}
&\sum_{i, j = 0}^{2^n-1} \exptn{\mathcal{D}^2_{h_1, h_2} \psi_{t^n_i} \psi_{t^n_j} } R \begin{pmatrix}
t^n_i & t^n_{i+1} \\
t^n_j & t^n_{j+1}
\end{pmatrix} \\
&\qquad\qquad\leq \left( \sum_{i, j = 0}^{2^n-1} \exptn{\mathcal{D}^2_{h_1, h_2} \psi_{t^n_i} \psi_{t^n_j} }^{\rho'} \right)^{\frac{1}{\rho'}} \left( \sum_{i,j=0}^{2^n-1} \left| R \begin{pmatrix}
t^n_i & t^n_{i+1} \\
t^n_j & t^n_{j+1}
\end{pmatrix} \right|^{\rho} \right)^{\frac{1}{\rho}} \\
&\qquad\qquad\leq C \, 2^{-2n \left( \frac{1}{\rho} - \frac{1}{\rho'} \right)} \norm{R}_{\rho-var; [0, T]^2}
= C \, 2^{-2n \left( \frac{2}{\rho} - 1 \right)} \norm{R}_{\rho-var; [0, T]^2},
\end{split}
\end{align}
which also vanishes as $n$ tends to infinity. \par
For the third term, we have
\begin{align} \label{sym3}
\begin{split}
\sum_{i, j=0}^{2^n-1} \exptn{\psi_{t^n_i} \psi_{t^n_j}} R \begin{pmatrix}
t^n_i & t^n_{i+1} \\
t^n_j & t^n_{j+1}
\end{pmatrix}^2
&\leq C \sum_{i,j=0}^{2^n-1} \abs{R \begin{pmatrix}
t_i^n & t^n_{i+1} \\
t_j^n & t^n_{j+1}
\end{pmatrix}}^{2 - \rho} \abs{R \begin{pmatrix}
t_i^n & t^n_{i+1} \\
t_j^n & t^n_{j+1}
\end{pmatrix}}^{\rho} \\
&\leq \frac{C}{2^{\frac{n (2 - \rho)}{\rho}}} \norm{R}^{\rho}_{\rho-var; [0, T]^2},
\end{split}
\end{align}
which vanishes as $n \rightarrow \infty$ since $\rho < 2$. \par
In the case where $a \neq b$, we let
\begin{align*}
h_1 := \mathds{1}^{(a)}_{\Delta^n_i} (\cdot), \quad h_2 := \mathds{1}^{(a)}_{\Delta^n_j} (\cdot), \quad
g_1 := \mathds{1}^{(b)}_{\Delta^n_i} (\cdot), \quad \mathrm{and} \:\: g_2 := \mathds{1}^{(b)}_{\Delta^n_j} (\cdot).
\end{align*}
We obtain
\begin{align*}
&\exptn{\psi_{t^n_i} \psi_{t^n_j} X^{(a)}_{\Delta^n_i} X^{(b)}_{\Delta^n_i} X^{(a)}_{\Delta^n_j} X^{(b)}_{\Delta^n_j}} \\
&\qquad = \exptn{\psi_{t^n_i} \psi_{t^n_j} \left( I_2(h_1 \tilde{\otimes} h_2) + \left\langle h_1, h_2 \right\rangle_{\mathcal{H}_1^d} \right) \left( I_2(g_1 \tilde{\otimes} g_2) + \left\langle g_1, g_2 \right\rangle_{\mathcal{H}_1^d} \right)} \\
&\qquad= \exptn{\psi_{t^n_i} \psi_{t^n_j} I_4 (h_1 \tilde{\otimes} h_2 \tilde{\otimes} g_1 \tilde{\otimes} g_2)} + \exptn{\psi_{t^n_i} \psi_{t^n_j} I_2 (h_1 \tilde{\otimes} h_2)} \left\langle g_1, g_2 \right\rangle_{\mathcal{H}_1^d} \\
&\qquad \qquad + \exptn{\psi_{t^n_i} \psi_{t^n_j} I_2 (g_1 \tilde{\otimes} g_2)} \left\langle h_1, h_2 \right\rangle_{\mathcal{H}_1^d} + \exptn{\psi_{t^n_i} \psi_{t^n_j}} \left\langle h_1, h_2 \right\rangle_{\mathcal{H}_1^d} \left\langle g_1, g_2 \right\rangle_{\mathcal{H}_1^d} \\
&\qquad = \exptn{\mathcal{D}^4_{h_1, h_2, g_1, g_2} \psi_{t^n_i} \psi_{t^n_j}}
+ \exptn{\psi_{t^n_i} \psi_{t^n_j}} R \begin{pmatrix}
t^n_i & t^n_{i+1} \\
t^n_j & t^n_{j+1}
\end{pmatrix}^2 \\
&\qquad \qquad + \left( \exptn{\mathcal{D}^2_{h_1, h_2} \psi_{t^n_i} \psi_{t^n_j}} + \exptn{\mathcal{D}^2_{g_1, g_2} \psi_{t^n_i} \psi_{t^n_j}} \right) R \begin{pmatrix}
t^n_i & t^n_{i+1} \\
t^n_j & t^n_{j+1}
\end{pmatrix}.
\end{align*}
Similar to the case where $a = b$, the sum over all $i, j$ of the first, second and third terms in the above expression can be bounded by \eqref{sym1}, \eqref{sym3} and \eqref{sym2} respectively, and hence vanish as $n \rightarrow \infty$. \par
(b) We will now handle the anti-symmetric part. We will use $\left( \mathbf{X}^2_{s,t} \right)^A (\pi(k))$ to denote the L\'{e}vy area of $X^{\pi(k)}$, the piece-wise linear approximation of $X$ over $\pi(k)$, i.e.
\begin{align*}
\left( \mathbf{X}^2_{s,t} \right)^A (\pi(k)) = \pi_2 \left( \log \left( S_2 \left( X^{\pi(k)}\right)_{s,t} \right)\right),
\end{align*}
where $\pi_2$ denotes projection onto the second level. Next, we define
\begin{align*}
\left( \mathbf{X}^2_{s,t} \right)^A (\Delta_{l+1}) := \left( \mathbf{X}^2_{s,t} \right)^A \left( \pi (l+1) \right)  - \left( \mathbf{X}^2_{s,t} \right)^A \left( \pi (l) \right),
\end{align*}
and noticing that $\left( \mathbf{X}^2_{t_i^n, t_{i+1}^n} \right)^A \left( \pi(n) \right) = 0$, we can use Theorem \ref{gaussianRP} to see that
\begin{align*}
\left( \mathbf{X}^2_{t_i^n, t_{i+1}^n} \right)^A = \lim_{m\rightarrow\infty}\sum_{k=1}^m
\left( \mathbf{X}^2_{t_i^n, t_{i+1}^n} \right)^A \left( \Delta_{n+k} \right)  \text{ for every }n\in \mathbb{N} \; \mathrm{and} \; i=0, 1, \ldots, 2^n - 1,
\end{align*}
where the limit is taken in $L^2(\Omega)$. \par
We want to show that
\begin{align*}
\norm{\sum_{i=0}^{2^n - 1} \psi_{t^n_i} \left( \left( \mathbf{X}^2_{t_i^n, t_{i+1}^n} \right)^A (\pi (n+m)) \right)^{(a, b)}}_{L^2(\Omega)} \rightarrow 0
\end{align*}
uniformly for all $m$ as $n \rightarrow \infty$. To begin, let
\begin{align} \label{sDefn}
s_u^{k, i} := t_i^n + \frac{u}{2^{n+k}} = t_{u + i 2^k}^{n+k},
\end{align}
and we will denote the intervals
\begin{align} \label{LRDefn}
\begin{split}
&\Delta_{u^L}^i := \left[ s^{k, i}_{2u}, s^{k, i}_{2u+1} \right], \quad \Delta_{u^R}^i := \left[ s^{k, i}_{2u+1}, s^{k, i}_{2u+2} \right], \\
&\Delta_u^i := \Delta_{u^L}^i \cup \Delta_{u^R}^i \subseteq \left[ t^n_i, t^n_{i+1} \right], \quad \forall u = 0, \ldots, 2^{k-1} - 1.
\end{split}
\end{align}
Note that we suppress the dependence on $k$ and $n$ in the notation for the variables on the left, and we will also use $X_{[s,t]}$ and $X_{s,t}$ ($s, t \in [0, T]$) interchangeably.
Continuing, we have
\begin{align*}
&\bigotimes_{u = 0}^{2^{k-1} - 1} \exp \left(X_{\Delta^i_{u^L}} \right) \otimes \exp \left(X_{\Delta^i_{u^R}} \right) - \bigotimes_{u = 0}^{2^{k-1} - 1} \exp\left(X_{\Delta^i_u}\right) \\
&\qquad \qquad= \bigotimes_{u = 0}^{2^{k-1} - 1} \left(1, X_{\Delta_u^i}, \frac{\left(X_{\Delta_u^i}\right)^{\otimes 2}}{2} \right) + \left( 0, 0, \frac{1}{2} \left[X_{\Delta_{u^L}^i}, X_{\Delta_{u^R}^i} \right] \right) \\
&\fqquad- \bigotimes_{u = 0}^{2^{k-1} - 1} \left(1, X_{\Delta_u^i}, \frac{\left(X_{\Delta_u^i}\right)^{\otimes 2}}{2} \right) \\
&\qquad \qquad= \sum_{u=0}^{2^{k-1} - 1} \left(0, 0, \frac{1}{2} \left[ X_{\Delta_{u^L}^i}, X_{\Delta_{u^R}^i} \right] \right),
\end{align*}
which means that
\begin{align*}
\left( \mathbf{X}^2_{t_i^n, t_{i+1}^n} \right)^A \left( \Delta_{n+k} \right)
= \sum_{u=0}^{2^{k-1}-1} \frac{1}{2} \left[ X_{\Delta_{u^L}^i}, X_{\Delta_{u^R}^i} \right]
\end{align*}
since only anti-symmetric terms are left in the difference. \par
Thus, we obtain
\begin{align} \label{2ndQuant}
\begin{split}
&\exptn{\left( \sum_{i=0}^{2^n - 1} \psi_{t^n_i} \left( \left( \mathbf{X}^2_{t_i^n, t_{i+1}^n} \right)^A (\pi (n+m)) \right)^{(a, b)} \right)^2} \\
&\quad= \exptn{\left( \sum_{i=0}^{2^n - 1} \psi_{t^n_i} \sum_{k=1}^m \left( \left( \mathbf{X}^2_{t_i^n, t_{i+1}^n} \right)^A (\Delta_{n+k}) \right)^{(a, b)} \right)^2} \\
&\quad= \sum_{i, j=0}^{2^n - 1} \exptn{ \psi_{t^n_i} \psi_{t^n_j} \sum_{k=1}^m \left( \left( \mathbf{X}^2_{t_i^n, t_{i+1}^n} \right)^A \left( \Delta_{n+k} \right) \right)^{(a, b)}  \sum_{l=1}^m \left( \left( \mathbf{X}^2_{t_j^n, t_{j+1}^n} \right)^A \left( \Delta_{n+l} \right) \right)^{(a, b)}} \\
&\quad= \frac{1}{4} \sum_{i, j=0}^{2^n - 1}\sum_{k,l=1}^m \sum_{u, v} \exptn{ \psi_{t^n_i} \psi_{t^n_j} \Big[ X_{\Delta_{u^L}^i}, X_{\Delta_{u^R}^i} \Big]^{(a, b)} \Big[X_{\Delta_{v^L}^j}, X_{\Delta_{v^R}^j} \Big]^{(a, b)} }.
\end{split}
\end{align}
Since the $(a, b)^{th}$ entry of $\left[ X_{\Delta_{u^L}^i}, X_{\Delta_{u^R}^i} \right] \in so(d)$, where $a \neq b$, is given by
\begin{align*}
X^{(a)}_{\Delta_{u^L}^i} X^{(b)}_{\Delta_{u^R}^i} - X^{(b)}_{\Delta_{u^L}^i} X^{(a)}_{\Delta_{u^R}^i},
\end{align*}
each summand in the last line of \eqref{2ndQuant} is of the form
\begin{align} \label{summands}
\begin{split}
&\exptn{\psi_{t^n_i} \psi_{t^n_j} X^{(a)}_{\Delta_{u^L}^i} X^{(b)}_{\Delta_{u^R}^i} X^{(a)}_{\Delta_{v^L}^j} X^{(b)}_{\Delta_{v^R}^j}}
- \exptn{\psi_{t^n_i} \psi_{t^n_j} X^{(a)}_{\Delta_{u^L}^i} X^{(b)}_{\Delta_{u^R}^i} X^{(b)}_{\Delta_{v^L}^j} X^{(a)}_{\Delta_{v^R}^j}} \\
&\quad - \exptn{\psi_{t^n_i} \psi_{t^n_j} X^{(b)}_{\Delta_{u^L}^i} X^{(a)}_{\Delta_{u^R}^i} X^{(a)}_{\Delta_{v^L}^j} X^{(b)}_{\Delta_{v^R}^j}}
\quad + \exptn{\psi_{t^n_i} \psi_{t^n_j} X^{(b)}_{\Delta_{u^L}^i} X^{(a)}_{\Delta_{u^R}^i} X^{(b)}_{\Delta_{v^L}^j} X^{(a)}_{\Delta_{v^R}^j}}.
\end{split}
\end{align}
Proceeding, we will use the first term in the above expression for the proof and omit the other terms as the result remains the same with trivial modifications to the notation. We first denote
\begin{align} \label{RSum4Quad}
\begin{split}
R_{\Delta_u^i \times \Delta_v^j} : = \abs{R \begin{pmatrix}
s^{k, i}_{2u} & s^{k, i}_{2u+1} \\
s^{l, j}_{2v} & s^{l, j}_{2v+1}
\end{pmatrix}} +
&\abs{R \begin{pmatrix}
s^{k, i}_{2u+1} & s^{k, i}_{2u+2} \\
s^{l, j}_{2v} & s^{l, j}_{2v+1}
\end{pmatrix}} \\
&+ \abs{R \begin{pmatrix}
s^{k, i}_{2u} & s^{k, i}_{2u+1} \\
s^{l, j}_{2v+1} & s^{l, j}_{2v+2}
\end{pmatrix}} +
\abs{R \begin{pmatrix}
s^{k, i}_{2u+1} & s^{k, i}_{2u+2} \\
s^{l, j}_{2v+1} & s^{l, j}_{2v+2}
\end{pmatrix}},
\end{split}
\end{align}
and note that
\begin{align*}
\sum_{i, j=0}^{2^n-1} \sum_{u, v}
R_{\Delta_u^i \times \Delta_v^j}^{\rho}
\leq 4^{\rho} \norm{R}_{\rho-var; [0, T]^2} ^{\rho}.
\end{align*}
Next we let
\begin{align*}
h_1 := \mathds{1}^{(a)}_{\Delta_{u^L}^i} (\cdot), \quad h_2 := \mathds{1}^{(a)}_{\Delta_{v^L}^j} (\cdot), \quad
g_1 := \mathds{1}^{(b)}_{\Delta_{u^R}^i} (\cdot), \quad \mathrm{and} \:\: g_2 := \mathds{1}^{(b)}_{\Delta_{v^R}^j} (\cdot),
\end{align*}
and applying \eqref{identity2}, we get
\begin{align*}
&\exptn{\psi_{t^n_i} \psi_{t^n_j} X^{(a)}_{\Delta_{u^L}^i} X^{(b)}_{\Delta_{u^R}^i} X^{(a)}_{\Delta_{v^L}^j} X^{(b)}_{\Delta_{v^R}^j}} \\
&\qquad = \exptn{\psi_{t^n_i} \psi_{t^n_j} \left( I_2(h_1 \tilde{\otimes} h_2) + \left\langle h_1, h_2 \right\rangle_{\mathcal{H}_1^d} \right) \left( I_2(g_1 \tilde{\otimes} g_2) + \left\langle g_1, g_2 \right\rangle_{\mathcal{H}_1^d} \right)} \\
&\qquad = \exptn{\psi_{t^n_i} \psi_{t^n_j} I_4 (h_1 \tilde{\otimes} h_2 \tilde{\otimes} g_1 \tilde{\otimes} g_2)} + \exptn{\psi_{t^n_i} \psi_{t^n_j} I_2 (h_1 \tilde{\otimes} h_2)} \left\langle g_1, g_2 \right\rangle_{\mathcal{H}_1^d} \\
&\qquad \qquad + \exptn{\psi_{t^n_i} \psi_{t^n_j} I_2 (g_1 \tilde{\otimes} g_2)} \left\langle h_1, h_2 \right\rangle_{\mathcal{H}_1^d} + \exptn{\psi_{t^n_i} \psi_{t^n_j}} \left\langle h_1, h_2 \right\rangle_{\mathcal{H}_1^d} \left\langle g_1, g_2 \right\rangle_{\mathcal{H}_1^d} \\
&\qquad =: A_1 + A_2 + A_3 + A_4.
\end{align*}
\begin{enumerate}[(i)]
\item
Terms of type $A_1$:
\begin{align*}
&\abs{\exptn{ \mathcal{D}_{h_1, h_2, g_1, g_2}^4 \left( \psi_{t^n_i} \psi_{t_j^n} \right) }} \\
&\leq C \norm{\Phi(h_1)}_{q-var; \left[  0,T\right]} \norm{\Phi(h_2)}_{q-var;\left[ 0, T\right] } \norm{\Phi(g_1)}_{q-var;\left[ 0,T \right] } \norm{\Phi(g_2)}_{q-var;\left[ 0,T\right] } \\
&\leq C \norm{ R\left(  s_{2u+1}^{k,i},\cdot\right) - R\left(s_{2u}^{k, i},\cdot\right) }_{q-var;\left[  0,T\right] } \norm{ R\left( s_{2v+1}^{l, j},\cdot\right) - R\left( s_{2v}^{l, j},\cdot\right) }_{q-var;\left[  0,T\right] } \\
&\qquad \times \norm{ R\left( s_{2u+2}^{k, i},\cdot\right) -R\left(s_{2u+1}^{k, i},\cdot\right) }_{q-var;\left[  0,T\right] } \norm{ R\left(  s_{2v + 2}^{l, j},\cdot\right) - R\left(  s_{2v+1}^{l, j},\cdot\right) }_{q-var;\left[  0,T\right] } \\
&\leq C \, 2^{\frac{- 2\left( n+k\right)}{\rho}} 2^{\frac{-2\left( n+l \right)}{\rho} },
\end{align*}
and thus we have
\begin{align*}
\sum_{i,j=0}^{2^n-1}\sum_{k,l=1}^m \sum_{u, v} \abs{ \exptn{\mathcal{D}_{h_1, h_2, g_1, g_2}^4 \left( \psi_{t_i^n} \psi_{t_j^n} \right) } }
&\leq C \, 2^{-2n \left( \frac{2}{\rho} - 1 \right)} \sum_{k, l=1}^{\infty} 2^{-k\left( \frac{2}{\rho} - 1 \right)} 2^{-l\left( \frac{2}{\rho} - 1 \right) } \\
&\leq C \, 2^{-2n\left( \frac{2}{\rho} - 1\right) } \rightarrow 0 \; \text{as} \; n \rightarrow\infty.
\end{align*}
\item
Terms of type $A_{2}$ and $A_{3}$: We only detail the argument for the $A_{2}$ terms; the $A_{3}$ terms can be dealt with in the same way.
Using H\"{o}lder's inequality and by exploiting the upper bound
\begin{align*}
\abs{ \exptn{\mathcal{D}_{h_1, h_2}^{2}\left( \psi_{t^n_i} \psi_{t^n_j} \right)  } }
\leq C \abs{ s_{2u + 1}^{k, i} - s_{2u}^{k, i} }^{\frac{1}{\rho}} \abs{ s_{2v + 1}^{l, j} - s_{2v}^{l, j} }^{\frac{1}{\rho}}
\leq C \, 2^{-\frac{( n+k)}{\rho}} 2^{-\frac{(n+l)}{\rho}},
\end{align*}
we obtain
\begin{align*}
&  \sum_{i, j=0}^{2^n - 1} \sum_{k, l=1}^m \sum_{u, v} \exptn{ \mathcal{D}_{h_1, h_2}^2 \left(  \psi_{t^n_i} \psi_{t_j^n} \right) }  \exptn{ X^{(b)}_{\Delta_{u^R}^i} X^{(b)}_{\Delta_{v^R}^j} } \\
&\qquad \leq C \sum_{k, l=1}^m \left( \sum_{i,j=0}^{2^n - 1} \sum_{u, v} \abs{ \exptn{ \mathcal{D}_{h_1, h_2}^2 \left( \psi_{t^n_i} \psi_{t_j^n} \right) }}^{\rho'} \right)^{\frac{1}{\rho'}} \left( \sum_{i,j=0}^{2^n - 1} \sum_{u, v} R_{\Delta^i_u \times \Delta_v^j}^{\rho}\right)^{\frac{1}{\rho}} \\
&\qquad \leq C \sum_{k, l=1}^{\infty} 2^{-2n \left( \frac{1}{\rho} - \frac{1}{\rho'} \right)}2^{-k\left( \frac{1}{\rho} - \frac{1}{\rho'} \right)}2^{-l\left( \frac{1}{\rho} - \frac{1}{\rho'} \right) } \norm{R}_{\rho\text{-var};\left[  0,T\right]^2}.
\end{align*}
Since $\frac{1}{\rho} - \frac{1}{\rho'} = \frac{2}{\rho} - 1 > 0$, we may sum over $k$ and $l$ to get
\begin{align*}
\sum_{i, j=0}^{2^n - 1} \sum_{k, l=1}^m \sum_{u, v} \exptn{ \mathcal{D}_{h_1, h_2}^2 \left(  \psi_{t^n_i} \psi_{t_j^n} \right) }  \exptn{ X^{(b)}_{\Delta_{u^R}^i} X^{(b)}_{\Delta_{v^R}^j} }
\leq C \, 2^{-2n \left( \frac{2}{\rho} - 1 \right) } \norm{R}_{\rho\text{-var};\left[  0,T\right]^2},
\end{align*}
which tends to zero as $n\rightarrow\infty.$
\item
Terms of type $A_4$: We have
\begin{align*}
&\exptn{\psi_{t^n_i} \psi_{t^n_j}} \left\langle h_1, h_2 \right\rangle_{\mathcal{H}_1^d} \left\langle g_1, g_2 \right\rangle_{\mathcal{H}_1^d} \\
&\qquad= \exptn{ \psi_{t^n_i} \psi_{t^n_j} }
R \begin{pmatrix}
s^{k, i}_{2u} & s^{k, i}_{2u+1} \\
s^{l, j}_{2v} & s^{l, j}_{2v+1}
\end{pmatrix}
R \begin{pmatrix}
s^{k, i}_{2u + 1} & s^{k, i}_{2u+2} \\
s^{l, j}_{2v+1} & s^{l, j}_{2v+2}
\end{pmatrix} \\
&\qquad\leq C R_{\Delta^i_u \times \Delta_v^j}^2.
\end{align*}
Using the fact that
\begin{align*}
&R_{\Delta^i_u \times \Delta^j_v}
\leq 2\norm{R \left(s^{k,i}_{2u+1}, \cdot\right) - R\left(s^{k,i}_{2u}, \cdot\right)}_{q-var;[0,T]} \\
&\fqquad+ 2\norm{R \left(s^{k,i}_{2u+2}, \cdot\right) - R\left(s^{k,i}_{2u+1}, \cdot\right)}_{q-var;[0,T]} \\
&\mathrm{and} \\
&R_{\Delta^i_u \times \Delta^j_v}
\leq 2 \norm{R \left(s^{l,j}_{2v+1}, \cdot\right) - R\left(s^{l,j}_{2v}, \cdot\right)}_{q-var;[0,T]} \\
&\fqquad+ 2\norm{R \left(s^{l,j}_{2v+2}, \cdot\right) - R\left(s^{l,j}_{2v+1}, \cdot\right)}_{q-var;[0,T]},
\end{align*}
we have
\begin{align*}
\sum_{k,l=1}^m \sum_{i,j=0}^{2^n-1} \sum_{u,v}
R_{\Delta_u^i \times \Delta^j_v}^2
&\leq C \sum_{k,l=1}^m 2^{-(n+k)\frac{2-\rho}{2\rho}} 2^{-(n+l)\frac{2-\rho}{2\rho}} \sum_{i.j=0}^{2^n-1} \sum_{u,v}
R_{\Delta_u^i \times \Delta_v^j}^{\rho} \\
&\leq C \, 2^{-2n \left(\frac{1}{\rho} - \frac{1}{2}\right)} \sum_{k,l=1}^{\infty} 2^{-k \left(\frac{1}{\rho} - \frac{1}{2} \right)} 2^{-l \left(\frac{1}{\rho} - \frac{1}{2} \right)} \norm{R}_{\rho-var;[0,T]^2}^{\rho},
\end{align*}
which converges to $0$ since $\frac{1}{\rho} - \frac{1}{2} > 0$.
\end{enumerate}
\end{proof}

Given the preceding proposition, the following corollary is straightforward.

\begin{corollary}
\label{skorodhodLimitEnhanced} For $2 \leq p < 4$, let $Y \in\mathcal{C}
^{p-var} \left(  [0, T]; \mathbb{R}^{d} \right) $ denote the path-level
solution to
\begin{align*}
\mathrm{d} Y_{t} = V(Y_{t}) \circ\mathrm{d} \mathbf{X}_{t}, \quad Y_{0} =
y_{0},
\end{align*}
where $\mathbf{X} \in\mathcal{C}^{0, p-var} \left(  [0, T]; G^{\lfloor p
\rfloor} \left( \mathbb{R}^{d} \right) \right) $ satisfies the same conditions
as in Proposition \ref{2ndlevel}. Then if $V \in\mathcal{C}^{\lfloor p
\rfloor+ 4}_{b} \left(  \mathbb{R}^{d}; \mathbb{R}^{d} \otimes\mathbb{R}^{d}
\right) $, we have
\begin{align}
\label{enhancedSkoEq}\lim_{\left\|  \pi(n)\right\|  \rightarrow0} \left\|
\sum_{i} V(Y_{t^{n}_{i}}) \left(  \mathbf{X}^{2}_{t^{n}_{i}, t^{n}_{i+1}} -
\frac{1}{2} \sigma^{2} \left(  t^{n}_{i}, t^{n}_{i+1} \right)  \mathcal{I}_{d}
\right) \right\| _{L^{2}(\Omega)} = 0.
\end{align}

\end{corollary}

\begin{proof}
Since $V \in \mathcal{C}^1_b$, $\abs{\exptn{V^{(a, b)}_s V^{(a, b)}_t}}$ is bounded for all $s, t, a$ and $b$. Now we have to show that bound \eqref{propCond} in Proposition \ref{2ndlevel} is satisfied with
\begin{align*}
\psi_t = V(Y_t) \in \mathbb{R}^d \otimes \mathbb{R}^d,
\end{align*}
to show \eqref{enhancedSkoEq}.
To do so, recall Proposition \ref{dir der est}, which states that almost surely we have
\begin{align}
\norm{\mathcal{D}_{h_1, \ldots, h_n}^n Y_{\cdot}}_{\infty}
\leq P_{d(n)} \left( \norm{\mathbf{X}}_{p-var; [0, T]}, \exp \left( C \, N^{\mathbf{X}}_{1; [0, T]} \right)\right) \prod\limits_{i=1}^{n} \left\| \Phi(h_{i})\right\|_{q-var; [0, T]}.
\end{align}
As both $\norm{\mathbf{X}}_{p-var; [0, T]}$ and $\exp \left( C \, N^{\mathbf{X}}_{1; [0, T]} \right)$ belong to $\bigcap_{r>0}L^{r}\left( \Omega \right)$, we have
\begin{align} \label{bound}
\norm{\mathcal{D}_{h_1, \ldots, h_n}^n Y_t}_{L^{r}\left( \Omega \right) }
\leq C_{n, q} \prod \limits_{i=1}^{n} \norm{\Phi(h_i)}_{q-var; \left[ 0,T\right] }
\end{align}
for any $r>0$.
Now we simply use the product and chain rule of Malliavin differentiation in conjunction with the fact that $V$ has bounded derivatives up to the appropriate order.
\end{proof}

\section{Conversion formula}

We are now ready to prove the main result of the paper. As before, $\pi(n) :=
\left\{  t^{n}_{i} \right\} , t^{n}_{i} := \frac{iT}{2^{n}}$, denotes the
sequence of dyadic partitions on $[0, T]$.

\begin{theorem}
\label{mainThm2ndLevel} For $1 \leq p < 3$, let $Y \in\mathcal{C}^{p-var}
\left(  [0, T]; \mathbb{R}^{d} \right) $ denote the path-level solution to
\begin{align*}
\mathrm{d} Y_{t} = V(Y_{t}) \circ\mathrm{d} \mathbf{X}_{t}, \quad Y_{0} =
y_{0},
\end{align*}
where $\mathbf{X} \in\mathcal{C}^{0, p-var} \left(  [0, T]; G^{\lfloor p
\rfloor} \left( \mathbb{R}^{d} \right) \right) $ denotes the geometric rough
path constructed from the limit of the piecewise-linear approximations of $X$,
a continuous, centered Gaussian process in $\mathbb{R}^{d}$ with i.i.d.
components and continuous covariance function satisfying $\left\|  R\right\|
_{\rho-var;[0, T]^{2}} < \infty$ for some $\rho\in\left[  1, \frac{3}%
{2}\right) $. In addition, we have the following assumptions:

\begin{enumerate}
[(i)]

\item If $1 \leq p < 2$, assume $V \in\mathcal{C}^{2}_{b} \left(
\mathbb{R}^{d}; \mathbb{R}^{d} \otimes\mathbb{R}^{d} \right) $, $\sigma^{2}(s,
t) \leq C \left|  t -s\right| ^{\theta}$ for some $\theta> 1$ and $\left\|
R(\cdot)\right\| _{q-var; [0, T]} < \infty$, where $\frac{1}{p} + \frac{1}{q}
> 1$.

\item If $2 \leq p < 3$, assume that $V \in\mathcal{C}^{6}_{b} \left(
\mathbb{R}^{d}; \mathbb{R}^{d} \otimes\mathbb{R}^{d} \right) $, and the
covariance function satisfies
\begin{align}
\left\|  R(t, \cdot) - R (s, \cdot) \right\| _{\rho-var; [0, T]} \leq C
\left|  t - s\right| ^{\frac{1}{\rho}},
\end{align}
for all $s, t \in[0, T]$.
\end{enumerate}

In either case, almost surely we have
\begin{align*}
\int_{0}^{T} Y_{t} \circ\mathrm{d} \mathbf{X}_{t}  & = \int_{0}^{T} Y_{t} \,
\mathrm{d} X_{t} + \frac{1}{2} \int_{0}^{T} \mathrm{tr} \left[  V(Y_{s})
\right]  \, \mathrm{d} R(s)\\
& \qquad+ \int_{[0, T]^{2}} \mathds{1}_{[0, t)} (s) \mathrm{tr} \left[
J^{\mathbf{X}}_{t \leftarrow s} V(Y_{s}) - V(Y_{t}) \right]  \, \mathrm{d}
R(s,t).
\end{align*}

\end{theorem}

\begin{proof}
Using regular Riemann-Stieltjes integration when $1 \leq p <2$ and Theorem \ref{controlledThm1} when $2 \leq p < 3$, $\int_0^T Y_t \strato{\textbf{X}_t} $ is equal almost surely to
\begin{align*}
\lim_{n \rightarrow \infty} \sum_i
\begin{cases}
\, Y_{t^n_i} \left(X_{t^n_i, t^n_{i+1}} \right), \quad &1 \leq p < 2, \\
\, Y_{t^n_i} \left(X_{t^n_i, t^n_{i+1}} \right) + V(Y_{t^n_i}) \left( \mathbf{X}^2_{t^n_i, t^n_{i+1}} \right), \quad &2 \leq p < 3.
\end{cases}
\end{align*}
We now apply Proposition \ref{skorohodLimit2ndLevel} in conjunction with Corollary \ref{skorodhodLimitEnhanced}. Upon extracting a subsequence (and reusing the index for notational simplicity), the Skorohod integral is given almost surely by
\begin{align*}
\int_0^T Y_t \wrt{X_t}
&= \lim_{n \rightarrow \infty} \sum_i \left[ Y_{t^n_i} \left( X_{t^n_i, t^n_{i+1}} \right) - \int_0^{t^n_i} \mathrm{tr} \left[J^{\mathbf{X}}_{t^n_i \leftarrow s} V(Y_s)\right] R\left( \Delta^n_i, \mathrm{d} s \right) + A^{(i)} \right],
\end{align*}
where
\begin{align*}
A^{(i)} := \begin{cases}
\, V(Y_{t^n_i}) \left( -\frac{1}{2} \sigma^2 \left( t^n_i, t^n_{i+1} \right) \mathcal{I}_d \right), \quad &1 \leq p < 2, \\
\, V(Y_{t^n_i}) \left( \left( \mathbf{X}^2_{t^n_i, t^n_{i+1}} \right) - \frac{1}{2} \sigma^2 \left( t^n_i, t^n_{i+1} \right) \mathcal{I}_d \right), \quad &2 \leq p < 3.
\end{cases}
\end{align*}
Note that when $1 \leq p < 2$, we can append $\sum_i V(Y_{t^n_i}) \left( -\frac{1}{2} \sigma^2 \left( t^n_i, t^n_{i+1} \right) \mathcal{I}_d \right)$ to the Riemann sum approximants of the Skorohod integral  because
\begin{align} \label{additionalZeroTerm}
\sum_{i} \mathrm{tr} \left[V \left( Y_{t^n_i} \right) \right] \sigma^2 \left( t^n_i, t^n_{i+1} \right)
\leq C \sum_{i} |t^n_{i+1} - t^n_i|^{\theta},
\end{align}
which vanishes as $n \rightarrow \infty$. \par
In both cases, subtracting the two integrals gives us
\begin{align} \label{diff1}
\begin{split}
&\int_0^T Y_t \strato{\textbf{X}_t} - \int_0^T Y_t \wrt{X_t} \\
&\qquad= \lim_{n \rightarrow \infty} \sum_i \int_0^{t^n_i} \mathrm{tr} \left[ J^{\mathbf{X}}_{t^n_i \leftarrow s} V(Y_s) \right] R(\Delta^n_i, \wrt{s} ) + \frac{1}{2} \sigma^2 \left( t^n_i, t^n_{i+1} \right) \mathrm{tr} \left[ V(Y_{t^n_i}) \right].
\end{split}
\end{align}
Subtracting $\mathrm{tr} \left[V(Y_{t^n_i})\right] R(\Delta^n_i, t^n_i)$ from the first term on the right of \eqref{diff1} and adding it to the second term gives us
\begin{align*}
\int_0^{t^n_i} \mathrm{tr} \left[ J^{\mathbf{X}}_{t^n_i \leftarrow s} V(Y_s) \right] R(\Delta^n_i, \wrt{s} ) + \frac{1}{2} \sigma^2 \left( t^n_i, t^n_{i+1} \right) \mathrm{tr} \left[ V(Y_{t^n_i})\right]  = F^{(i)} + G^{(i)},
\end{align*}
where
\begin{align*}
&F^{(i)} := \int_0^{t^n_i} \mathrm{tr} \left[  J^{\mathbf{X}}_{t^n_i \leftarrow s} V(Y_s) \right] R(\Delta^n_i, \mathrm{d}s) - \mathrm{tr} \left[V(Y_{t^n_i})\right] R(\Delta^n_i, t^n_i), \\
&G^{(i)} := \frac{1}{2} \sigma^2 \left( t^n_i, t^n_{i+1} \right) \mathrm{tr} \left[V(Y_{t^n_i})\right] + \mathrm{tr} \left[V(Y_{t^n_i})\right] R(\Delta^n_i, t^n_i).
\end{align*}
We have
\begin{align*}
F^{(i)}
&= \int_0^{t^n_i} \mathrm{tr} \left[  J^{\mathbf{X}}_{t^n_i \leftarrow s} V(Y_s) - V(Y_{t^n_i})\right] \, R\left( \Delta^n_i, \mathrm{d}s \right) \\
&= \int_0^T h (s, t^n_i) R \left(\Delta^n_i, \mathrm{d}s \right),
\end{align*}
where we denote
\begin{align*}
h(s, t) &:= \mathds{1}_{[0, t)} (s) \, \mathrm{tr} \left[  J^{\mathbf{X}}_{t \leftarrow s} V(Y_s) - V(Y_t) \right].
\end{align*}
Since $h(s, t)$ vanishes on the diagonal, it is continuous almost surely on $[0, T]^2$. Furthermore, we have complementary regularity since $\frac{1}{p} + \frac{1}{\rho} > 1$, in which case Theorem \ref{2Dintegral} tells us that
\begin{align*}
\int_{[0, T]^2} h(s, t) \wrt{R(s, t)},
\end{align*}
exists. Thus, we have some partition $\pi' = \{ s_k \} \times \{ t^n_i \}$ such that
\begin{align*}
\abs{\int_{[0, T]^2} h(s, t) \wrt{R(s, t)} - \sum_{i, k} h(s_k, t^n_i) R \begin{pmatrix}
s_k & s_{k+1} \\
t^n_i & t^n_{i+1}
\end{pmatrix}} < \frac{\varepsilon}{2}.
\end{align*}
Refining $\{ s_k \}$ if necessary, we also have for each $i$
\begin{align*}
\abs{\int_0^T h(s, t^n_i) R\left(\Delta^n_i, \mathrm{d}s\right) - \sum_k h(s_k, t^n_i) R \begin{pmatrix}
s_k & s_{k+1} \\
t^n_i & t^n_{i+1}
\end{pmatrix}} < \frac{\varepsilon}{2} \left( \frac{1}{2^n} \right),
\end{align*}
and note that these estimates hold for all $\pi = \pi_1 \times \pi_2$ where $\norm{\pi} \leq \norm{\pi'}$ and $\norm{\pi_2} \leq \norm{\pi(n)}$. Thus
\begin{align*}
\sum_{i} F^{(i)} \rightarrow \int_{[0, T]^2} h (s,t) \wrt{R(s, t)}.
\end{align*}
For the $G$ terms we have
\begin{align*}
\sum_i G^{(i)}
&= \sum_i \mathrm{tr} \left[V(Y_{t^n_i})\right] \left( R(t^n_{i+1}, t^n_i) -  R(t^n_i, t^n_i) + \frac{1}{2} \sigma^2 \left( t^n_i, t^n_{i+1} \right) \right) \\
&= \frac{1}{2} \sum_i \mathrm{tr} \left[V(Y_{t^n_i})\right] \big( R\left( t^n_{i+1}, t^n_{i+1} \right) - R \left(t^n_i, t^n_i \right) \big),
\end{align*}
which converges to $\frac{1}{2} \int_0^T \mathrm{tr} \left[V(Y_t)\right] \wrt{R(t)}$ as $Y$ and $R(\cdot)$ have complementary regularity.
\end{proof}

The limit in \eqref{diff1} necessarily exists almost surely because it is the
difference of almost sure convergent sequences. However, we add and subtract
$\mathrm{tr} \left[ V(Y_{t^{n}_{i}})\right]  R(\Delta^{n}_{i}, t^{n}_{i})$
because in general, if considered separately, neither term can be expected to
be a convergent sequence.

Consider the case when $R(s, t)$ is the covariance function of fractional
Brownian motion where $\frac{1}{3} < H < \frac{1}{2}$. For the first term of
\eqref{diff1}, formally one would expect convergence to the Young integral
\begin{align*}
\int_{[0, T]^{2}} \mathds{1}_{[0, t)}(s) \, \mathrm{tr} \left[  J^{\mathbf{X}%
}_{t \leftarrow s} V(Y_{s}) \right]  \, \mathrm{d} R(s, t)
\end{align*}
since we have complementary regularity. However the discontinuity of the
integrand at the diagonal poses a problem, as can be illustrated by the
following simple example; if we take the sequence of square partitions
$\left\{  \left(  t^{n}_{i}, t^{n}_{j} \right)  \right\} $, the
Riemann-Stieltjes sums of $\int_{[0, T]^{2}} \mathds{1}_{[0, t)}(s) \,
\mathrm{d} R(s, t)$ are given by
\begin{align*}
\sum_{j} \sum_{i < j} R(\Delta_{i}, \Delta_{j}) = \sum_{j} R \left(  t^{n}%
_{j}, t^{n}_{j+1} \right)  - R\left(  t^{n}_{j}, t^{n}_{j} \right)
\rightarrow-\infty
\end{align*}
and thus $\int_{[0, T]^{2}} \mathds{1}_{[0, t)}(s) \, \mathrm{d} R(s, t)$ does
not exist as a Young-Stieltjes integral. For the second term of \eqref{diff1},
if $V$ is bounded from below, we have
\begin{align*}
\sum_{i} \frac{1}{2} \sigma^{2} \left(  t^{n}_{i}, t^{n}_{i+1} \right)
\mathrm{tr} \left[ V \left(  Y_{t^{n}_{i}} \right)  \right]  \geq C \sum_{i}
\left|  t^{n}_{i+1} - t^{n}_{i}\right| ^{2H},
\end{align*}
which also diverges.

Now consider the following theorem from \cite{d2000}.

\begin{theorem}
\label{orig} Let $X$ be fractional Brownian motion with Hurst parameter $H<
\frac{1}{2}$. \newline If $u\in\mathbb{D}^{1,2}\left( I_{0^{+}}^{\frac{3}%
{2}-H}(L^{2})\right) $, then $\delta^{X}(\mathcal{K}^{*} u)$ and
$\mathrm{trace}(\mathcal{D} u)$ are well defined, and the sequence
\begin{align*}
\sum_{i}\frac{1}{t_{i+1}-t_{i}}\int_{t_{i}}^{t_{i+1}}u_{t} \, \mathrm{d} t
\left(  X_{t_{i+1}}-X_{t_{i}} \right)
\end{align*}
converges in $L^{2}(\Omega)$ to $\delta^{X}(\mathcal{K}^{*}u) + \mathrm{trace}%
(\mathcal{D}u)$.
\end{theorem}

Formally, $\mathcal{K}^{*}$ is the operator $K^{*} \circ D^{1}_{T^{-}}$, where
$D^{1}_{T^{-}}$ is the adjoint of the derivative operator; see
\cite{dfond2005}. It is well-known that the Besov-Liouville space $I_{0^{+}%
}^{\frac{3}{2}-H}(L^{2})$ can be embedded continuously in $\mathcal{C}^{0, 1-
H}$ (see \cite{nualart2006}, \cite{skm1993}, \cite{dfond2005}), the space of
$(1-H)$ H\"{o}lder continuous paths starting at zero. This imposes a strong
condition on the integrand as one essentially requires Young-complementary
regularity of $u$ and $X$.

Thus, when the integrand solves an RDE, Theorem \ref{mainThm2ndLevel} extends
this theorem to cases where the integrand and integrator do not have
complementary regularity. Furthermore, when $1 \leq p < 2$, although
\newline$\int_{0}^{t^{n}_{i}} \mathrm{tr} \left[ J^{\mathbf{X}}_{t^{n}_{i}
\leftarrow s} V(Y_{s})\right]  R\left(  \Delta^{n}_{i}, \mathrm{d} s \right) $
in general converges, by augmenting the Skorohod integral with $A_{i}$ and
re-balancing the terms, we can identify the trace term in Theorem \ref{orig}
more precisely.

\subsection{Application of the correction formula to fractional Brownian
motion}

We now apply the correction formula to fractional Brownian motion with $H >
\frac{1}{3}$.

\begin{theorem}
For $1 \leq p < 3$, let $Y \in\mathcal{C}^{p-var} \left(  [0, T];
\mathbb{R}^{d} \right) $ denote the path-level solution to
\begin{align*}
\mathrm{d} Y_{t} = V(Y_{t}) \circ\mathrm{d} \mathbf{X}_{t}, \quad Y_{0} =
y_{0}.
\end{align*}
We assume that $V \in\mathcal{C}^{k}_{b} \left(  \mathbb{R}^{d};
\mathbb{R}^{d} \otimes\mathbb{R}^{d} \right) $, with
\begin{align}
\label{kDefn}k =
\begin{cases} \, 2, \quad 1 \leq p < 2, \\ \, 6, \quad 2 \leq p < 3, \end{cases}
\end{align}
and $\mathbf{X} \in\mathcal{C}^{0, p-var} \left(  [0, T]; G^{\lfloor p
\rfloor} \left( \mathbb{R}^{d} \right) \right) $ is the geometric rough path
constructed from the limit of the piecewise-linear approximations of
fractional Brownian motion with Hurst parameter $H > \frac{1}{3}$, and
covariance function
\begin{align*}
R(s, t) = \frac{1}{2} \left( s^{2H} + t^{2H} - \left|  t - s\right| ^{2H}
\right) .
\end{align*}
Then almost surely, we have
\begin{align*}
\int_{0}^{T} Y_{t} \circ\mathrm{d} \mathbf{X}_{t}  & = \int_{0}^{T} Y_{t} \,
\mathrm{d} X_{t} + H \int_{0}^{T} \mathrm{tr} \left[  V(Y_{s}) \right]  \,
s^{2H - 1} \, \mathrm{d} s\\
& \qquad+ \int_{[0, T]^{2}} \mathds{1}_{[0, t)} (s) \mathrm{tr} \left[
J^{\mathbf{X}}_{t \leftarrow s} V(Y_{s}) - V(Y_{t}) \right]  \, \mathrm{d}
R(s, t).
\end{align*}

\end{theorem}

\begin{proof}
We will show that fractional Brownian motion fulfills all the requirements needed to apply Theorem \ref{mainThm2ndLevel} when $H > \frac{1}{3}$.
Let $\rho := \frac{1}{2H}$ and $q := \frac{1}{\rho} \vee 1$. The proof that $\norm{R}_{\rho-var; [0, T]^2} < \infty$ can be found in \cite{fv2010a}; see also \cite{fv2011}. Note also that $R(t) = t^{2H}$ is of bounded variation, and thus has finite q-variation. \par
In the case $1 \leq p < 2$, or $H > \frac{1}{2}$, the geometric rough path is simply $\left( 1, B^H_t \right)$, and for $H \leq \frac{1}{2}$, one can invoke Theorem \ref{gaussianRP} to construct the geometric rough path. \par
Finally, it is proved in Example 1 of \cite{fv2011} that
\begin{align*}
\norm{R(t, \cdot) - R (s, \cdot) }_{\rho-var; [0, T]}  \leq C \left| t - s\right|^{\frac{1}{\rho}}, \quad \forall \, s, t \in [0, T].
\end{align*}
\end{proof}

\begin{appendices}
\section{}
\contThm*
\begin{proof}
Let $0 \leq u < s < v \leq t$ and define
\begin{align*}
\Xi_{u, v} := \phi_u x_{u, v} + \phi'_u \mathbf{x}^2_{u, v},
\end{align*}
which yields the defect of additivity,
\begin{align*}
\abs{\Xi_{u, s} + \Xi_{s, v} - \Xi_{u, v}}
\leq \abs{R^{\phi}_{u, s} x_{s, v}} + \abs{\phi'_{u, s} \mathbf{x}^2_{s, v}}.
\end{align*}
Now let $\theta := \frac{3}{p}$. Then the following function
\begin{align*}
\omega(u, v) := \norm{R^{\phi}}^{\frac{1}{\theta}}_{\frac{p}{2}-var; [u, v]} \norm{x}^{\frac{1}{\theta}}_{p-var; [u, v]} + \norm{\phi'}^{\frac{1}{\theta}}_{p-var; [u, v]} \norm{\mathbf{x}^2}^{\frac{1}{\theta}}_{\frac{p}{2}-var; [u, v]}
\end{align*}
is a control by Lemma \ref{controlLem1} as $\frac{3}{p} \geq 1$. Moreover, following the proof for Young integration (see \cite{fh2014}), for any partition $\pi = \{r_i\}$ of $[u, v]$ with $k$ sub-intervals, there necessarily exists some $r_j \in \pi$ such that
\begin{align*}
\abs{\Xi_{r_{j-1}, r_j} + \Xi_{r_j, r_{j+1}} - \Xi_{r_{j-1}, r_{j+1}}}
&\leq \abs{R^{\phi}_{r_{j-1}, r_j} x_{r_j, r_{j+1}}} + \abs{\phi'_{r_{j-1}, r_j} \mathbf{x}^2_{r_j, r_{j+1}}} \\
&\leq 2 \, \omega(r_{j-1}, r_{j+1})^{\theta}
\leq 2 \left( \frac{2}{k-1} \right)^{\theta} \omega(u, v)^{\theta}.
\end{align*}
Extracting $r_j$ leaves one with $k - 1$ sub-intervals, and we can repeat this procedure until only $[u, v]$ remains. Since $\theta > 1$, we obtain the sub-maximal inequality (cf. \cite{lyons98}, \cite{fh2014})
\begin{align} \label{controlledOldBound1}
\abs{\int_{\pi} \phi_r \strato{\mathbf{x}_r} - \left( \phi_u x_{u, v} + \phi'_u \mathbf{x}^2_{u, v} \right)}
\leq C \, \zeta(\theta) \, \omega(u, v)^{\theta},
\end{align}
where $\zeta$ is the Riemann zeta function and
\begin{align*}
\int_{\pi} \phi_r \strato{\mathbf{x}_r}
:= \sum_i \phi_{r_i} x_{r_i, r_{i+1}} + \phi'_{r_i} \mathbf{x}^2_{r_i, r_{i+1}}.
\end{align*}
Proving \eqref{controlledRPdefn1} is equivalent to showing that
\begin{align*}
\sup_{\norm{\pi} \vee \norm{\pi'} < \varepsilon} \abs{\int_{\pi} \phi_r \strato{\mathbf{x}_r} - \int_{\pi'} \phi_r \strato{\mathbf{x}_r}} \rightarrow 0 \quad \mathrm{as} \, \varepsilon \rightarrow 0,
\end{align*}
where the supremum is taken over all partitions of $[0, t]$. Without loss of generality, we can assume $\pi'$ refines $\pi$, in which case $\norm{\pi} \vee \norm{\pi'} = \norm{\pi}$ and
\begin{align*}
\abs{\int_{\pi} \phi_r \strato{\mathbf{x}_r} - \int_{\pi'} \phi_r \strato{\mathbf{x}_r}}
&= \abs{\sum_{[u, v] \in \pi} \left( \phi_u x_{u, v} + \phi'_u \mathbf{x}^2_{u, v} - \int_{\pi' \cap [u, v]} \phi_r \strato{\mathbf{x}_r} \right)} \\
&\leq C_p \sum_{[u, v] \in \pi} \omega(u, v)^{\theta},
\end{align*}
which vanishes as $\norm{\pi} \rightarrow 0$. \par
Continuing, we define
\begin{align} \label{remainder1}
R^z_{s,t} := \int_s^t \phi_r \strato{x_r} - \phi_s x_{s,t},
\end{align}
and using \eqref{controlledOldBound1}, we obtain
\begin{align*}
\abs{z_{s,t}}^p, \abs{z'_{s,t}}^p
&\leq C_p \left( \norm{\phi}_{\mathcal{V}^p; [0, T]} + \norm{\phi'}_{\mathcal{V}^p; [0, T]} + \norm{R^{\phi}}_{\frac{p}{2}-var; [0, T]} \right)^p \\
&\fqquad\times \left( \norm{x}^p_{p-var; [s,t]} + \norm{\mathbf{x}^2}^p_{\frac{p}{2}-var; [s,t]} \right), \\
\abs{R^z_{s,t}}^{\frac{p}{2}}
&\leq C_p \left( \norm{\phi'}^{\frac{p}{2}}_{\mathcal{V}^p; [0, T]} \norm{\mathbf{x}^2}^{\frac{p}{2}}_{\frac{p}{2}-var; [s,t]} +  \norm{x}^{\frac{p}{2}}_{p-var;[0, T]} \norm{R^{\phi}}^{\frac{p}{2}}_{\frac{p}{2}-var; [s,t]} \right).
\end{align*}
From the super-additivity of the quantities on the right side in the above expression, the fact that $(z, z')$ is controlled with norm \eqref{controlledBound1} follows immediately.
\end{proof}
\contSM*
\begin{proof}
See Lemma 7.3 in \cite{fh2014} for the proof in H\"{o}lder topology; the $p$-variation estimates will be derived similarly. Using the mean-value theorem, \eqref{csmEstimateA} can be obtained easily. To show \eqref{csmEstimateB} and that $(\phi(y), \nabla(y) y')$ is controlled by $\mathbf{x}$, we first use Taylor's theorem to obtain
\begin{align} \label{z1}
(\phi(y))_{s,t}=\nabla \phi \left( y_{s}\right) y_{s,t} + R_{s,t}^{Taylor}
\end{align}
for all $s < t$ in $\left[ 0,T\right]$, where $\abs{R_{s,t}^{Taylor}} \leq \norm{\phi}_{\mathcal{C}_b^2} \abs{y_{s,t}}^2$. From this it follows that
\begin{align} \label{Tremainder1}
\norm{R^{Taylor}}_{\frac{p}{2}-var; \left[ 0,T\right] }
\leq \norm{\phi}_{\mathcal{C}_b^2} \norm{y}_{p-var; \left[ s,t\right] }^2.
\end{align}
We next use the fact that $(y, y')$ is controlled by $\mathbf{x}$ in equation \eqref{z1}, which yields
\begin{align*}
(\phi(y))_{s,t}
&= \underset{\displaystyle =: (\phi(y))'_s}{\underbrace{\nabla \phi \left( y_{s}\right) y'_s }} \, x_{s,t} + \underset{\displaystyle =: R^{\phi(y)}_{s,t}}{\underbrace{\nabla \phi \left( y_{s}\right) R_{s,t}^{y} + R_{s,t}^{Taylor}}},
\end{align*}
and also gives
\begin{align*}
\norm{R^{\phi(y)}}_{\frac{p}{2}-var; [0, T]}
\leq \norm{\nabla \phi (y)}_{\infty} \norm{R^y}_{\frac{p}{2}-var; [0, T]} + \norm{R^{Taylor}}_{\frac{p}{2}-var; [0, T]}.
\end{align*}
\end{proof}
\contLeibniz*
\begin{proof}
The statement can be seen as a corollary to the previous proposition if we consider the smooth map $\Phi (\phi, \psi) = \phi \psi$. However, we will prove it directly to get the precise bounds \eqref{leibnizBound1a} and \eqref{leibnizBound1b}. \par
For the first part, it is trivial to see that $\norm{\phi\psi}$ and $\norm{\phi'\psi + \phi\psi'}_{p-var;[0, T]}$ are bounded by the right side of \eqref{leibnizBound1a}. For the remainder term, we note that
\begin{align*}
\left( \phi \psi \right)_{s,t} - \left( \phi'_s \psi_s + \phi_s \psi'_s \right) x_{s,t}
&= \phi_t \psi_t - \phi_s \psi_s - \phi_{s,t} \psi_s - \phi_s \psi_{s,t} + R^{\phi}_{s,t} \psi_s + \phi_s R^{\psi}_{s,t} \\
&= \phi_t \psi_t - \phi_t \psi_s - \phi_s \psi_t + \phi_s \psi_s + R^{\phi}_{s,t} \psi_s + \phi_s R^{\psi}_{s,t} \\
&= \phi_{s,t} \psi_{s,t} + R^{\phi}_{s,t} \psi_s + \phi_s R^{\psi}_{s,t},
\end{align*}
and thus
\begin{align*}
\norm{R^{\phi\psi}}_{\frac{p}{2}-var; [0,T]}
&\leq \norm{\phi}_{p-var; [0, T]} \norm{\psi}_{p-var; [0, T]} + \norm{\psi}_{\infty} \norm{R^{\phi}}_{\frac{p}{2}-var; [0, T]} \\
&\fqquad+ \norm{\phi}_{\infty} \norm{R^{\psi}}_{\frac{p}{2}-var; [0, T]}.
\end{align*}
For the second part, note that $\norm{\phi\psi}$ and $\norm{\phi' \psi}_{p-var; [0, T]}$ are bounded by the right side of \eqref{leibnizBound1b}. Moreover, we have
\begin{align*}
\left( \phi \psi \right)_{s,t} - \left( \phi'_s \psi_s \right) x_{s,t}
&= \phi_t \psi_t - \phi_s\psi_s - \phi_{s,t} \psi_s + R^{\phi}_{s,t} \psi_s \\
&= \phi_t \psi_{s,t} + R^{\phi}_{s,t} \psi_s =: R^{\phi\psi}_{s,t},
\end{align*}
which gives
\begin{align*}
\norm{ R^{\phi\psi}}_{\frac{p}{2}-var; [0, T]}
\leq \norm{\phi}_{\infty} \norm{\psi}_{\frac{p}{2}-var; [0, T]} + \norm{\psi}_{\infty} \norm{R^{\phi}}_{\frac{p}{2}-var; [0, T]}.
\end{align*}
\end{proof}
\end{appendices}

\section*{Acknowledgment}

The authors wish to thank Samy Tindel for helpful discussion during the course
of this work. Part of this work was undertaken on a trip of the first-named
author to the National University of Singapore (NUS).


\end{document}